\documentclass[a4paper]{amsart}
\usepackage[dvipsnames]{xcolor}
\usepackage{comment}
\usepackage{amsmath,amsthm,amssymb,relsize,mathrsfs}
\usepackage[nobysame,alphabetic]{amsrefs}
\usepackage[margin=30mm, marginpar=25mm]{geometry}

\usepackage{booktabs, multirow, adjustbox, array}

\usepackage{pifont}% http://ctan.org/pkg/pifont

\usepackage{enumitem}
\usepackage{tikz}
\usetikzlibrary{
  cd,
  calc,
  positioning,
  fit,
  arrows,
  decorations.pathreplacing,
  decorations.markings,
  shapes.geometric,
  backgrounds,
  bending
}
\usepackage{tikzsymbols}
\usepackage[arrow,matrix,curve,cmtip,ps]{xy}
\usepackage{pb-diagram}
\usepackage{pb-xy}

\definecolor{darkblue}{rgb}{0,0,0.6}

\usepackage[breaklinks, pdftex, ocgcolorlinks,colorlinks=true, citecolor=darkblue, filecolor=darkblue, linkcolor=darkblue, urlcolor=darkblue]{hyperref}

\usepackage[capitalize,noabbrev]{cleveref}

\makeatletter
\newtheorem*{rep@theorem}{\rep@title}
\newcommand{\newreptheorem}[2]{%
\newenvironment{rep#1}[1]{%
 \def\rep@title{#2 \ref{##1}}%
 \begin{rep@theorem}}%
 {\end{rep@theorem}}}
\makeatother

\newtheorem{proposition}{Proposition}[section]
\newtheorem{theorem}[proposition]{Theorem}

\newtheorem{lemma}[proposition]{Lemma}

\newtheorem{thmx}{Theorem}

\newtheorem{corx}[thmx]{Corollary}

\theoremstyle{definition}
\newtheorem{definition}[proposition]{Definition}
\newtheorem{question}[proposition]{Question}
\newtheorem{example}[proposition]{Example}

\theoremstyle{remark}
\newtheorem{remark}[proposition]{Remark}
\newtheorem{claim}[proposition]{Claim}
\newtheorem*{remark*}{Remark}

\newtheorem{construction}[proposition]{Construction}

\newreptheorem{theorem}{Theorem}
\newreptheorem{lemma}{Lemma}
\newreptheorem{proposition}{Proposition}
\newreptheorem{corollary}{Corollary}
\newreptheorem{question}{Question}

\AddToHook{env/proposition/begin}{\crefalias{proposition}{proposition}}
\AddToHook{env/theorem/begin}{\crefalias{proposition}{theorem}}
\AddToHook{env/thmx/begin}{\crefalias{proposition}{thmx}}
\AddToHook{env/corx/begin}{\crefalias{proposition}{corx}}
\AddToHook{env/assumption/begin}{\crefalias{proposition}{assumption}}
\AddToHook{env/hypothesis/begin}{\crefalias{proposition}{hypothesis}}
\AddToHook{env/corollary/begin}{\crefalias{proposition}{corollary}}
\AddToHook{env/definition/begin}{\crefalias{proposition}{definition}}
\AddToHook{env/lemma/begin}{\crefalias{proposition}{lemma}}
\AddToHook{env/question/begin}{\crefalias{proposition}{question}}
\AddToHook{env/example/begin}{\crefalias{proposition}{example}}
\AddToHook{env/conjecture/begin}{\crefalias{proposition}{conjecture}}
\AddToHook{env/remark/begin}{\crefalias{proposition}{remark}}
\AddToHook{env/const/begin}{\crefalias{proposition}{const}}

\crefname{assumption}{Assumption}{Assumptions}
\crefname{hypothesis}{Hypothesis}{Hypotheses}
\crefname{theorem}{Theorem}{Theorems}
\crefname{thmx}{Theorem}{Theorems}
\crefname{proposition}{Proposition}{Propositions}
\crefname{corollary}{Corollary}{Corollaries}
\crefname{definition}{Definition}{Definitions}
\crefname{lemma}{Lemma}{Lemmas}
\crefname{question}{Question}{Questions}
\crefname{example}{Example}{Examples}
\crefname{conjecture}{Conjecture}{Conjectures}
\crefname{remark}{Remark}{Remarks}
\crefname{const}{Construction}{Constructions}

\renewcommand\mp[1]{\marginpar{\tiny mp: #1}}

\newcommand{\R}{\mathbb{R}}
\newcommand{\Z}{\mathbb{Z}}

\newcommand{\im}{\operatorname{Im}}
\newcommand{\Id}{\operatorname{Id}}
\newcommand{\id}{\operatorname{Id}}

\newcommand{\wt}{\widetilde}
\newcommand{\wh}{\widehat}
\newcommand{\sm}{\setminus}

\newcommand{\ks}{\mathrm{ks}}
\newcommand{\Sm}{\mathsf{Sm}}
\newcommand{\KS}{\mathsf{KS}}
\newcommand{\CP}{\mathbb{C}P}

\DeclareMathOperator{\Sq}{Sq}

\DeclareMathOperator{\Aut}{Aut}

\DeclareMathOperator{\Wh}{Wh}

\DeclareMathOperator{\pt}{pt}
\DeclareMathOperator{\pr}{pr}
\DeclareMathOperator{\ev}{ev}
\DeclareMathOperator{\colim}{colim}

\newcommand{\CAT}{\mathrm{CAT}}
\newcommand{\Top}{\mathrm{Top}}
\newcommand{\Diff}{\mathrm{Diff}}
\newcommand{\PL}{\mathrm{PL}}
\newcommand{\OO}{\mathrm{O}}
\newcommand{\G}{\mathrm{G}}
\DeclareMathOperator{\MSTop}{MSTop}
\newcommand{\BTop}{\mathrm{BTop}}
\newcommand{\BTOP}{\mathrm{BTop}}
\newcommand{\BSTop}{\mathrm{BSTop}}
\newcommand{\BG}{\mathrm{BG}}
\newcommand{\BO}{\mathrm{BO}}

\newcommand{\BPL}{\mathrm{BPL}}
\newcommand{\BB}{\mathrm{B}}
\newcommand{\Ad}{\mathrm{Ad}}
\newcommand{\SW}{\mathrm{SW}}

\DeclareMathOperator{\Th}{Th}

\DeclareMathOperator{\Homeo}{Homeo}
\newcommand{\Diffeo}{\operatorname{Diffeo}}
\newcommand{\pseudoHomeo}{\wt{\operatorname{Homeo}}\mkern 0mu}
\newcommand{\pseudoDiffeo}{\wt{\operatorname{Diffeo}}\mkern 0mu}

\newcommand{\hAut}{\wt{\operatorname{hAut}}\mkern 0mu}
\newcommand{\sAut}{\wt{\operatorname{sAut}}\mkern 0mu}

\newcommand{\Embt}{\operatorname{Emb}^{\Top}}

\newcommand\lra{\longrightarrow}

\newcommand{\bS}{\mathbb{S}}
\newcommand{\mft}{\mathfrak{t}}

\begin{document}
\title{Smoothing topological pseudo-isotopies of 4-manifolds}

\author[P.~Orson]{Patrick Orson}
\address{Mathematics Department, California Polytechnic State University, USA}
\email{porson@calpoly.edu}

 \author[M.~Powell]{Mark Powell}
 \address{School of Mathematics and Statistics, University of Glasgow, United Kingdom}
 \email{mark.powell@glasgow.ac.uk}

 \author[O.~Randal-Williams]{Oscar Randal-Williams}
  \address{Centre for Mathematical Sciences, University of Cambridge, United Kingdom}
 \email{or257@cam.ac.uk}

\def\subjclassname{\textup{2020} Mathematics Subject Classification}
\expandafter\let\csname subjclassname@1991\endcsname=\subjclassname
%\expandafter\let\csname subjclassname@2000\endcsname=\subjclassname

\begin{abstract}
Given a closed, smooth $4$-manifold $X$  and a self-diffeomorphism $f$ that is topologically pseudo-isotopic to the identity, we study the question of whether $f$ is moreover smoothly pseudo-isotopic to the identity. If the fundamental group of $X$ lies in a certain class, which includes trivial, free, and finite groups of odd order, we show the answer is always affirmative. On the other hand, we produce the first examples of manifolds $X$ and diffeomorphisms~$f$ where the answer is negative. Our investigation is motivated by the question, which remains open, of whether there exists a self-diffeomorphism of a closed $4$-manifold that is topologically isotopic to the identity, but not stably smoothly isotopic to the identity.
\end{abstract}
\maketitle

%\tableofcontents

\section{Introduction}

A principle in $4$-manifold theory says that topological results about smooth 4-manifolds often become true smoothly after stabilisation, meaning after taking connected sum with some number of copies of $S^2 \times S^2$.
The most famous instance is the theorem of Wall \cite{MR163323,MR163324} and Freedman~\cite{F} that homeomorphic closed, smooth, simply connected $4$-manifolds are stably diffeomorphic.  Similarly Kreck~\cite{Kreck-isotopy-classes}, Quinn~\cite{Quinn:isotopy} (with a correction by \cite{GGHKP}), and Gabai~\cite{Gabai-22} showed that topologically isotopic diffeomorphisms of such 4-manifolds are stably smoothly isotopic.
Further examples of this principle are not confined to the simply connected case e.g.~Quinn's stable $s$-cobordism theorem~\cite{quinn-1983}, Gompf's extension~\cite{Gompf-stable} of the Wall--Freedman result above to all oriented compact 4-manifolds, and Cappell--Shaneson's stable surgery sequence~\cite{CS-on-4-d-surgery}. In the realm of embedded surfaces, Cha--Kim~\cite{chakim-lightbulb} proved that topologically embedded surfaces are stably smoothable, and Galvin~\cite{GalvinCS} proved that topologically isotopic surfaces in a simply connected $4$-manifold are stably smoothly isotopic.

Lest this long list of results be misconstrued as evidence for a completely general principle, it should be contrasted with exotic phenomena detected by Rochlin's theorem, or equivalently the Kirby--Siebenmann invariant.  For example Kreck~\cite{MR764582}, Cappell--Shaneson~\cite{CS-new-four-mflds-annals}, and Akbulut~\cite{Akbulut-on-fake,Akbulut-fake-4-manifold} each constructed an exotic pair of nonorientable 4-manifolds that fail to be stably diffeomorphic.  Similarly, Cappell--Shaneson and Galvin's non-smoothable homeomorphisms~\cite{CS-on-4-d-surgery,GalvinCS} are not stably smoothable.

In this article, we use the Kirby--Siebenmann invariant to investigate the principle in the context of isotopy of diffeomorphisms.
Given a diffeomorphism $f$ of a $4$-manifold $X$, we can smoothly isotope the diffeomorphism to fix a $4$-ball pointwise and, taking connected sums with $S^2\times S^2$ in this $4$-ball, can then extend the diffeomorphism by the identity on the stabilisation; we refer to any such diffeomorphism of a stabilisation of $X$ as a \emph{stabilisation of $f$}. Two diffeomorphisms are \emph{stably isotopic} if they have isotopic stabilisations.
We are motivated by the following question.

\begin{question}\label{question:motivation}
Given a pair of self-diffeomorphisms of a closed $4$-manifold that are topologically isotopic, are they  moreover stably smoothly isotopic?
\end{question}

This remains open, but our results, concerning the intermediate notion of pseudo-isotopy, indicate that the answer may be negative in general, while providing a positive answer in some cases.

\subsection{Topologically pseudo-isotopic does not imply smoothly pseudo-isotopic}
Let $X$ be a smooth, connected, compact 4-manifold, possibly with nonempty boundary, and let $f_0, f_1 \colon X \to X$ be diffeomorphisms that restrict to the identity on the boundary. A homeomorphism $F \colon X \times I \to X \times I$ restricting to $f_i$ on $X\times\{i\}$ and to the identity on $(\partial X) \times I$ is called a \emph{topological pseudo-isotopy} from $f_0$ to $f_1$. If $F$ is moreover a diffeomorphism, it is called a \emph{smooth pseudo-isotopy} from~$f_0$ to $f_1$. Our main result answers the analogue of \cref{question:motivation} for pseudo-isotopy rather than isotopy, in the negative.

\begin{thmx}\label{thm:main-example}
There exists a closed, smooth $4$-manifold $X$ and a diffeomorphism $f \colon X \to X$ such that $f$ is topologically pseudo-isotopic but not stably smoothly pseudo-isotopic to the identity.
\end{thmx}

The manifolds $X$ in Theorem~\ref{thm:main-example} have the form $X = A \# g(S^2 \times S^2)$, where $A$ is an aspherical~$4$-manifold satisfying various properties and $g$ is sufficiently large. There are many explicit manifolds~$A$ with these properties;~see Example~\ref{ex:hyperbolic}. On the other hand, the diffeomorphisms~$f$ constructed in the course of the proof are inexplicit.

In investigating the analogue of \cref{question:motivation} for pseudo-isotopy we obtained several further results of a more technical flavour. Many of them are used during the proof of Theorem~\ref{thm:main-example}, but they are also of independent interest. In the remainder of the introduction we give an exposition of these results.

\subsection{Smoothing topological pseudo-isotopies}
Firstly, we discuss the complete obstruction to isotoping a topological pseudo-isotopy to a smooth pseudo-isotopy, following Kirby--Siebenmann \cite[Essay V]{Kirby-Siebenmann:1977-1}.
Let $X$ be a compact, connected smooth $4$-manifold.
Given a topological pseudo-isotopy $F\colon  X\times I\to X\times I$, we can give the boundary of $X \times I \times I$ a new smooth structure as follows. Writing $\sigma$ for the smooth structure on $X \times I$, endow $X \times I \times \{1\}$ with the smooth structure $F^*\sigma$, and leave the smooth structure alone on the rest of the boundary: write $\partial(X \times I \times I)_F$ for this smooth structure. Then $(X \times I \times I, \partial(X \times I \times I)_F)$ is a topological~$6$-manifold with a smooth structure on its boundary, so has a relative Kirby--Siebenmann invariant
\[
\KS(F) := \ks(X \times I \times I, \partial(X \times I \times I)_F) \in H^4(X \times I \times I ,\partial(X \times I \times I); \Z/2)\cong_{PD} H_2(X;\Z/2)
\]
which obstructs the extension of the smooth structure $\partial(X \times I \times I)_F$ to the whole of $X \times I \times I$.

\begin{thmx}\label{thm:main-invariant}
  Given a topological pseudo-isotopy $F$ from $f_0$ to $f_1$, the obstruction $\KS(F) \in H_2(X;\Z/2)$ vanishes if and only if $F \colon X \times I \to X \times I$ is topologically isotopic rel.\ $\partial(X \times I)$ to a smooth pseudo-isotopy.
\end{thmx}

A result of Gabai~\cite[Theorem~2.5]{Gabai-22} states that, for a diffeomorphism, the existence of a smooth pseudo-isotopy to the identity with vanishing primary Hatcher--Wagoner invariants, implies smoothly stably isotopic to the identity. Combining this with work of Galvin--Nonino~\cite{Galvin:2025aa} on topological Hatcher--Wagoner invariants,  we deduce the following.

\begin{corx}\label{cor:main-B}
  If $F$ is a topological isotopy from $f_0$ to $f_1$ with $\KS(F) = 0 \in H_2(X;\Z/2)$, then~$f_0$ and $f_1$ are smoothly stably isotopic.
\end{corx}

\subsection{Sometimes topologically pseudo-isotopic implies smoothly pseudo-isotopic}

Secondly, we establish a condition under which diffeomorphisms are smoothly pseudo-isotopic if and only if they are topologically pseudo-isotopic. By \cref{thm:main-invariant} this is equivalent to showing, for a given pair of diffeomorphisms, that there exists a topological pseudo-isotopy $F$ between them with $\KS(F)=0$.
To formulate our result we will make use of a certain map
\[
I_2\colon H_2(\pi;\Z_{(2)} )\lra L_6(\Z[\pi])_{(2)},
\]
arising from the assembly map in algebraic $L$-theory.
 We will describe this map and elucidate some of its properties in Section~\ref{sec:algsurgery} (in particular in Definition~\ref{defn:I2}). The following will be proved using methods from surgery theory.

\begin{thmx}\label{thm:main-guarantee}
Let $X$ be an oriented, smooth, connected, compact 4-manifold, write $\pi:=\pi_1(X)$, and suppose that
\begin{enumerate}[label=(\roman*)]
\item\label{it:main-guarantee1} the map $I_2\colon H_2(\pi;\Z_{(2)})\to L_6(\Z[\pi])_{(2)}$ is trivial, and

\item\label{it:main-guarantee2} $H_1(\pi ; \Z_{(2)})$ has no torsion.
\end{enumerate}
Then a diffeomorphism which is topologically pseudo-isotopic to the identity is smoothly pseudo-isotopic to the identity.
\end{thmx}

For some $\pi_1(X)$ it was already known that homotopic diffeomorphisms are smoothly pseudo-isotopic, and thus the conclusion of Theorem~\ref{thm:main-guarantee} was already known for these cases. For closed $4$-manifolds this is the case for the trivial group due to Kreck~\cite[Theorem~1]{Kreck-isotopy-classes}, and more generally for free groups by Krannich--Kupers~\cite[Theorem~B]{KK24}. The following gives additional examples where the hypotheses of Theorem~\ref{thm:main-guarantee} are satisfied, and hence provides $4$-manifolds~$X$ where topological pseudo-isotopy of diffeomorphisms implies smooth pseudo-isotopy.

\begin{example}
Any group $\pi$ with $H_1(\pi;\Z_{(2)})$ torsion-free and $H_2(\pi;\Z_{(2)})=0$ satisfies the hypotheses; for example, free groups, any finite group of odd order, or knot groups. Further ad hoc specific examples of this sort can also be constructed; for example, $\pi_1(S^1\times \Sigma)$ or $\pi_1(\Sigma)$, where $\Sigma$ is a $\Z_{(2)}$-homology sphere of some dimension.

If $\pi$ is a finite group whose 2-Sylow subgroup is abelian or generalised quaternion, then it follows from \cite{Stein} (see also \cite[p.~177]{MR557167}) that the map $I_2$ is trivial. So such groups with~$H_1(\pi;\Z_{(2)})=0$ satisfy the hypotheses. For example, the alternating group $A_5$ satisfies this, so~$I_2=0$ even though~$H_2(A_5;\Z_{(2)}) \cong \Z/2 \neq 0$.

Moreover, by functoriality of the assembly map, and hence of $I_2$, the assumptions \ref{it:main-guarantee1} and \ref{it:main-guarantee2} are closed under free products, and hence Theorem~\ref{thm:main-guarantee} holds for $X_1 \# X_2$ if it holds for $X_1$ and $X_2$.
\end{example}

\begin{corx}\label{cor:mainD}
Let $X$ be an oriented, smooth, connected, compact 4-manifold, write $\pi:=\pi_1(X)$. Assume $\pi$ satisfies the hypotheses \ref{it:main-guarantee1} and \ref{it:main-guarantee2} of Theorem~\ref{thm:main-guarantee} and also that $\mathrm{Wh}_2(\pi)=0$. Then a self-diffeomorphism of $X$ that is topologically isotopic to the identity is moreover stably smoothly isotopic to the identity.
\end{corx}

\begin{proof}
Consider the topological isotopy as topological pseudo-isotopy. By \cref{thm:main-guarantee} there is a smooth pseudo-isotopy, and then by Gabai's theorem again \cite[Theorem~2.5]{Gabai-22}, which applies since $\mathrm{Wh}_2(\pi)=0$, there is a stable smooth isotopy.
\end{proof}

Using \cref{cor:mainD}, we can expand the class of groups for which Question~\ref{question:motivation} is known to have a positive answer, which was previously limited to free groups, as discussed above. We note that the conditions of \cref{cor:mainD} are closed under free products of groups.

\begin{example}
It is conjectured that $\Wh_2(\pi)=0$ for all torsion-free groups~\cite[Conjecture~5.22]{FarrellJones}, and this conjecture is known to hold whenever the torsion-free group satisfies the ``Farrell--Jones Conjecture for $K$-theory for torsion-free groups and regular rings'', which itself follows from the full Farrell--Jones conjecture~\cite[Theorem~13.65~(i),~(ii),~(vii)]{FarrellJones}. For a summary of the many classes of groups known to satisfy these conjectures, see~\cite[Theorem 16.1]{FarrellJones}.
In particular, every 3-manifold group satisfies the full Farrell--Jones conjecture~\cite{Bartels-Farrell-Luck-FJ}, and so every classical knot group satisfies the conditions of \cref{cor:mainD}.
\end{example}

\begin{example}
Dunwoody~\cite{Dunwoody} showed that $\Wh_2(\Z/3)=0$, and so $\pi=\Z/3$ also satisfies the conditions of \cref{cor:mainD}.
\end{example}

%\medskip
\subsection{Realising Kirby--Siebenmann obstructions}
Thirdly, we discuss the possibility of realising a given second $\Z/2$-homology class as the Kirby--Siebenmann obstruction to smoothing a topological pseudo-isotopy. The following shows that after stabilising a finite number of times we can always do so: this provides a plentiful supply of diffeomorphisms with topological pseudo-isotopies to the identity.

\begin{thmx}\label{thm:main-realise}
For each smooth, connected, compact $4$-manifold $X$ there is a $g \geq 0$ such that every class $x \in H_2(X \# g(S^2 \times S^2) ;\Z/2)$ arises as $\KS(F)$ for some topological pseudo-isotopy $F$ from the identity to some diffeomorphism $f\colon X \# g(S^2 \times S^2) \to X \# g(S^2 \times S^2)$.
\end{thmx}

\subsection{Towards the proof of Theorem~\ref{thm:main-example}}
Let us comment on how the results described above contribute to the proof of Theorem~\ref{thm:main-example}. To prove Theorem~\ref{thm:main-example}, we will choose an aspherical 4-manifold~$A$ with fundamental group $\pi$, and an $x \in H_2(A;\Z/2) = H_2(\pi;\Z/2)$ that maps nontrivially under the Bockstein $\delta_* \colon H_2(\pi;\Z/2) \to H_1(\pi;\Z)$. (In particular, $H_1(\pi ; \Z_{(2)})$ has torsion, in contrast with the assumption made in \cref{thm:main-guarantee}.) Using Theorem~\ref{thm:main-realise}, there is a $g$ such that the class $x \oplus 0 \in H_2(A \# g(S^2 \times S^2) ; \Z/2)$ is $\KS(F)$ for some topological pseudo-isotopy $F$ from the identity to a diffeomorphism $f$ of $X := A \# g(S^2 \times S^2)$. For well-chosen $A$, this will be the diffeomorphism in the statement of Theorem~\ref{thm:main-example}. As $x \neq 0$ the topological pseudo-isotopy $F$ cannot be smoothed, by Theorem~\ref{thm:main-invariant}, but to prove Theorem~\ref{thm:main-example} we must show that there is no smooth pseudo-isotopy from the identity to $f$. We will do so by showing that for carefully chosen $A$ the nontrivial class $\delta_*(x)  \in H_1(\pi ; \Z)$ depends only on $f$, up to smooth pseudo-isotopy, and not on $F$. This uses surgery theory, and comes down to a delicate analysis of $\PL$ normal invariants of $X$ and $A$.

The further statement that $f$ is not even smoothly stably pseudo-isotopic to the identity follows from choosing the pseudo-isotopy $F$ carefully enough that it fixes a disc throughout, on which we perform the stabilisations, and then showing the obstruction $\KS(F)$ persists in this stabilisation. This shows that a particular stabilisation of $f$ is not smoothly pseudo-isotopic to the identity. Some extra care is required afterwards as stabilising a diffeomorphism entails a choice to fix a disc, so there is no unique smooth pseudo-isotopy class of stabilisation of a diffeomorphism, when the $4$-manifold is not simply connected; see~\cite{MR3355110}*{\textsection 5} for a discussion. However, by making a precise analysis of the ambiguity introduced by the choice, we are able to complete the proof of Theorem~\ref{thm:main-example}.

\subsection{Discussion}
For a pair of diffeomorphisms of a compact smooth 4-manifold, the following implications hold between the different notions of isotopy rel.\ boundary.
\[\text{Smoothly stably isotopic} \, \Rightarrow \, \text{Smoothly pseudo-isotopic} \, \Rightarrow \, \text{Top.\ pseudo-isotopic} \, \Leftarrow \, \text{Top.\  isotopic.}
\]
The first implication is due to Gabai~\cite{Gabai-22}, while the others are immediate from the definitions. \cref{cor:mainD} yields some cases where the first two implications can be reversed and \cref{question:motivation} has a positive answer.

There is not currently a useful theory for improving topological pseudo-isotopies of non-simply connected 4-manifolds to topological isotopies, but if such a theory could be developed, then we would hope it could be used to upgrade our examples from \cref{thm:main-example} and produce diffeomorphisms topologically isotopic to the identity, but not smoothly stably so.

\vspace{2ex}

\noindent\textbf{Organisation.}
In \cref{section:reformulation} we reformulate the main theorems in terms of block automorphisms, which will aid us to structure the proofs.
\cref{section:proofs-A-B} contains the proofs of \cref{thm:main-invariant} and \cref{cor:main-B}.
In \cref{sec:proof-thm-C} we prove \cref{thm:main-guarantee}, in
\cref{sec:proof-realisation-thm} we prove \cref{thm:main-realise}
In \cref{sec:newsection} we prove a required result about the effect of stabilisation on topological pseudo-isotopy classes, and finally in \cref{sec:proof-thm-main-example} we construct the examples promised in \cref{thm:main-example}.

\vspace{2ex}

\noindent\textbf{Acknowledgements.}
We thank Jim Davis for generously discussing $\mathbb{L}$-spectra with us, and the anonymous referees for helpful suggestions and comments that led to improvements in the article, in particular to the creation of Section 7.
MP was partially supported by EPSRC New Investigator grant EP/T028335/2 and EPSRC New Horizons grant EP/V04821X/2.
ORW was partially supported by the ERC under the European Union’s Horizon 2020 research and innovation programme (grant agreement No.\ 756444) and by the Danish National Research Foundation through the
Copenhagen Centre for Geometry and Topology~(DNRF151).

\vspace{2ex}

\noindent\textbf{Conventions.} For pointed topological spaces $A$ and $B$, the notation $[A,B]$ will indicate the set of pointed continuous maps $A\to B$, up to pointed homotopy. As a consequence, $[A_+,B]$ denotes unbased homotopy classes of maps from $A$ to $B$. For spectra $L$ and $K$, the notation~$[L,K]$ will indicate morphisms $L\to K$, up to homotopy. The sense in which the notation~$[-,-]$ is being used will always be clear from the objects in question. For maps of spectra $f,g\colon L\to K$, we will write~$f-g\colon L\to K$ to be any representative map of spectra for the homotopy class $[f]-[g]\in[L,K]$. A map $f \colon X \to X$ from a manifold to itself is \emph{rel.\ boundary} if $f|_{\partial X} = \Id_{\partial X}$.

\section{Reformulation via block automorphisms}\label{section:reformulation}

We will now reformulate Theorems~\ref{thm:main-example},~\ref{thm:main-invariant},~\ref{thm:main-guarantee}, and~\ref{thm:main-realise} in a unified framework that will make proofs clearer, and illuminate the structure of the general problem of smoothing pseudo-isotopies. Throughout this section, let $X$ be a smooth, compact, connected $4$-manifold.

\subsection{Block automorphism groups}\label{sec:block}

Let $\pseudoHomeo^+_\partial(X)$ and $\pseudoDiffeo^+_\partial(X)$ denote the geometric realisations of the Kan semi-simplicial groups of orientation-preserving block self-homeomorphisms and block self-diffeomorphisms of $X$ rel.~boundary, respectively; see e.g.~\cite[Appendix I, \textsection3]{BLR} with corrections in the smooth case in \cite[Section 2.2]{HLLRW}. Eliding some technical details regarding collars in the smooth case, the $p$-simplices consist of homeomorphisms (or diffeomorphisms) $\phi\colon X \times \Delta^p \to X \times \Delta^p$ which are the identity on $(\partial X) \times \Delta^p$, preserve orientation, and send~$X \times \sigma$ to $X \times \sigma$ for every face $\sigma \subseteq \Delta^p$. In particular, a 0-simplex is precisely a homeomorphism (or diffeomorphism) of $X$, and a 1-simplex is precisely a topological (or smooth) pseudo-isotopy. Thus the groups of path components,
\[
\pi_0 (\pseudoHomeo^+_\partial(X)) \quad\text{and}\quad \pi_0(\pseudoDiffeo^+_\partial(X)),
\]
are the groups of orientation-preserving self-homeomorphisms up to topological pseudo-isotopy (or self-diffeomorphisms up to smooth pseudo-isotopy). In other words, they are the \emph{pseudo-mapping class groups}.

There is an analogous Kan semi-simplicial monoid $\hAut^+_\partial(X)$, whose $p$-simplices are as above but relaxing the condition that $\phi$ be a homeomorphism, by only asking that it be a homotopy equivalence; see e.g.\ \cite[Section 1.5]{Krannich} for a discussion, which in particular explains why the geometric realisation of this semi-simplicial set is weakly equivalent to the topological space ${\operatorname{hAut}\mkern 0mu}^+_\partial(X)$ of boundary- and orientation-preserving homotopy automorphisms of $X$. Further insisting that $\phi$ be a simple homotopy equivalence defines a Kan semi-simplicial monoid $\sAut^+_\partial(X)$, to which the same discussion applies: the  geometric realisation of $\sAut^+_\partial(X)$ corresponds to those path components ${\operatorname{sAut}\mkern 0mu}^+_\partial(X) \subseteq {\operatorname{hAut}\mkern 0mu}^+_\partial(X)$ consisting of simple homotopy automorphisms.
In this model there are forgetful maps
\[
\pseudoDiffeo^+_\partial(X) \lra \pseudoHomeo^+_\partial(X) \lra \sAut^+_\partial(X) \lra \hAut^+_\partial(X)
\]
of Kan semi-simplicial monoids.

\subsection{Pseudo-isotopy classes of topological pseudo-isotopies}
We wish to study topological pseudo-isotopies from the identity to a diffeomorphism, and this will be clearest if we carefully define a certain set of such.

\begin{definition}
We consider topological pseudo-isotopies $F\colon X \times I \to X \times I$ with $F|_{X \times \{0\}} = \Id_X$ and such that $f := F|_{X \times \{1\}}$ is a diffeomorphism, all rel.~$\partial X$. We consider these up to the equivalence relation of a further pseudo-isotopy: define $F_0 \sim F_1$ if there is a homeomorphism $G$ of $X \times I \times I$ which restricts to $F_i$ on $X \times I \times \{i\}$, to the identity on~$\big(X \times \{0\} \times I\big) \cup \big((\partial X) \times I \times I\big)$, and to a diffeomorphism on $X \times \{1\} \times I$ (which is then a smooth pseudo-isotopy from $f_0$ to $f_1$). We write $Q(X)$ for the set of such equivalence classes, and note that it is a group under composition.

\begin{figure}[h]
\begin{tikzpicture}[scale=0.8]
\draw[thick, fill=white] (0,0) rectangle (2,2);
\node at (1,1) {$G$};
\node[left] at (0,1) {$\Id$};
\node[below] at (1,0) {$F_0$};
\node[below] at (2,0) {$f_0$};
\node[right] at (2,1) {$\cong_{C^\infty}$};
\node[above] at (1,2) {$F_1$};
\node[above] at (2,2) {$f_1$};
\end{tikzpicture}
\caption{Schematic for $X\times I\times I$ with the $X$ direction suppressed. Maps defining the equivalence relation $(f_0,F_0)\sim (f_1,F_1)$ for $Q(X)$ are depicted.}\label{fig:QX}
\end{figure}
\end{definition}

For $[F] \in Q(X)$, we explained in the introduction how to produce a smooth manifold $\partial(X \times I \times I)_F$ by pulling back the smooth structure $\sigma$ on $X \times I \times \{1\}$ using the homeomorphism $F$. Then $X \times I \times I$ is a topological $6$-manifold with a given smooth structure on its boundary so has a Kirby--Siebenmann invariant
\[
\KS(F) := \ks(X \times I \times I, \partial(X \times I \times I)_F) \in H^4(X \times I \times I, \partial(X \times I \times I) ; \Z/2).
\]
In the following section, once we describe the details of smoothing theory, we will prove the following.

\begin{lemma}\label{lem:ksConcordanceInvariant}
The function  $[F] \mapsto \KS(F)\colon Q(X) \to H_2(X;\Z/2)$ is well-defined and is a homomorphism.
\end{lemma}

The proof will be given after \cref{lem:stronger}.

\subsection{The reformulation}\label{sec:reformulation}

The definition of $Q(X)$ was obtained by spelling out the description of~$\pi_0$ of the homotopy fibre of the map $\pseudoDiffeo^+_\partial(X) \to \pseudoHomeo^+_\partial(X)$. Thus there is an exact sequence
\begin{equation}\label{eqn:key-long-exact-seq}
\pi_1(\pseudoDiffeo^+_\partial(X)) \lra \pi_1(\pseudoHomeo^+_\partial(X)) \overset{\alpha}\lra Q(X) \overset{\beta}\lra \pi_0(\pseudoDiffeo^+_\partial(X)) \overset{\gamma}\lra \pi_0(\pseudoHomeo^+_\partial(X)).
\end{equation}
Theorems~\ref{thm:main-example}, ~\ref{thm:main-invariant},~\ref{thm:main-guarantee}, and~\ref{thm:main-realise} in the introduction can be phrased in terms of the behaviour of this exact sequence, and the map $\KS\colon Q(X) \to H_2(X;\Z/2)$, as follows.

\begin{theorem}\label{thm:main-technical}
Let $X$ be a smooth, connected, compact 4-manifold.
\begin{enumerate}[label=(\Alph*)]
\item There exists an $X$ for which $\gamma$ is not injective, i.e.~there is a diffeomorphism of $X$ which is topologically pseudo-isotopic to the identity but not smoothly so.

\item\label{it:main-technicalA} The homomorphism $\KS\colon Q(X) \to H_2(X;\Z/2)$ is injective.

\setcounter{enumi}{3}

\item\label{it:main-technicalC} Suppose $X$ is orientable and the fundamental group $\pi = \pi_1(X)$ is such that the map
\[
I_2\colon H_2(\pi ; \Z_{(2)}) \lra L_6^s(\Z[\pi])_{(2)}
\]
specified in Definition~\ref{defn:I2} is trivial. Then the image of the composition
\[
\pi_1(\pseudoHomeo^+_\partial(X)) \overset{\alpha}\lra Q(X) \overset{\KS}\lra H_2(X;\Z/2)
\]
contains the kernel $\ker\big(H_2(X;\Z/2)\xrightarrow{\delta_*}H_1(X;\Z_{(2)})\big)$ of the Bockstein operator $\delta_*$ for the coefficient sequence $0\to\Z_{(2)}\overset{2}\to\Z_{(2)}\to\Z/2\to 0$.
\setcounter{enumi}{5}
\item There is a $g \geq 0$ such that for the manifold $X' := X \# g(S^2 \times S^2)$, the homomorphism $\KS\colon Q(X') \to H_2(X';\Z/2)$ is a bijection.

\end{enumerate}
\end{theorem}

\begin{remark}$\,$
\begin{enumerate}[label=(\roman*)]
\item Theorem \ref{thm:main-technical} \ref{it:main-technicalA} sounds a little weaker than Theorem \ref{thm:main-invariant}, as the latter says that if $\KS(F)=0$ then $F$ is topologically \emph{isotopic} to a smooth pseudo-isotopy, whereas the former only says that $F$ is topologically \emph{pseudo-isotopic} to a smooth pseudo-isotopy. But in fact they are equivalent, as we will show in Theorem \ref{thm:obstruction-thy}.
\item Theorem~\ref{thm:main-technical}~\ref{it:main-technicalA} and~\ref{it:main-technicalC} imply Theorem~\ref{thm:main-guarantee}, as we explain below.
Theorem~\ref{thm:main-technical}~\ref{it:main-technicalC} is potentially stronger than Theorem~\ref{thm:main-guarantee}, as it could also be used to analyse particular diffeomorphisms of $X$ when $H_1(\pi;\Z_{(2)})$ has $2$-torsion.
\end{enumerate}
\end{remark}

In the following sections we will prove the results in the introduction, but we will usually prove them in the form given in Theorem \ref{thm:main-technical}.
For \cref{thm:main-guarantee} we explain the connection next.

\begin{proof}[Proof of \cref{thm:main-guarantee} assuming Theorem~\ref{thm:main-technical}~\ref{it:main-technicalA} and~\ref{it:main-technicalC}]
The Postnikov truncation $\mft\colon X\to \BB\pi$ determines a commutative diagram
\[
\begin{tikzcd}
H_2(X;\Z/2)\ar[r, "\delta_*"]\ar[d, two heads, "\mft_*"]
& H_1(X;\Z_{(2)})\ar[d, "\cong"', "\mft_*"]
\\
H_2(\pi;\Z/2)\ar[r, "\delta_*"]
& H_1(\pi;\Z_{(2)})
\end{tikzcd}
\]
where surjectivity of $\mft_*$ on $H_2$ is easily seen by considering a model for $\BB\pi$ consisting of a cell complex for $X$ with additional cells of dimension 3 or greater attached. The hypothesis that $H_1(\pi;\Z_{(2)})$ has no $2$-torsion is equivalent to the lower $\delta_*$ being the 0 map. A diagram chase shows this is equivalent to the upper $\delta_*$ being the 0 map. By Theorem~\ref{thm:main-technical}~\ref{it:main-technicalC}, the composition $\KS\,{\circ}\,\alpha$ is then surjective, so in particular $\KS$ is surjective. Then by Theorem~\ref{thm:main-technical}~\ref{it:main-technicalA}, $\KS$ is an isomorphism, and so $\alpha$ is surjective and hence $\gamma$ is injective by exactness of \eqref{eqn:key-long-exact-seq}. The latter is equivalent to saying that every self-diffeomorphism of $X$ which is topologically pseudo-isotopic to the identity is in fact smoothly pseudo-isotopic to the identity.
\end{proof}

\section{Proof of Theorem \ref{thm:main-invariant} and Corollary \ref{cor:main-B}}\label{section:proofs-A-B}

\subsection{Smoothing theory}\label{sec:smoothing}

Let $Y$ be a compact $d$-dimensional topological manifold whose boundary $\partial Y$ is endowed with a fixed smooth structure. The maps classifying the topological tangent microbundle of $Y$, and the once-stabilised tangent vector bundle of $\partial Y$, yield the solid arrows in the following commutative diagram, where the right-hand map is modelled by a fibration:
\begin{equation}\label{eqn-smoothing-thy}
\begin{tikzcd}
\partial Y  \ar[rr, "T (\partial Y) \oplus \varepsilon^1"] \ar[d] && \BO(d) \ar[d] \\
Y  \ar[rr,"\tau_Y"] \ar[urr,dashed] && \BTop(d).
\end{tikzcd}
\end{equation}
The fundamental theorem of smoothing theory is then as follows \cite[Essay V]{Kirby-Siebenmann:1977-1}.

\begin{theorem}[Kirby--Siebenmann]
For $d \geq 5$, the space of smooth structures on $Y$ extending that on $\partial Y$ is homotopy equivalent to the space of dashed maps making both triangles commute.
\end{theorem}

In particular, there exists a smooth structure on $Y$ extending that on $\partial Y$ if and only if there exists a dashed map making both triangles commute. This lifting problem may be made more practical by considering the following diagram.
\begin{equation*}
\begin{tikzcd}
\BO(d) \ar[d] \rar & \BO \dar \rar & \mathrm{E}(\Top/\OO) \simeq * \dar\\
\BTOP(d) \rar & \BTop \rar{\theta} & \BB(\Top/\OO)
\end{tikzcd}
\end{equation*}
The right-hand square is formed using the fact that $\BO \to \BTop$ is an (infinite) loop map so deloops to give $\BB(\Top/\OO)$. The down-right  composition consists of two maps in a homotopy fibre sequence, so the square is cartesian, i.e.\ is a homotopy pull-back square.  The left-hand square is formed by stabilisation, and is $(d+1)$-cartesian \cite[p.\ 246 (4)]{Kirby-Siebenmann:1977-1} as long as $d \geq 5$.
(Recall we say that a square is  $n$-cartesian if the canonical map from the top left entry to the homotopy pullback of the other three entries  is $n$-connected, or equivalently if the total homotopy fibre is~$(n-1)$-connected.)
Pasting these squares with \eqref{eqn-smoothing-thy}, it follows that for $d \geq 5$ there is a smooth structure on $Y$ extending that on $\partial Y$ if and only if there is dashed map in the diagram
\begin{equation*}
\begin{tikzcd}
\partial Y  \ar[rr] \ar[d] && \mathrm{E}(\Top/\OO) \simeq * \ar[d] \\
Y  \ar[rr,"\theta \,{\circ}\, \tau_Y"'] \ar[urr,dashed] && \BB(\Top/\OO)
\end{tikzcd}
\end{equation*}
making both triangles commute. That is, the solid arrows define a class
$$\mathrm{sm}(Y, \partial Y) \in \left[Y/\partial Y, \BB(\Top/\OO)\right]$$
which is trivial if, and for $d \geq 5$ only if, $Y$ admits a smooth structure extending that on $\partial Y$.
There is a map
$$\ks\colon \BB(\Top/\OO) \lra K(\Z/2,4)$$
which is 8-connected. This follows from \cite[p.\ 246 (5)]{Kirby-Siebenmann:1977-1} and the fact that the low-dimensional groups of homotopy spheres satisfy $\Theta_5=\Theta_6=0$. Post-composing $\mathrm{sm}(Y, \partial Y)$ with the map $\ks$ gives a class
\[
\ks(Y, \partial Y) \in H^4(Y, \partial Y ; \Z/2),
\]
whose vanishing is necessary, and for $5 \leq d \leq 7$ sufficient, for the existence of a smooth structure on $Y$ extending that on $\partial Y$. (For all $d \geq 5$ the vanishing of this class is necessary and sufficient for the existence of a $\PL$-structure on $Y$ extending that on $\partial Y$ induced by its smooth structure.)

\medskip

Finally, we use this discussion to prove Lemma \ref{lem:ksConcordanceInvariant}, and in fact prove something a little stronger.
As mentioned above $\BB(\Top/\OO)$ is an infinite loop space. For a based space $Z$, this determines a group structure on $[Z,\BB(\Top/\OO)]$.

\begin{lemma}\label{lem:stronger}
Let $X$ be a smooth, compact, connected $4$-manifold. The function
\[
\Sm \colon Q(X) \lra \left[\tfrac{X \times I \times I}{\partial (X \times I \times I)}, \BB(\Top/\OO)\right];\qquad [F] \longmapsto \mathrm{sm}(X \times I \times I, \partial(X \times I \times I)_F)
\]
is well-defined and is a homomorphism.
\end{lemma}
\begin{proof}
If $[F_0] = [F_1] \in Q(X)$ then by definition there is a homeomorphism $G$ of $X \times I \times I$ which restricts to $F_i$ on $X \times I \times \{i\}$, to the identity on $(X \times \{0\} \times I) \cup ((\partial X) \times I \times I)$, and to a diffeomorphism on $X \times \{1\} \times I$. We endow
\[
\partial(X \times I \times I) \times I\subseteq \partial(X \times I \times I \times I)
\]
with a smooth structure by giving $X \times I \times \{1\} \times I$ the smooth structure $(X \times I \times \{1\} \times I)_G$ given by pulling the standard one back along $G$, and giving the rest of $\partial(X \times I \times I) \times I$ the standard smooth structure. One checks that these two rules are compatible where they overlap.
This smooth structure restricts to $\partial (X \times I \times I)_{F_i} \times \{i\}$, so the data
\begin{equation*}
\begin{tikzcd}
\partial(X \times I \times I) \times I \dar \rar& \BO(7) \ar[d] \rar & \BO \dar \rar & \mathrm{E}(\Top/\OO) \simeq * \dar\\
(X \times I \times I) \times I \rar & \BTOP(7) \rar & \BTop \rar{\theta} & \BB(\Top/\OO)
\end{tikzcd}
\end{equation*}
define a homotopy from $\mathrm{sm}(X \times I \times I, \partial(X \times I \times I)_{F_0})$ to $\mathrm{sm}(X \times I \times I, \partial(X \times I \times I)_{F_1})$.

We show the map is a homomorphism. The group structure on the set $\big[\tfrac{X \times I \times I}{\partial (X \times I \times I)}, \BB(\Top/\OO)\big]$ is the group structure given by the infinite loop space $\BB(\Top/\OO)$.  The domain is homeomorphic to the suspension $\Sigma^2 (X/\partial X)$,
hence is a co-$H$-space, while the codomain is an $H$-space. It then follows from the Eckmann--Hilton argument \cite[p.~43--4]{Spanier}, that the group structure arising from the codomain is equivalent to the group structure arising the domain, i.e.\ from ``stacking'' in the second (say) interval direction of $X\times I\times I$.

Let $[F]$ and $[F']$ be two elements of $Q(X)$. We observe that the topological manifold with smooth boundary $(X \times I \times I, \partial(X \times I \times I)_{F' \,{\circ}\, F})$ can be obtained up to isomorphism by stacking the topological manifolds with smoothed boundaries
\[
(X \times I \times I, \partial(X \times I \times I)_F) \text{ and } (X \times I \times I, (F \times\id_I)^* \partial(X \times I \times I)_{F'})
\]
in the second interval direction. Here, $(F \times\id_I)^* \partial(X \times I \times I)_{F'})$ indicates the smooth structure obtained by pulling back the smooth structure $ \partial(X \times I \times I)_{F'}$, using the restriction of the homeomorphism $F \times\id_I$, to $\partial(X\times I\times I)$; see Figure~\ref{fig:stacking}.
\begin{figure}[h]
\begin{tikzpicture}[scale=0.8]
\draw[thick, fill=white] (0,0) rectangle (4,2);
\draw[thick, fill=white] (0,2) rectangle (4,4);
\node at (2,1) {$X \times I \times I$};
\node at (2,3) {$X \times I \times I$};
\node[below] at (2,0) {$\sigma$};
\node[right] at (4,1) {$\sigma$};
\node[left] at (0,1) {$\sigma$};
\node[above] at (2,2) {$F^*\sigma$};
\node[above] at (2,4) {$(F' \circ F)^* \sigma = F^*((F')^*\sigma)$};
\node[right] at (4,3) {$(F|_{X \times \{1\}} \times \Id_I)^* \sigma  = \sigma$};
\node[left] at (0,3) {$(F|_{X \times \{0\}} \times \Id_I)^* \sigma  = \sigma$};
\end{tikzpicture}
\caption{We denote the standard smooth structure on $X \times I$ by $\sigma$. Suppressing the $X$ direction, we indicate the stacking proposed, together with the smooth structures on the various copies of $X \times I$ that make up the boundaries of the two copies of $X \times I \times I$.}\label{fig:stacking}
\end{figure}

Using this, we see that
\begin{align*}
\Sm(F' \,{\circ}\, F) &= \mathrm{sm} (X \times I \times I, \partial(X \times I \times I)_{F' \,{\circ}\, F}) \\
&= \mathrm{sm}\big((X \times I \times I, \partial(X \times I \times I)_F) \cup (X \times I \times I, (F \times\id_I)^* \partial(X \times I \times I)_{F'}) \big)\\
&= \mathrm{sm}\big(X \times I \times I, \partial(X \times I \times I)_F\big) + \mathrm{sm}\big(X \times I \times I, (F \times\id_I)^* \partial(X \times I \times I)_{F'}\big)\\
 & \in \left[\tfrac{X \times I \times I}{\partial (X \times I \times I)}, \BB(\Top/\OO)\right].
\end{align*}
The additivity of $\mathrm{sm}$ with respect to $\cup$ is due to the Eckmann-Hilton argument explained above.
There is an identification of topological manifolds with smoothed boundaries
\[
F \times\id_I\colon (X \times I \times I, (F \times\id_I)^* \partial(X \times I \times I)_{F'}) \overset{\cong}\lra (X \times I \times I, \partial(X \times I \times I)_{F'}),
\]
and therefore
\begin{align*}
  & \,\mathrm{sm}\big(X \times I \times I, \partial(X \times I \times I)_F\big) + \mathrm{sm}\big(X \times I \times I, (F \times\id_I)^* \partial(X \times I \times I)_{F'}\big) \\ = & \,\mathrm{sm}\big(X \times I \times I, \partial(X \times I \times I)_F\big) + (F \times\id_I)^* \mathrm{sm}\big(X \times I \times I,  \partial(X \times I \times I)_{F'}\big) \\ = & \,\Sm(F) + (F \times\id_I)^*\Sm(F').
\end{align*}
The map of pairs
$$F\colon (X \times I, \partial(X \times I)) \lra (X \times I , \partial(X \times I))$$
is the identity on $X \times \{0\} \cup \partial(X) \times I$ and sends $X \times \{1\}$ into itself, so by Lemma \ref{lem:excrosseye} below it is homotopic as a map of pairs to the identity.
This implies that
\[ \Sm(F) + (F \times \Id_I)^*\Sm(F') =   \Sm(F) + \Sm(F').\]
We deduce that $\Sm(F' \,{\circ}\, F) = \Sm(F) + \Sm(F') \in \left[\tfrac{X \times I \times I}{\partial (X \times I \times I)}, \BB(\Top/\OO)\right]$,
as required.
\end{proof}

We now prove the short lemma that was just used in the proof of~\cref{lem:stronger}.

\begin{lemma}\label{lem:excrosseye}
Let $Y$ be a compact manifold, possibly with boundary, and let
$$f \colon (Y\times I, \partial (Y \times I)) \lra (Y\times I, \partial (Y \times I))$$
be a map of pairs that restricts to the identity on $Y \times \{0\} \cup \partial(Y) \times I$ and sends $Y \times \{1\}$ into itself. Then $f$ is homotopic as a map of pairs to the identity.
\end{lemma}
\begin{proof}
We exhibit a homotopy. Write $f(y,t) = (\gamma(y,t), \tau(y,t))$, where this defines the functions $\gamma$ and $\tau$. There is a preliminary homotopy given by the formula
$$f_s(y,t) = (\gamma(y,t), (1-s)\tau(y,t) + s t),$$
satisfying $f_0=f$ and $f_1(y,t) = (\gamma(y,t), t)$. This homotopy is fixed on $Y \times \{0\} \cup \partial(Y) \times I$, and if~$(y,1) \in Y \times \{1\}$ so that $f(y,1) = (\gamma(y,1), 1)$ then $f_s(y,1) = (\gamma(y,1), 1) = f(y,1) \in \partial (Y, \times I)$: thus it is a homotopy which is fixed on $\partial (Y \times I)$.

We make a second homotopy by the formula
$$f_{1,s}(y,t) = (\gamma(y, (1-s)t),t).$$
This satisfies $f_{1,0} = f_1$, and $f_{1,1}(y,t) = (\gamma(y,0), t) = (y,t)$, so $f_{1,1}$ is the identity. One verifies that it sends $\partial(Y \times I)$ into itself, so gives a homotopy of maps of pairs.
\end{proof}

\begin{proof}[Proof of Lemma \ref{lem:ksConcordanceInvariant}]
Postcompose the function from Lemma~\ref{lem:stronger} with
\begin{equation}\label{eqn:display-ks}
\ks\colon \left[\tfrac{X \times I \times I}{\partial (X \times I \times I)}, \BB(\Top/\OO)\right] \lra \left[\tfrac{X \times I \times I}{\partial (X \times I \times I)}, K(\Z/2,4)\right] \cong \widetilde{H}^4(X \times I \times I,\partial(X \times I \times I);\Z/2),\end{equation}
and  Poincar{\'e} duality. The composition $PD \circ \ks \circ \Sm \colon Q(X) \to H_2(X;\Z/2)$ is exactly the function~$\KS$ in the statement of  Lemma \ref{lem:ksConcordanceInvariant}. We need to see that it is a homomorphism. By Lemma~\ref{lem:stronger}, the function $\Sm$
is a homomorphism, and certainly Poincar\'{e} duality is.  The $H$-space structures on $\BB(\Top/\OO)$ and $K(\Z/2,4)$ need not agree a priori,  so it is not automatic that the map in~\eqref{eqn:display-ks} is a homomorphism.
 However, the map $\BB(\Top/\OO) \to K(\Z/2,4)$ factors through the Postnikov truncation:  $\BB(\Top/\OO) \to  \tau_{\leq 4} \BB(\Top/\OO) \xrightarrow{\simeq} K(\Z/2,4)$. The Postnikov truncation inherits an $H$ space structure, which follows from obstruction theory (the obstructions to extending the multiplication have coefficients in $\pi_i(\tau_{\leq k}X)$ for $i \geq k+1$, which vanishes). Hence the first map in our factorisation is an $H$-space map. Again by obstruction theory, Eilenberg--Maclane spaces have a unique $H$-space structure, so the second map is also an $H$-space map.
 (Since the domain has two~$I$ factors, we could also have applied the Eckmann--Hilton argument~\cite[p.~43--4]{Spanier}  again.)
 It follows that $\ks$ is also a homomorphism, which completes the proof.
\end{proof}

Recall from the introduction that given a compact, smooth, connected~$4$-manifold $X$ and a homeomorphism $F \colon X \times I \to X \times I$ restricting to diffeomorphisms $f_i\colon X\times\{i\}\to X\times\{i\}$ for $i=0,1$, we can endow the boundary of $X \times I \times I$ with a smooth structure $\partial(X \times I \times I)_F$. Using the discussion in the previous section we define
\[
\KS(F) := \ks(X \times I \times I, \partial(X \times I \times I)_F) \in H^4(X \times I \times I, \partial(X \times I \times I)_F; \Z/2) \cong H_2(X;\Z/2).
\]
We discuss the behaviour of this invariant under connected sums, which will be needed later for the proof of Theorem~\ref{thm:main-example}.

\begin{lemma}
\label{lem:cheaperlemma}
For $i=1,2$ let $Z_i$ be a compact, connected smooth $4$-manifold and suppose $F_i\colon Z_i\times I\to Z_i\times I$ is a topological pseudo-isotopy from a diffeomorphism $f_i$ to the identity. Suppose there exist $4$-balls $D_i\subseteq Z_i$ such that $(F_i)|_{D_i\times I}=\Id$ and use these to glue the maps $F_1$ and $F_2$ to form a pseudo-isotopy
\[
F_1\# F_2\colon (Z_1\#Z_2)\times I\lra (Z_1\#Z_2)\times I.
\]
from $f_1\#f_2$ to the identity. Then the natural isomorphism
\[
H_2(Z_1;\Z/2)\oplus H_2(Z_2;\Z/2)\xrightarrow{\cong}H_2(Z_1\#Z_2;\Z/2)
\]
sends $(\KS(F_1),\KS(F_2))$ to $\KS(F_1\#F_2)$.
\end{lemma}

\begin{proof}
Define $X_i:=Z_i\sm D_i$ for $i=1, 2$, and identify $\partial D_i=Y\cong S^3$. Write $X' :=X_1\cup_Y X_2$ for $Z_1 \# Z_2$. As the topological pseudo-isotopy $F_1 \# F_2$ is smooth along $Y \times I$, there is an induced factorisation of the map $\Sm(F_1 \# F_2) : \tfrac{X' \times I \times I}{\partial(X' \times I \times I)} \to \BB(\Top/\OO)$ through the quotient
$$\tfrac{X' \times I \times I}{\partial(X' \times I \times I) \cup Y \times I \times I} \cong \tfrac{X_1 \times I \times I}{\partial(X_1 \times I \times I) \cup Y \times I \times I} \vee \tfrac{X_2 \times I \times I}{\partial(X_2 \times I \times I) \cup Y \times I \times I}.$$
The induced map from this quotient is $\Sm(F_1\vert_{X_1 \times I}) \vee \Sm(F_2\vert_{X_2 \times I})$. By the same reasoning, as $F_i\vert_{D_i \times I}$ is smooth, we have factorisations
$$\Sm(F_i) : \tfrac{Z_i \times I \times I}{\partial(Z_i \times I \times I)} \lra \tfrac{Z_i \times I \times I}{\partial(Z_i \times I \times I) \cup D_i \times I \times I} \cong \tfrac{X_i \times I \times I}{\partial(X_i \times I \times I) \cup Y \times I \times I} \xrightarrow{\Sm(F_i\vert_{X_i \times I})}  \BB(\Top/\OO).$$
The first map here is given by collapsing the subspace $\tfrac{D_i \times I \times I}{D_i \times \partial(I \times I)} \simeq S^2$, and as $\BB(\Top/\OO)$ is 3-connected $\Sm(F_i)$ uniquely determines $\Sm(F_i\vert_{X_i \times I})$ up to homotopy. Applying $PD \circ \ks$, this translates to the statement that
under the isomorphisms
$$H_2(Z_1;\Z/2)\oplus H_2(Z_2;\Z/2)\xleftarrow{\cong} H_2(X_1;\Z/2)\oplus H_2(X_2;\Z/2)\xrightarrow{\cong}H_2(Z_1\#Z_2;\Z/2)$$
the classes $(\KS(F_1), \KS(F_2))$ and $\KS(F_1 \# F_2)$ correspond.
\end{proof}

\subsection{Proof of Theorem \ref{thm:main-invariant}}

\begin{theorem}\label{thm:obstruction-thy}
 The following are equivalent:
\begin{enumerate}[label=(\roman*)]
\item\label{item:1} $\KS(F)=0$;
\item\label{item:2} $X \times I \times I$ admits a smooth structure extending $\partial(X \times I \times I)_F$;
\item\label{item:4} $F$ is topologically isotopic, rel.~boundary, to a smooth pseudo-isotopy;
\item\label{item:4prime} $F$ is topologically pseudo-isotopic, rel.~boundary, to a smooth pseudo-isotopy.
\end{enumerate}
\end{theorem}

\begin{proof}
That \ref{item:1}$\iff$\ref{item:2} is by the discussion in the previous section, using that $\dim(X \times I \times I)=6$.

We prove that \ref{item:2}$\implies$\ref{item:4}. If $X \times I \times I$ is endowed with a smooth structure extending~$\partial(X \times I \times I)_F$, then it is a concordance rel.\ boundary from the smooth structure $(X \times I)_\sigma$ to $(X \times I)_{F^*\sigma}$.
Since $X \times I$ has dimension five and the smooth structures already agree on the 4-dimensional boundary $X \times \{0,1\}$, we may apply Kirby--Siebenmann's concordance implies isotopy~\cite[Essay~I,~Theorem~4.1]{Kirby-Siebenmann:1977-1} for smooth structures, to obtain an isotopy of homeomorphisms $G_t \colon X \times I \to X \times I$, for $t \in I$, such that $G_0=\Id_{X \times I}$ and $G_1^* \sigma = F^* \sigma$.
Then~$F \,{\circ}\, G_t^{-1} \colon X \times I \to X \times I$ is an isotopy from $F$ to $F \,{\circ}\, G_1^{-1}$, and \[(F \,{\circ}\, G_1^{-1})^{*}\sigma = (G_1^{-1})^*F^*\sigma = (G_1^{-1})^*G_1^*\sigma = (G_1 \,{\circ}\, G_1^{-1})^* \sigma = \Id^* \sigma = \sigma.\]
Hence $F \,{\circ}\, G_1^{-1} \colon (X \times I)_{\sigma} \to (X \times I)_{\sigma}$ is a diffeomorphism, i.e.~a smooth pseudo-isotopy of~$X \times I$, that is topologically isotopic to $F$. Thus \ref{item:4} holds.

Certainly \ref{item:4}$\implies$\ref{item:4prime}. To see that \ref{item:4prime}$\implies$\ref{item:1},
 suppose that $F \colon X \times I \to X \times I$ is topologically pseudo-isotopic rel.\ boundary to a smooth pseudo-isotopy, via a pseudo-isotopy~$G \colon X \times I \times I \to X \times I \times I$. Let $f_i := F|_{X \times\{i\}} \colon X \to X$ be the diffeomorphisms obtained by restricting $F$, for~$i=0,1$.
Define $F' := F \circ (f_0 \times \Id_I)^{-1}$. This is a topological pseudo-isotopy from $\Id_X$ to the diffeomorphism~$f_1 \circ f_0 \colon X \to X$, and hence determines an element $[F'] \in Q(X)$.  Moreover~$G' := G \circ (f_0 \times \Id_I \times \Id_I)^{-1}$ is a pseudo-isotopy from $F'$ to a diffeomorphism of $X \times I$, i.e.\ to a smooth pseudo-isotopy. It follows that $[F']$ is equivalent in $Q(X)$ to a smooth pseudo-isotopy~$F'' := G'|_{X \times I \times \{1\}}$, and so $\KS(F') = \KS(F'')$ by Lemma \ref{lem:ksConcordanceInvariant}.
Since $F''$ is a diffeomorphism,~$\partial(X \times I \times I)_{F''} \cong \partial(X \times I \times I)$, where the latter denotes the standard smooth structure on the boundary. But $\ks(X \times I \times I, \partial(X \times I \times I)) = 0$, so $\KS(F) = \KS(F'') =0$.
\end{proof}

In particular, the implication \ref{item:1}$\implies$\ref{item:4} is Theorem \ref{thm:main-invariant}. The implication \ref{item:1}$\implies$\ref{item:4prime} is Theorem~\ref{thm:main-technical}~\ref{it:main-technicalA}, which seems slightly weaker than Theorem \ref{thm:main-invariant} but as we have just seen is in fact not.

\subsection{Proof of Corollary \ref{cor:main-B}}\label{sub:proof-cor-B}

For a compact smooth $d$-manifold $X$, let $\mathcal{P}^{\Diff}(X)$ denote the space of smooth pseudo-isotopies of $X$, i.e.\ the set of diffeomorphisms $F \colon X \times I \to X \times I$ such that $F$ acts as the identity near $(X \times \{0\}) \cup (\partial X \times I)$, endowed with the Whitney topology.
Similarly let~$\mathcal{P}^{\Top}(X)$ denote the space of topological pseudo-isotopies of $X$, with the compact-open topology.
For $d \geq 4$, Hatcher--Wagoner~\cite{HW} defined
a homomorphism
\[\Sigma^{\Diff} \colon \pi_0\mathcal{P}^{\Diff}(X) \lra \Wh_2(\pi_1(X)).\]
For $d \geq 5$, Burgelea--Lashof--Rothenberg and Pedersen \cite{BLR}*{Appendix A} indicated how to extend the definition of~$\Sigma^{\Diff}$ to the topological category. The details of the analogous construction were worked out for $d=4$ by Galvin--Nonino~\cite[Theorem~1.1]{Galvin:2025aa}, giving rise to a homomorphism~$\Sigma^{\Top} \colon \mathcal{P}^{\Top}(X) \to \Wh_2(\pi_1(X))$ such that $\Sigma^{\Diff}$ factors through the forgetful map as
\begin{equation}\label{eqn:defn-wh2-obstruction-top-PI}
\Sigma^{\Diff} \colon  \pi_0 \mathcal{P}^{\Diff}(X) \to \pi_0 \mathcal{P}^{\Top}(X) \xrightarrow{\Sigma^{\Top}} \Wh_2(\pi_1(X)).
\end{equation}
 Finally, recall a theorem of Gabai.

\begin{theorem}[{\cite[Theorem~2.5]{Gabai-22}}]\label{thm:gabai}
Let $f\colon X\xrightarrow{\cong} X$ be an orientation-preserving diffeomorphism of a smooth, compact, oriented $4$-manifold $X$. Then $f$ is smoothly stably isotopic to the identity if and only if $f$ is smoothly pseudo-isotopic to $\Id_X$ via a smooth pseudo-isotopy $F$ with vanishing Hatcher--Wagoner obstruction $\Sigma^{\Diff}([F])=0\in \Wh_2(\pi_1(X))$.
\end{theorem}

We can now prove Corollary \ref{cor:main-B}.

\begin{proof}[Proof of Corollary \ref{cor:main-B}]
By Theorem \ref{thm:main-invariant} if $F$ has $\KS(F)=0$ then it is topologically isotopic rel.~$\partial (X \times I)$ to a smooth pseudo-isotopy $F'$. If moreover $F$ is a topological isotopy then $[F] = [\id_X] \in \pi_0 \mathcal{P}^{\Top}(X)$ so $\Sigma^{\Top}([F])=0\in \operatorname{Wh}_2(\pi_1(X))$.  By \eqref{eqn:defn-wh2-obstruction-top-PI}, and since $[F'] \mapsto [F]$ under the forgetful map, we deduce that $\Sigma^{\Diff}([F']) = \Sigma^{\Top}([F]) =0$.
By Gabai's Theorem~\ref{thm:gabai}, this is equivalent to $f$ being smoothly stably isotopic to the identity.
\end{proof}

\section{Dualities and orientations}
We establish notation and recall some fundamental concepts from stable homotopy theory.
First, we introduce our notation for the standard spectra we shall invoke.
We write $\bS$ for the sphere spectrum, and $\bS^d:=\Sigma^d\bS$ for the sphere spectrum shifted by $d$. Of course, $\bS^0=\bS$, and note $\pi_n(\bS^d) \cong \pi_{n-d}^s(\pt_+) = \pi_{n-d}^s$.  For an abelian group $A$ and $k \in \Z$, we write $\mathrm{H}A[k]$ for the Eilenberg--Maclane spectrum with $\pi_k(\mathrm{H}A[k]) \cong A$ and $\pi_j(\mathrm{H}A[k])=0$ for $j \neq k$. Finally let $\MSTop$ be the oriented Thom spectrum.

Given spectra $X$ and $E$, we define the $E$-theory homology and cohomology, respectively, as
\[
E_r(X):=[\bS^r,X\wedge E]
\qquad\text{and}\qquad
E^r(X):=[X,\Sigma^r E].
\]
If $Y$ is a pointed space, we write $E_r(Y):=E_r(\Sigma^\infty Y)$ and $E^r(Y):=E^r(\Sigma^\infty Y)$ for the $E$-theory homology and cohomology.

\subsection{Spanier--Whitehead and Atiyah duality}\label{sec:Atiyah}
Spectra $A$ and $B$ are \emph{Spanier--Whitehead dual} if there exists a \emph{duality morphism} (see \cite[IV.2.3(a)]{Rudyak}), which is a \emph{coevaluation} map $\text{coev}\colon \bS\to A\wedge B$ satisfying certain properties. The existence of a  coevaluation map is equivalent to the existence of an \emph{evaluation map} $\text{ev}\colon B\wedge A\to\bS$ satisfying certain properties~\cite[IV.2.6(a)]{Rudyak}.
Moreover, if $A$ and~$B$ Spanier--Whitehead dual then for all spectra $C$ and $D$ coevaluation induces an isomorphism
\[
[D\wedge A, C]\xrightarrow{\cong} [D,C\wedge B];\qquad \varphi\mapsto (\varphi\wedge\Id_{B})\circ (\Id_{D}\wedge\,\text{coev})
\]
(see~\cite[IV.2.5(ii)]{Rudyak}).

Given spectra $A$ and $B$ with respective Spanier--Whitehead duals $A^\vee$
and $B^\vee$, the \emph{Spanier--Whitehead dual}
of a map $f\colon A\to B$ is the map $f^\vee \colon B^\vee \to A^\vee$, given by the composition
\[
f^\vee \colon B^\vee\xrightarrow{\simeq}\bS\wedge B^\vee\xrightarrow{\text{coev}\wedge \Id} A^\vee\wedge A\wedge B^\vee \xrightarrow{\Id\wedge f\wedge \Id} A^\vee\wedge B\wedge B^\vee \xrightarrow{\Id\wedge \text{ev}} A^\vee\wedge\bS\xrightarrow{\simeq} A^\vee.
\]
We will write $\SW(A) = A^\vee$ for the Spanier--Whitehead dual of $A$, when it exists, and $\SW(f) = f^\vee$ for the Spanier--Whitehead dual of a map between objects which have Spanier--Whitehead duals.

\begin{remark}\label{remark:useful-properties}
If $\text{coev} \colon \bS \to A \wedge B$ exhibits $A$ and $B$ as Spanier--Whitehead dual, then postcomposing it with the factor switching map $\mathrm{sw}\colon A\wedge B\to B\wedge A$ gives a coevaluation map exhibiting $B$ and~$A$ as Spanier--Whitehead dual. That is, being in duality is a symmetric property of $A$ and $B$.

We also note that $\bS$ is self-dual, and that if $A$ and $B$ are Spanier--Whitehead dual, then~$\Sigma^dA$ and $\Sigma^{-d}B$ are Spanier--Whitehead dual for all $d$.
\end{remark}

Assume $X$ is a $d$-dimensional topological manifold with possibly nonempty boundary.
 Recall that the \emph{stable normal microbundle} $\nu_X^{\Top}\colon X\to \BTop$ is the map classifying the virtual topological microbundle $[\nu^\Top_e - \varepsilon^N]$ of virtual dimension zero, where $\nu^\Top_e$ the normal microbundle of any locally flat embedding $e\colon (X, \partial X)\hookrightarrow (D^{d+N}, \partial D^{d+N})$ with $N \gg 0$, with $N$ large enough for normal microbundles to exist~\cite{Stern-normal-MBs-exist}.
There are Spanier--Whitehead dual pairs
\[
\SW(\Sigma^{\infty-d}(X/\partial X))=\Th(\nu^{\Top}_X)\qquad \text{and}\qquad \SW(\Sigma^{\infty-d}X_+)= \Th(\nu^{\Top}_X)/\Th(\nu^{\Top}_{\partial X});
\]
see \cite[V.2.3.(i), V.2.14.(a)]{Rudyak}. Recall that $X/\partial X = X_+$ if $\partial X$ is empty (this is consistent with the convention that $X/A$ refers to the push out of $\{*\} \leftarrow A \rightarrow X$) and in this case, the two Spanier--Whitehead dualities above agree.

We will describe the duality morphism of the latter. Let $\varepsilon\times\nu_{X}^{\Top}\colon X\times X\to \BTop$ denote the external direct sum of the stable trivial bundle over $X$ and $\nu_{X}^{\Top}$. By definition of the Whitney sum, there is a stable bundle map~$\nu_{X}^{\Top}\simeq\varepsilon\oplus\nu_{X}^{\Top}\to \varepsilon\times\nu_{X}^{\Top}$ covering $\text{diag}\colon X\to X\times X$. We may consider this bundle map to cover the diagonal map $(X,\partial X)\to X\times (X,\partial X)$. Passing to Thom spaces and noting that~$\Th(\varepsilon)=\Sigma^\infty X_+$, we obtain the \emph{Thom diagonal}
\[
\Delta\colon \Th(\nu_{X}^{\Top})/\Th(\nu^{\Top}_{\partial X})\lra \Sigma^\infty X_+\wedge \Th(\nu_{X}^{\Top})/\Th(\nu^{\Top}_{\partial X}).
\]
We then have a map
\[
\bS^d\xrightarrow{c_X}  \Th(\nu_{X}^{\Top})/\Th(\nu^{\Top}_{\partial X})\xrightarrow{\Delta}\Sigma^\infty X_+\wedge \Th(\nu_{X}^{\Top})/\Th(\nu^{\Top}_{\partial X}),
\]
where $c_X$ denotes the Thom collapse map. Desuspending this composition $d$ times gives the desired duality map~\cite[V.2.3.(i)]{Rudyak} (cf.~\cite[I.1.15]{Browder}).
The duality map for the other claimed Spanier--Whitehead duality is described similarly, but by considering the diagonal as the map of pairs $(X,\emptyset)\to (X,\partial X)\times X$.

So for each spectrum $E$ and each $r\in\Z$ there are group isomorphisms
\begin{multline*}
E_r(X/\partial X)=[\Sigma^r\bS,\Sigma^{\infty}(X/\partial X)\wedge E]=[\bS,\Sigma^{\infty-d}(X/\partial X)\wedge \Sigma^{d-r}E]\\
\cong [\Th(\nu^{\Top}_X), \Sigma^{d-r}E] = E^{d-r}(\Th(\nu^{\Top}_X)),
\end{multline*}
and
\begin{multline*}
E^r(X/\partial X)=[\Sigma^\infty(X/\partial X),\Sigma^rE]=[\Sigma^{\infty-d}(X/\partial X),\Sigma^{r-d}E]\\
\cong [\bS, \Th(\nu^{\Top}_X)\wedge\Sigma^{r-d}E] = E_{d-r}(\Th(\nu^{\Top}_X)).
\end{multline*}
These are referred to as \emph{Atiyah duality}~\cite{MR131880}. We will denote the map from the latter display, from cohomology to homology, by~$AD$.

We record the following, for later use.

\begin{lemma}\label{lem:SWcollapse}
If $X$ is a $d$-dimensional compact topological manifold then the Spanier--Whitehead dual of the Thom collapse map $\Sigma^{-d}c_X\colon \bS\to \Sigma^{-d}\big(\Th(\nu_{X}^{\Top})/\Th(\nu^{\Top}_{\partial X})\big)$ is the map $C_X\colon\Sigma^\infty X_+\to\bS$, induced by the constant map $X \to \pt$, adding basepoints and taking suspension spectra.
Moreover, under the $\Sigma$-$\Omega$ adjunction, the map $C_X$ corresponds to the constant map $X_+\to \Omega^\infty\bS=\mathrm{colim}_k\Omega^kS^k$ to the class of the identity map.
\end{lemma}

\begin{proof}
Write $\Delta':=\mathrm{sw}\circ\Delta$ for the Thom diagonal followed by the factor switch map; this will determine the coevaluation map needed for this proof (see~\cref{remark:useful-properties}). By definition of the duality, the
Spanier--Whitehead dual of the map $C_X \colon\Sigma^\infty X_+\to\bS$ is the composition
\begin{multline*}
\bS\xrightarrow{\simeq}\bS\wedge\bS\xrightarrow{\Sigma^{-d}c_X\wedge\Id}\Sigma^{-d}\big(\Th(\nu_{X}^{\Top})/\Th(\nu^{\Top}_{\partial X})\big)\wedge\bS\xrightarrow{\Delta'\wedge\Id} \Sigma^{-d}\big(\Th(\nu_{X}^{\Top})/\Th(\nu^{\Top}_{\partial X})\big)\wedge\Sigma^\infty X_+\wedge\bS
\\
\xrightarrow{\Id\wedge C_X\wedge\Id} \Sigma^{-d}\big(\Th(\nu_{X}^{\Top})/\Th(\nu^{\Top}_{\partial X})\big)\wedge\bS\wedge\bS\xrightarrow{\simeq} \Sigma^{-d}\big(\Th(\nu_{X}^{\Top})/\Th(\nu^{\Top}_{\partial X})\big).
\end{multline*}
Applying the $d$-fold suspension one obtains
\[
\bS^d \overset{c_X}\lra \Th(\nu_{X}^{\Top})/\Th(\nu^{\Top}_{\partial X})\xrightarrow{\Delta'} \Sigma^{\infty} X_+\wedge \Th(\nu_{X}^{\Top})/\Th(\nu^{\Top}_{\partial X}) \xrightarrow{\mathrm{Id}\wedge C_X} \Th(\nu_{X}^{\Top})/\Th(\nu^{\Top}_{\partial X})\wedge\bS.
\]
But by definition of the Thom diagonal the composition $(\mathrm{Id} \wedge C_X) \circ \Delta'$
 is the identity. Hence $\SW(\Sigma^{-d}c_X)=C_X$, as claimed.

The final assertion follows from adding basepoints to the map $X\to\pt$ then smashing with~$S^k$ to get a map~$S^k\wedge X_+\to S^k\wedge S^0$, which is the identity on the first factor. Taking adjoints results in a constant map to the identity element $X_+\to \Omega^k(S^k\wedge S^0)=\Omega^k(S^k)$. The result follows by taking the colimit.
\end{proof}

\subsection{Orientations and fundamental classes}\label{subsec:orientation}

If $E$ is a ring spectrum, and $X$ is a $d$-dimensional compact topological manifold, then an \emph{$E$-orientation} of, or \emph{$E$-theory fundamental class} for, $X$ is a class $[X]_E \in E_d(X/\partial X)$ such that for any coordinate chart $\mathbb{R}^d \subseteq X$ with corresponding collapse map
\[c\colon X/\partial X \to X/(X - \mathrm{Int}(D^d)) \cong D^d/\partial D^d \cong S^d\]
the class $c_*([X]_E) \in E_d(S^d)$ corresponds under the suspension isomorphism to $\pm 1 \in E_0(S^0)$~\cite[V.2.1]{Rudyak}.
Here $\pm 1 \in E_0(S^0)$ is the image of  $\pm 1 \in \pi_0^s(S^0)\cong \Z$ under the map $\pi_0^s(S^0) \to E_0(S^0)$ induced by the unit map $\eta \colon \bS \to E$ of the ring spectrum.
The notion of $E$-orientability of a manifold is closely connected to the definition of $E$-orientability of the stable normal microbundle $\nu^{\Top}_X\colon X\to \BTop$ (\cite[V.1.12]{Rudyak}). Indeed, under Atiyah duality, $E$-theory fundamental classes~$[X]_E \in E_d(X/\partial X)$ correspond to
$E$-theory Thom classes $U_E\in E^0(\Th(\nu^{\Top}_X))$ for $\nu^{\Top}_X$ (also known as $E$-orientations for $\nu^{\Top}_X$)~\cite[Corollary V.2.6]{Rudyak}. Given a module spectrum $F$ over $E$, and an $E$-theory fundamental class $[X]_E$ on the manifold $X$, there are induced $F$-theory Poincar\'{e} duality isomorphisms
\[
[X]_E\frown-\colon F^r(X/\partial X) \overset{\cong}\lra F_{d-r}(X_+),
\]
which coincide with Atiyah duality for $F$
 followed by the homological Thom isomorphism; see~\cite[V.1.3]{Rudyak}) and the second diagram of~\cite[Theorem~14.41]{MR1886843}. Note that in this normalisation the Thom class, and hence the Thom isomorphism, has degree zero.

\section{Proof of Theorem \ref{thm:main-guarantee}}\label{sec:proof-thm-C}

We will prove Theorem \ref{thm:main-guarantee} in the form of Theorem \ref{thm:main-technical} \ref{it:main-technicalC}, so will show that the image of the composition
\[
\pi_1(\pseudoHomeo^+_\partial(X)) \overset{\alpha}\lra Q(X) \overset{\KS}\lra H_2(X;\Z/2)
\]
contains the kernel $\ker(H_2(X;\Z/2)\xrightarrow{\delta_*}H_1(X;\Z_{(2)}))$ of the Bockstein operator for the coefficient sequence $0\to\Z_{(2)}\overset{2}\to\Z_{(2)}\to\Z/2\to 0$.
It suffices to do so after precomposing with the connecting map
\[
\partial\colon \pi_2\bigg(\frac{\sAut^+_\partial(X)}{\pseudoHomeo^+_\partial(X)}\bigg) \lra \pi_1\big(\pseudoHomeo^+_\partial(X)\big)
\]
in the long exact sequence of semi-simplicial homotopy groups (see e.g.~\cite[\textsection1.7]{MR2840650} for a definition) associated to the fibre sequence $\pseudoHomeo^+_\partial(X) \to \sAut^+_\partial(X)\to \frac{\sAut^+_\partial(X)}{\pseudoHomeo^+_\partial(X)}$ of Kan semi-simplicial sets. The domain of this connecting map has the advantage that it may be described via surgery theory, as we now explain.

\subsection{Geometric surgery}\label{sec:GeomSurgery}
 First we briefly recall the \emph{simple structure set}, and describe its relation to domain of $\partial$, above. Then we recall the surgery exact sequence of Browder--Novikov--Sullivan--Wall
 and show how it can be related to the maps $\alpha$, $\mathrm{ks}$, and $\Sm$.

\subsubsection{The structure set}
In the following we let $\CAT \in \{\Top, \Diff, \PL\}$ be a category of manifolds. Let $Y$ be a $d$-dimensional compact $\CAT$ manifold with (possibly empty) boundary~$\partial Y$.
The basic object of study in surgery theory is the simple $\CAT$-structure set $\mathcal{S}^{\CAT}_\partial(Y)$, described as follows.

\begin{definition}
Elements of $\mathcal{S}^{\CAT}_\partial(Y)$ are equivalence classes of maps of pairs $(f, \partial f)\colon (M, \partial M) \to (Y, \partial Y)$ from a $\CAT$ manifold such that $\partial f$ is a $\CAT$ isomorphism and $f$ is a
simple homotopy equivalence. The equivalence relation on such maps is $s$-cobordism: if there is an $s$-cobordism $W$ from $(M_0, \partial M_0)$ to $(M_1, \partial M_1)$, trivial on the boundary, and a map $(F, \partial_- F)\colon (W, \partial_- W) \to (Y, \partial Y)$ restricting to $(f_0, \partial f_0)\colon (M_0, \partial M_0) \to (Y, \partial Y)$ and $(f_1, \partial f_1)\colon (M_1, \partial M_1) \to (Y, \partial Y)$ at the ends, then $(f_0, \partial f_0) \sim (f_1, \partial f_1)$.
\end{definition}

Alternatively, if $d \geq 5$ (or $d=4$, $\CAT =\Top$, and $\pi_1(Y)$ is good) then the $\CAT$ $s$-cobordism theorem applies and we can rephrase the equivalence relation as $(f_0, \partial f_0) \sim (f_1, \partial f_1)$ if and only if there is a $\CAT$ isomorphism $(\phi, \partial \phi)\colon (M_0, \partial M_0) \to (M_1, \partial M_1)$ with $(f_1, \partial f_1) \,{\circ}\, (\phi, \partial \phi) \simeq (f_0, \partial f_0)$.

For each $k\geq 0$, we now describe a map
\[
\pi_k\bigg(\frac{\sAut^+_\partial(Y)}{\pseudoHomeo^+_\partial(Y)}\bigg) \lra \mathcal{S}^{\Top}_\partial(Y \times D^k),
\]
and show it is an injection for $k=0$ and a bijection for $k > 0$, as long as $k+d \geq 5$. The domain is a (semi)-simplicial homotopy group of a quotient of a semi-simplicial set by a semi-simplicial group. Unravelling definitions shows that elements are represented by simple homotopy equivalences
\[
f\colon Y \times \Delta^k \lra Y \times \Delta^k
\]
which preserve faces, which restrict to the identity on $(\partial Y) \times \Delta^k$, and which restrict to homeomorphisms on $Y \times \partial \Delta^k$ (and are well-defined up to a homotopy which is an isotopy on $Y \times \partial \Delta^k$ and fixed on $(\partial Y) \times \Delta^k$, and precomposing by block homeomorphisms of $Y \times \Delta^k$ which are the identity on $(\partial Y) \times \Delta^k$). Such maps define elements of $\mathcal{S}^{\Top}_\partial(Y \times D^k)$, and this assignment can be shown to be well-defined.  To see that it is surjective for $k>0$, we observe that an element of~$\mathcal{S}^{\Top}_\partial(Y \times D^k)$ is represented by a simple homotopy equivalence $\phi\colon (M, \partial M) \to (Y \times D^k, \partial(Y \times D^k))$ which is a homeomorphism on the boundary. Identifying $D^k = I^k$ we may view $M$ as an $h$-cobordism from $Y \times I^{k-1} \times \{0\}$ to $Y \times I^{k-1} \times \{1\}$ rel.~boundary, but as $M$ is simple homotopy equivalent to $Y \times I^k$ this is in fact an $s$-cobordism. By the $s$-cobordism theorem there is a homeomorphism $M \cong Y \times I^k$, and so the map $\phi$ may be considered as a simple homotopy equivalence~$\phi\colon (Y \times D^k, \partial(Y \times D^k)) \to (Y \times D^k, \partial(Y \times D^k))$ which is a homeomorphism on the boundary, but this is precisely our description of elements coming from $\pi_k(\sAut^+_\partial(Y)/\pseudoHomeo^+_\partial(Y))$. As usual, injectivity is proved by a relative form of this surjectivity argument.

\subsubsection{The surgery exact sequence}
The surgery theory of Browder--Novikov--Sullivan--Wall~\cite{Wall-surgery-book,CLM} describes the sets $\mathcal{S}^{\Top}_\partial(Y \times D^k)$ for $d+k \geq 5$ via the surgery exact sequence
\[
\begin{tikzcd}[row sep = tiny]
\cdots\ar[r] & L_{d+3}(\Z[\pi]) \ar[r] & \mathcal{S}^{\Top}_\partial(Y \times I \times I) \ar[draw=none]{d}[name=Y, anchor=center]{} \ar[r] & \left[\tfrac{Y \times I \times I}{\partial (Y \times I \times I)}, \G/\Top\right] \ar[r] & \phantom{X}
                      \\
\phantom{X} \ar[r] & L_{d+2}(\Z[\pi]) \ar[r] & \mathcal{S}^{\Top}_\partial(Y \times I) \ar[draw=none]{d}[name=X, anchor=center]{} \ar[r] & \left[\tfrac{Y \times I}{\partial (Y \times I)}, \G/\Top\right] \ar[r] & \phantom{X}
\\
\phantom{X} \ar[r]
& L_{d+1}(\Z[\pi]) \ar[r]
& \mathcal{S}^{\Top}_\partial(Y) \ar[r, "{\eta_{\Top}}"]
& \left[Y/\partial Y, \G/\Top\right] \ar[r, "{\sigma}"]
&L_{d}(\Z[\pi]),
\end{tikzcd}
\]
where $\pi := \pi_1(Y)$. Exactness is just as for the long exact sequence of homotopy groups for a fibration of pointed spaces: the last three terms are pointed sets, the next three are groups, and the rest are abelian groups, with the maps being homomorphisms to the extent possible (most important for us will be that the top row consists of abelian groups).
The $L$-groups $L_k(\Z[\pi])$ are Wall's (simple) surgery obstruction groups. The topological monoid $\G$ is the structure group for stable spherical fibrations, and the space $\G/\Top$ is defined as the homotopy fibre of the map~$\BTop\to \BG$ induced by the forgetful map.

 The stable normal microbundle $\nu_Y^{\Top}\colon Y\to \BTop$ is a preferred lift of the \emph{Spivak normal fibration}~$\nu_Y^{\G}\colon Y\to \BG$.
 The group $\left[Y/\partial Y, \G/\Top\right]$ acts freely transitively on the
 set of lifts of the Spivak normal fibration, relative to a fixed lift on the boundary $\partial Y$.

Given an element $[(f,\partial f)\colon (M,\partial M)\to (Y,\partial Y)]$ of the structure set $\mathcal{S}^{\Top}_\partial(Y)$, the homotopy equivalence $f$ determines a stable microbundle $(f^{-1})^*\nu^{\Top}_M\colon Y\to \BTop$. As $f$ is a homotopy equivalence, the underlying spherical fibration of $(f^{-1})^*\nu^{\Top}_M$ is canonically homotopic to the Spivak normal fibration $Y\to \BG$. As $f$ is a homotopy equivalence, the underlying spherical fibrations of~$(f^{-1})^*\nu^{\Top}_M$ and $\nu_Y^\Top$ are (canonically) identified. As $\partial f$ is a homeomorphism this identification extends to one of stable microbundles over $\partial Y$. This gives a \emph{difference class}
$$\eta_{\Top}([(f,\partial f)]):=d((f^{-1})^*\nu^{\Top}_M,\nu^{\Top}_Y)\in \left[Y/\partial Y, \G/\Top\right].$$
See~\cite{Wall-surgery-book,CLM} for further details on the surgery exact sequence.

\subsubsection{Mapping the surgery exact sequence in}
For a 4-manifold $X$ we may form the diagram
\begin{equation}\label{eqn:diagram-mapping-surgery-seq-in}
\begin{tikzcd}[column sep=1em]
\pi_2\Big(\frac{\sAut^+_\partial(X)}{\pseudoHomeo^+_\partial(X)}\Big) \arrow[r, "\cong"] \dar{\partial} & \mathcal{S}^{\Top}_\partial(X \times I \times I) \rar{\eta_\Top} & {\left[\tfrac{X \times I \times I}{\partial (X \times I \times I)}, \G/\Top\right]} \dar
\\
\pi_1(\pseudoHomeo^+_\partial(X)) \rar{\alpha} & Q(X) \rar{\Sm} & {\left[\tfrac{X \times I \times I}{\partial (X \times I \times I)}, \BB(\Top/\OO)\right]} \arrow[r, "\ks", "\cong"'] & \widetilde{H}^4\Big(\tfrac{X \times I \times I}{\partial (X \times I \times I)};\Z/2\Big).
\end{tikzcd}
\end{equation}
The right-hand vertical map is induced by the composition $\G/\Top \to \BTop \to \BB(\Top/\OO)$. The left-hand vertical map is the connecting homomorphism from the long exact sequence of the fibration.

\begin{lemma}\label{lem:normalcommute}
The diagram \eqref{eqn:diagram-mapping-surgery-seq-in} commutes.
\end{lemma}

\begin{proof}
Let $\phi\colon (M, \partial M) \to (X \times I \times I, \partial (X \times I \times I))$ represent an element of $\mathcal{S}^{\Top}_\partial(X \times I \times I)$, so it is a simple homotopy equivalence which is a homeomorphism on the boundary. In particular,~$M$ can be considered as an $h$-cobordism from $X \times I \times \{0\}$ to $X \times I \times \{1\}$ rel.~boundary. By the composition formula its Whitehead torsion satisfies
$$f_*\tau(M, X \times I \times \{0\}) + \tau(\phi) = \tau(X \times I \times I, X \times I \times \{0\}) = 0 \in \mathrm{Wh}(\pi_1(X \times I \times I)),$$
and as $\phi$ is a simple homotopy equivalence it follows that $\tau(M, X \times I \times \{0\})=0$, and thus $M$ is an $s$-cobordism. The $s$-cobordism theorem then gives a homeomorphism
$$\psi\colon M\overset{\cong}\lra X \times I \times I \quad \text{ rel.~} \quad X \times I \times \{0\} \cup \partial(X \times I) \times I,$$ which restricts to a homeomorphism $F\colon X \times I \times \{1\} \overset{\cong}\to X \times I \times \{1\}$ rel.~boundary. Spelling out the isomorphism and definition of $\partial$ in the diagram shows that
\[
\partial([\phi]) = [F]\in \pi_0(\pseudoHomeo^+_\partial(X\times I))=\pi_1(\pseudoHomeo^+_\partial(X)).
\]
Now $\alpha([F]) = [\id, F]$, which is just $F$ considered as a topological pseudo-isotopy from the identity to the identity diffeomorphism, so $\ks \,{\circ}\, \Sm \,{\circ}\, \alpha([F]) \in \left[\tfrac{X \times I \times I}{\partial (X \times I \times I)}, K(\Z/2,4)\right]$
 is given by
\begin{equation*}
\begin{tikzcd}
\partial (X \times I \times I)_F  \ar[rrr, "T\partial (X \times I \times I)_F \oplus \varepsilon^1"] \ar[d] && & \BO(6) \ar[d] \ar[r] & \BO \dar \rar & \mathrm{E}(\Top/\OO) \dar \rar & P K(\Z/2, 4) \simeq * \dar\\
X \times I \times I  \ar[rrr, "\tau_{X \times I \times I}"]  &&& \BTop(6) \rar & \BTop \rar{\theta} & \BB(\Top/\OO) \rar{\ks} & K(\Z/2,4).
\end{tikzcd}
\end{equation*}
Now $\psi$ induces an isomorphism of topological manifolds with smoothed boundary
$$\psi\colon (M, \partial M) \overset{\cong}\lra (X \times I \times I, \partial(X \times I \times I)_F)$$
so we have
$$\ks \,{\circ}\, \Sm \,{\circ}\, \alpha([F]) = \ks(X \times I \times I, \partial(X \times I \times I)_F) = (\psi^{-1})^* \ks(M, \partial M) \in \widetilde{H}^4\left(\tfrac{X \times I \times I}{\partial(X \times I \times I)} ; \Z/2\right).$$
The composition $\phi \circ \psi^{-1} \colon (X \times I \times I, \partial(X \times I \times I)) \to (X \times I \times I, \partial(X \times I \times I))$ is the identity on $\quad X \times I \times \{0\} \cup \partial(X \times I) \times I$ and sends $X \times I \times \{1\}$ into itself, so by Lemma \ref{lem:excrosseye} it is homotopic as a map of pairs to the identity. It follows that $\psi^{-1}$ is homotopic as  map of pairs to $\phi^{-1}$, so the right-hand side of the previous equation
may be written as $(\phi^{-1})^* \ks(M, \partial M)$.

On the other hand, by definition of the normal invariant map $\eta_\Top$ the composition
$$\mathcal{S}^{\Top}_\partial(X \times I \times I) \overset{\eta_\Top}\lra {\left[\tfrac{X \times I \times I}{\partial (X \times I \times I)}, \G/\Top\right]} \lra {\left[\tfrac{X \times I \times I}{\partial (X \times I \times I)}, \BTop \right]}$$
sends the simple homotopy equivalence $[\phi]$ to the map classifying the stable microbundle $(\phi^{-1})^*\nu^{\Top}_M - \nu^{\Top}_{X \times I \times I}$
together with the trivialisation over $\partial (X \times I \times I)$ given by the fact that $\phi$ is a homeomorphism on $\partial(X \times I \times I)$.
Post-compose with $\BTop \overset{\theta}\to \BB(\Top/\OO) \overset{\ks}\to K(\Z/2,4)$.
Since the Kirby--Siebenmann class is additive for Whitney sum of bundles \cite[Annex 3 Lemma 15.5]{Kirby-Siebenmann:1977-1},~$\nu^{\Top}_Y$ and the stable tangent microbundle $\tau^{\Top}_Y$ have the same Kirby--Siebenmann class.
We therefore obtain
\begin{align*}
\ks((\phi^{-1})^*\nu^{\Top}_M - \nu^{\Top}_{X \times I \times I}) &=  (\phi^{-1})^*\ks(\nu^{\Top}_M) - \ks(\nu^{\Top}_{X \times I \times I}) = (\phi^{-1})^*\ks(\tau^{\Top}_M) - \ks(\tau^{\Top}_{X \times I \times I}) \\
&= (\phi^{-1})^*\ks(M,\partial M) - \ks(X \times I \times I,\partial) \\ &=
(\phi^{-1})^* \ks(M, \partial M) \in \widetilde{H}^4\Big(\tfrac{X \times I \times I}{\partial(X \times I \times I)} ; \Z/2\Big),
  \end{align*}
where we have again used additivity, that $\nu^{\Top}_Y$ and $\tau^{\Top}_Y$ are stable inverses, the definition of $\ks(Y,\partial Y)$, and that  $X \times I \times I$ is smooth so has trivial Kirby--Siebenmann class.
This completes the verification that both passages around the diagram yield the same result.
\end{proof}

We give an outline of our strategy for the proof of Theorem \ref{thm:main-technical} \ref{it:main-technicalC}.
Recall that our goal is to show that
the image of $PD\circ \ks \,{\circ}\, \,\Sm \,{\circ}\, \alpha \,{\circ}\, \partial$ contains $\ker \big(\delta_* \colon H_2(X;\Z/2) \to H_1(X ; \Z_{(2)})\big)$.
As the codomain of the right vertical map in diagram~\eqref{eqn:diagram-mapping-surgery-seq-in} is a vector space over $\Z/2$, this map factors through the $2$-localisation of its domain.
This allows us to augment diagram~\eqref{eqn:diagram-mapping-surgery-seq-in} as follows
\begin{equation}\label{eq:diagram-pf-strategy-for-23C}
\begin{tikzcd}[column sep=1em]
\pi_2\Big(\frac{\sAut^+_\partial(X)}{\pseudoHomeo^+_\partial(X)}\Big) \arrow[r, "\cong"] \ar[dd, "\partial"] & \mathcal{S}^{\Top}_\partial(X \times I \times I) \ar[d]\ar[r, "\eta_\Top"] & {\left[\tfrac{X \times I \times I}{\partial (X \times I \times I)}, \G/\Top\right]} \ar[d] &
\\
& \mathcal{S}^{\Top}_\partial(X \times I \times I)_{(2)} \ar[r, "\eta_\Top"] & {\left[\tfrac{X \times I \times I}{\partial (X \times I \times I)}, \G/\Top\right]_{(2)}} \dar \ar[r, "\sigma"]& L_{6}(\Z[\pi])_{(2)}
\\
\pi_1(\pseudoHomeo^+_\partial(X)) \rar{\alpha} & Q(X) \rar{\Sm} & {\left[\tfrac{X \times I \times I}{\partial (X \times I \times I)}, \BB(\Top/\OO)\right]} \ar[d, "\ks","\cong"']&
\\
&& \widetilde{H}^4\left(\tfrac{X \times I \times I}{\partial (X \times I \times I)};\Z/2\right) \arrow[r,"\cong"', "PD"] & H_2(X;\Z/2).
\end{tikzcd}
\end{equation}
The $2$-localised row in diagram~\eqref{eq:diagram-pf-strategy-for-23C} is meaningful because the surgery exact sequence for $X \times I \times I$ is an exact sequence of abelian groups, and can therefore be localised at 2, which furthermore preserves exactness.
Diagram~\eqref{eqn:diagram-mapping-surgery-seq-in}, with the right vertical arrow factored through the $2$-localisation, is a subdiagram of diagram~\eqref{eq:diagram-pf-strategy-for-23C}, and so this big square commutes by Lemma~\ref{lem:normalcommute}. The smaller square in diagram~\eqref{eq:diagram-pf-strategy-for-23C} commutes, by naturality of localisation. So diagram~\eqref{eq:diagram-pf-strategy-for-23C} commutes overall.

We will produce, for each $u \in \ker \big(\delta_* \colon H_2(X;\Z/2) \to H_1(X ; \Z_{(2)})\big)$, a lift to an element $\xi_u \in {\big[\tfrac{X \times I \times I}{\partial (X \times I \times I)}, \G/\Top\big]_{(2)}}$  such that $\sigma(\xi_u)=0$. By exactness we will then obtain an element in~$\mathcal{S}^{\Top}_\partial(X \times I \times I)_{(2)}$ mapping to $u$. After multiplying by an odd integer, this lifts to an element~$\mathcal{S}^{\Top}_\partial(X \times I \times I)$, also mapping to $u$. It follows that $u$ is in the image of the clockwise composition starting in the top left corner. By commutativity of the big square, we can conclude $u$ lies in the image of~$PD\circ\ks \,{\circ}\, \,\Sm \,{\circ}\, \alpha \,{\circ}\, \partial$. This will complete the proof of Theorem \ref{thm:main-technical}~\ref{it:main-technicalC}.

\subsection{Ranicki--Sullivan duality and Poincar\'{e} duality}\label{sec:algsurgery}
To do what we have just described, we will use Ranicki's algebraic surgery exact sequence; see~\cite{MR1211640}. In the topological category, this is an exact sequence isomorphic to the geometric surgery exact sequence, and expressed in terms of the \emph{$\mathbb{L}$-spectra}. In the algebraic surgery exact sequence, the surgery obstruction map~$\sigma$ from the geometric surgery exact sequence factors through the \emph{Ranicki--Sullivan duality} map, followed by the \emph{algebraic assembly} map, as we shall describe below. The main task of this section is to relate Ranicki--Sullivan duality back to ordinary Poincar\'{e} duality.

\subsubsection{$\mathbb{L}$-spectra}
Given a ring $R$ with anti-involution (often called an involution in the surgery literature),
 there are defined spectra $\mathbb{L}^s(R)$ and $\mathbb{L}^q(R)$, respectively the \emph{symmetric} and \emph{quadratic}~$L$-spectra. The reader is referred to~\cite{MR1211640} for the construction, but note we are using the naming convention of~\cite{lurienotes}, whereas Ranicki writes~$\mathbb{L}^\bullet(R)$ and~$\mathbb{L}_\bullet(R)$ respectively for the symmetric and quadratic spectra. Products in $L$-theory (see~\cite[Appendix~B]{MR1211640}, \cite{Ranicki-ATS-I})
 endow~$\mathbb{L}^s(R)$ with the structure of a ring spectrum and $\mathbb{L}^q(R)$ with the structure of an $\mathbb{L}^s(R)$-module spectrum. We will denote by $\mathbb{L}^s = \tau_{\geq 0}\mathbb{L}^s(\Z)$, the 0-connective cover of $\mathbb{L}^s(\Z)$, by $\mathbb{L}^q = \tau_{\geq 0}\mathbb{L}^q(\Z)$ the 0-connective cover of $\mathbb{L}^q(\Z)$, and by $\mathbb{L}^q\langle 1 \rangle = \tau_{\geq 1}\mathbb{L}^q(\Z)$ the 1-connective cover of $\mathbb{L}^q(\Z)$; $\mathbb{L}^s$ is again a ring spectrum and $\mathbb{L}^q$ and $\mathbb{L}^q\langle 1 \rangle$ are $\mathbb{L}^s$-module spectra.
We will continue with the notation~$E_r(-)$ and~$E^r(-)$ to refer to the corresponding (co)homology theories for $E=\mathbb{L}^s, \mathbb{L}^q, \mathbb{L}^q\langle 1\rangle$, where Ranicki uses the notation $H_r(-;E)$ and $H^r(-;E)$.

\subsubsection{The symmetric $L$-theory fundamental class}
There is a morphism $\sigma\colon \MSTop\to \mathbb{L}^s$ of ring spectra, referred to as the \emph{Ranicki orientation}; see~\cite[\textsection16]{MR1211640}. For an oriented topological $d$-manifold $Y$, write $\nu^{\Top}_Y \colon Y\to \BSTop$ for the oriented stable normal microbundle of the pair $(Y,\partial Y)$. We have an the induced map of spectra $\Th(\nu^{\Top}_Y)\to\MSTop$, and
the composition $\Th(\nu^{\Top}_Y) \to \MSTop  \xrightarrow{\sigma} \mathbb{L}^s$ results in a
$\mathbb{L}^s$-theory Thom class $U_{\mathbb{L}^s}\in
(\mathbb{L}^s)^0(\Th(\nu^{\Top}_Y))
$; see~\cite[\textsection16]{MR1211640}. Under Atiyah duality
$(\mathbb{L}^s)^0(\Th(\nu^{\Top}_Y))\cong (\mathbb{L}^s)_d(Y/\partial Y)$
we therefore obtain an~$\mathbb{L}^s$-theory fundamental class~$[Y]_{\mathbb{L}^s}\in (\mathbb{L}^s)_d(Y/\partial Y)$.

\subsubsection{The surgery obstruction and assembly}\label{subsection:surgery-obstr-and-assembly}
In~\cite{MR561227}, Ranicki constructed a homotopy equivalence $\G/\Top\simeq \Omega^\infty \mathbb{L}^q\langle 1 \rangle$, which we use to identify these two spaces from now on. In particular, if $Y$ is a topological $d$-manifold as above then we have an identification
\begin{equation}\label{eqn:defn-of-D}
\begin{tikzcd}[column sep=3.5ex]
\mathrm{D}\colon \left[Y/\partial Y, \G/\Top\right] \cong \left[Y/\partial Y, \Omega^\infty \mathbb{L}^q\langle 1 \rangle\right] \cong \left[\Sigma^\infty (Y/\partial Y), \mathbb{L}^q\langle 1 \rangle\right] \ar[rr, "{[Y]_{\mathbb{L}^s} \frown -}","{\cong}"']
&& \left[\bS^d, \Sigma^\infty Y_+ \wedge \mathbb{L}^q\langle 1 \rangle\right],
\end{tikzcd}
\end{equation}
given by Poincar{\'e} duality in $\mathbb{L}^q\langle 1 \rangle$-theory (this is the duality often referred to as \emph{Ranicki--Sullivan duality}).
Note that the first two terms refer to maps of spaces, while the last two refer to maps of spectra.
Under this identification, the surgery obstruction map from~\cite[\textsection3]{Wall-surgery-book} factors as
\begin{equation*}
\begin{tikzcd}
{\left[Y/\partial Y, \G/\Top\right]} \arrow[rr, "\sigma"] \arrow[d, "\mathrm{D}","\cong"'] & & L_d(\Z[\pi])\\
{\left[\bS^d, \Sigma^\infty Y_+ \wedge \mathbb{L}^q\langle 1 \rangle\right]} \rar & {\left[\bS^d, \Sigma^\infty Y_+ \wedge \mathbb{L}^q\right]} \rar{\mft_*} & {\left[\bS^d, \Sigma^\infty \BB\pi_+ \wedge \mathbb{L}^q\right]} \arrow[u, "A_\pi"]
\end{tikzcd}
\end{equation*}
where the first lower map is induced by the canonical map $\mathbb{L}^q\langle 1 \rangle \to \mathbb{L}^q$ given by the connective cover, the second is induced by the Postnikov truncation map $\mft\colon Y \to \BB\pi$, and the third map~$A_\pi$ is the assembly map in quadratic $L$-theory, defined in~\cite{MR561227}.
This is shown in \cite[Appendix~2, Theorem~2]{HMTW-surgery-finite}, where it is credited to Quinn~\cite{MR282375} and Ranicki~\cite{MR561227,MR1211640}, with a  contribution by Nicas~\cite{Nicas}.

\medskip

Our objective for the remainder of this subsection is prove Lemma~\ref{lem:TheTwoIsomorphisms}, below, where we relate Ranicki--Sullivan duality back to ordinary Poincar\'{e} duality.

\subsubsection{Splitting $L$-theory spectra}

The 2-localised spectra $\mathbb{L}^s_{(2)}$ and $\mathbb{L}^q\langle 1 \rangle_{(2)}$ are generalised Eilenberg--Maclane spectra; see~\cite[\textsection2]{MR557167}.
Moreover, Taylor and Williams~\cite{MR557167} write down \emph{specific} equivalences of $\mathrm{H}\Z$-module spectra
\begin{equation}\label{eq:symmetric}
\begin{pmatrix} l^s \\ r\end{pmatrix}\colon \mathbb{L}^s_{(2)} \xrightarrow{\simeq} \bigoplus _{i\geq0}\mathrm{H}\Z_{(2)}[4i]\,\oplus \bigoplus _{j\geq0}\mathrm{H}\Z/2[4j+1],
\end{equation}
and
\begin{equation}\label{eq:quadratic}
\begin{pmatrix} l^q \\ k\end{pmatrix}\colon \mathbb{L}^q\langle 1 \rangle_{(2)} \xrightarrow{\simeq} \bigoplus _{i\geq 1}\mathrm{H}\Z_{(2)}[4i]\,\oplus \bigoplus _{j\geq0}\mathrm{H}\Z/2[4j+2].
\end{equation}

\begin{lemma}\label{lem:KSformula}
Under the identification $\G/\Top\simeq \Omega^\infty \mathbb{L}^q\langle 1 \rangle$ and with respect to the Taylor--Williams splitting~\eqref{eq:quadratic}, the cohomology class $\G/\Top \to \BTop \to \BB(\Top/\OO) \overset{\ks}\to K(\Z/2,4)$ is given by $\mathrm{red}_2(l^q_4) + k_2 \smile k_2$.  That is, the following diagram commutes
\[
\begin{tikzcd}[column sep = huge]
  \G/\Top \ar[r] \ar[d] & \Omega^\infty \mathbb{L}^q\langle 1 \rangle_{(2)} \ar[d,"l^q_4 \times k_2","\eqref{eq:quadratic}"'] \\
  K(\Z/2,4) &  K(\Z_{(2)},4) \times K(\Z/2,2). \ar[l,"\mathrm{red}_2(-) + - \smile -"]
\end{tikzcd}
\]
\end{lemma}

This is surely known in some form, but we were unable to find a proof.
We give one here that manages to avoid the specifics of the definitions of $l^q_4$ or $k_2$.

\begin{proof}
Recall that $\Top/\PL \simeq K(\Z/2,3)$, and that the composition in the statement of the lemma is the map $u$ in the fibre sequence $\G/\PL \to \G/\Top \overset{u}\to \BB(\Top/\PL) \simeq K(\Z/2,4)$. With respect to the decomposition
\[(\G/\Top)_{(2)} \simeq \Omega^\infty \mathbb{L}^q\langle 1 \rangle_{(2)} \simeq \prod_{i \geq 1} K(\Z_{(2)}, 4i) \times \prod_{j \geq 0} K(\Z/2, 4j+2)\]
induced by the maps $l^q_{4i}$ and $k_{4j+2}$ from \eqref{eq:quadratic}, we have
\[u \in H^4(\G/\Top ; \Z/2) \cong \Z/2\{\mathrm{red}_2(l^q_4), k_2 \smile k_2\}.\]
To see this isomorphism, first note that since $\G/\Top \to \G/\Top_{(2)}$ induces an isomorphism on homology with $\Z_{(2)}$-coefficients (by definition of localisation), the same map induces an isomorphism on (co)homology with any $\Z_{(2)}$-module coefficients by universal coefficients.  Since $\Z/2$ is a~$\Z_{(2)}$-module, we have an isomorphism
$H^4(\G/\Top;\Z/2) \cong H^4(\G/\Top_{(2)};\Z/2)$.
In addition, the map $\G/\Top_{(2)} \to K(\Z_{(2)},4) \times K(\Z/2,2)$ is 6-connected, so the latter group $H^4(\G/\Top_{(2)};\Z/2)$ is isomorphic to
\[H^4(K(\Z_{(2)},4) \times K(\Z/2,2);\Z/2) \cong \Z/2\{\mathrm{red}_2(l^q_4),k_2 \smile k_2\},\]
as claimed.

It follows that $u = A\cdot \mathrm{red}_2(l^q_4) + B \cdot k_2 \smile k_2$ for some $A, B \in \Z/2$.
In these terms, the 5-truncation of the 2-localised $\G/\PL$ is described as a homotopy pullback
\begin{equation}\label{eq:GPLpullback}
\begin{tikzcd}
\tau_{\leq 5}(\G/\PL)_{(2)} \dar{k_2} \rar{l^q_4} & K(\Z_{(2)}, 4) \dar{A\cdot \mathrm{red}_2}\\
K(\Z/2, 2) \rar{B \cdot \mathrm{Sq^2}} & K(\Z/2, 4).
\end{tikzcd}
\end{equation}
Now Sullivan has shown (see e.g.\ \cite[Theorem 4.8]{MadsenMilgram}) that $\tau_{\leq 5} (\G/\PL)_{(2)}$ is homotopy equivalent to the homotopy fibre of $\delta_* \,{\circ}\,\mathrm{Sq}^2\colon K(\Z/2, 2) \to K(\Z_{(2)},5)$, where $\delta_*\colon K(\Z/2, 4) \to K(\Z_{(2)},5)$ denotes the map classifying the cohomology Bockstein. From this we first see that $A=1$, otherwise the pullback \eqref{eq:GPLpullback} would give the wrong~$\pi_3$. Given that $A=1$, the diagram \eqref{eq:GPLpullback} expresses~$\tau_{\leq 5}(\G/\PL)_{(2)}$ as the homotopy fibre of $B \cdot \delta_* \,{\circ}\,\mathrm{Sq}^2$, because we can extend the diagram downwards as follows, such that the right-hand column is a fibration sequence corresponding to the Bockstein long exact sequence.
\[\begin{tikzcd}
\tau_{\leq 5}(\G/\PL)_{(2)} \dar{k_2} \rar{l^q_4} & K(\Z_{(2)}, 4) \dar{\mathrm{red}_2}\\
K(\Z/2, 2) \ar[d,"B\cdot \delta_* \circ \Sq^2"] \rar{B \cdot \mathrm{Sq^2}} & K(\Z/2, 4) \ar[d,"\delta_*"]\\
K(\Z_{(2)},5) \ar[r,"="] & K(\Z_{(2)},5)
\end{tikzcd}\]
If $B=0$ then we would have $\tau_{\leq 5}(\G/\PL)_{(2)} \simeq K(\Z/2,2) \times K(\Z_{(2)},4)$.  If this were true then the fibration sequence $K(\Z_{(2)},4) \to \tau_{\leq 5}(\G/\PL)_{(2)} \xrightarrow{f} K(\Z/2,2) \xrightarrow{\delta_* \,{\circ}\,\mathrm{Sq}^2} K(\Z_{(2)},5)$ arising from Sullivan's theorem would admit a splitting $i_1 \colon K(\Z/2,2) \to \tau_{\leq 5}(\G/\PL)_{(2)}$ with $f \circ i_1 \simeq \Id$,
whence we have homotopies of maps \[\ast = \ast \circ i_1 \simeq (\delta_* \,{\circ}\,\mathrm{Sq}^2) \circ f \circ i_1 \simeq (\delta_* \,{\circ}\,\mathrm{Sq}^2) \circ \Id = \delta_* \,{\circ}\,\mathrm{Sq}^2 \colon K(\Z/2,2) \to  K(\Z_{(2)},5).\]
Here
$\ast$ denotes a constant map.  This contradicts the nontriviality of the cohomology operation~$\delta_* \,{\circ}\,\mathrm{Sq}^2$.
Thus we must have $B=1$.
\end{proof}

\begin{remark}
An alternative proof that $A=B=1$ could be made that proceeds by evaluating on degree one normal maps $E_8 \to S^4$ and $*\CP^2 \to \CP^2$,
where $E_8$ denotes the unique closed, simply connected $4$-manifold with signature $8$~\cite{F}*{Theorem~1.7} and $\ast\CP^2$ denotes the unique closed $4$-manifold that is homotopy equivalent to $\CP^2$ but not homeomorphic to it~\cite{F}*{p.~370}.
\end{remark}

Recall that for a space $X$ and a double loop space or spectrum $Y$,
 localising induces an isomorphism of abelian groups  $[X,Y_{(2)}] = [X,Y]_{(2)}$.
We will use this without further comment.

Via $\G/\Top\simeq \Omega^\infty \mathbb{L}^q\langle 1 \rangle$, the Taylor--Williams splitting~\eqref{eq:quadratic} of $\mathbb{L}^q\langle 1 \rangle_{(2)}$ gives an identification
\begin{equation}\label{TW-identification-one}
\left[Y/\partial Y, \G/\Top\right]_{(2)} \! \cong \! \left[\Sigma^\infty (Y/\partial Y), \mathbb{L}^q\langle 1 \rangle_{(2)}\right] \! \underset{TW}{\cong} \! \bigoplus_{i\geq 1} H^{4i}(Y, \partial Y ;\Z_{(2)}) \oplus \bigoplus _{j\geq0} H^{4j+2}(Y, \partial Y ; \Z/2),
\end{equation}
and using (ordinary) Poincar{\'e} duality, with respect to the given $\Z$-orientation of $Y$, we can identify the latter with
\begin{equation}\label{eqn:RHS-of-two-identifications}
  \bigoplus_{i\geq 1} H_{d-4i}(Y ;\Z_{(2)})\,\oplus\, \bigoplus_{j\geq0} H_{d-4j-2}(Y; \Z/2).
\end{equation}
On the other hand, the Taylor--Williams splitting~\eqref{eq:quadratic} of $\mathbb{L}^q\langle 1 \rangle_{(2)}$ also gives an identification
\begin{equation}\label{TW-identification-two}
\left[Y/\partial Y, \G/\Top\right]_{(2)} \overset{\mathrm{D}}{\underset{\cong}\lra} \left[\bS^d, \Sigma^\infty Y_+ \wedge \mathbb{L}^q\langle 1 \rangle_{(2)} \right] \underset{TW}{\cong} \bigoplus_{i\geq 1} H_{d-4i}(Y ;\Z_{(2)}) \oplus \bigoplus_{j\geq0} H_{d-4j-2}(Y; \Z/2).
\end{equation}
\emph{The two identifications of  $\left[Y/\partial Y, \G/\Top\right]_{(2)}$ with \eqref{eqn:RHS-of-two-identifications} are not the same.}
Their difference is measured by the homology classes which correspond to $[Y]_{\mathbb{L}^s}$ under the Taylor--Williams splitting~\eqref{eq:symmetric} of $\mathbb{L}^s_{(2)}$, because \eqref{TW-identification-two} used capping with $[Y]_{\mathbb{L}^s}$, whereas passing from \eqref{TW-identification-one} to \eqref{eqn:RHS-of-two-identifications} uses capping with the ordinary homology fundamental class $[Y] = [Y]_{\mathrm{H}\Z}$. These are described in the next lemma.
For its statement, we let
\[\ell=1+\ell_1+\ell_2+\cdots \in \bigoplus_{i\geq0} H^{4i}(\BSTop;\Z_{(2)})\] denote the Morgan--Sullivan class~\cite{MR350748}, with $\ell_i \in H^{4i}(\BSTop;\Z_{(2)})$, and we write
\[\mathbb{V} := \sum_{j\geq 0}V_{2j}\smile\Sq^1(V_{2j})\in \bigoplus_{j\geq0} H^{4j+1}(\BSTop;\Z/2),
\]
where $V_k\in H^k(\BSTop;\Z/2)$ denotes the $k^{\text{th}}$ Wu class.

\begin{lemma}\label{lem:LThyFundamentalClassUnderSplitting}
Under the Taylor--Williams identification~\eqref{eq:symmetric}, the 2-local fundamental class $[Y]_{\mathbb{L}^s_{(2)}}\in (\mathbb{L}^s_{(2)})_d(Y/\partial Y)$ is given by
\[\big([Y]\frown ((\nu^{\Top}_Y)^*\ell), [Y]\frown ((\nu^{\Top}_Y)^*\mathbb{V})\big)\\
\in \bigoplus_{i\geq0} H_{d-4i}(Y,\partial Y;\Z_{(2)})\oplus \bigoplus_{j\geq0} H_{d-(4j+1)}(Y,\partial Y;\Z/2).\]
\end{lemma}

\begin{proof}
By \cite[(1.9)]{MR557167} the pullbacks under the Ranicki orientation
\[
\sigma^*l^s\in\bigoplus_{i\geq0} H^{4i}(\MSTop;\Z_{(2)})\qquad \text{and}\qquad \sigma^*r\in\bigoplus_{j\geq0} H^{4j+1}(\MSTop;\Z/2)
\]
are the images under the Thom isomorphism of the classes $\ell$ and $\mathbb{V}$ respectively. Pulling these back along $\Th (\nu_Y^{\Top}) \to \MSTop$ we obtain the identities
\begin{align*}
l^s_*U_{\mathbb{L}_{(2)}^s} &= ((\nu_Y^\Top)^*\ell) \smile U_{\mathrm{H}\Z_{(2)}} \in \bigoplus_{i\geq0} H^{4i}(\Th (\nu_Y^{\Top});\Z_{(2)})\\
r_*U_{\mathbb{L}_{(2)}^s} &= ((\nu_Y^\Top)^*\mathbb{V}) \smile U_{\mathrm{H}\Z_{(2)}} \in \bigoplus_{j\geq0} H^{4j+1}(\Th (\nu_Y^{\Top});\Z/2).
\end{align*}
Applying Atiyah duality ($AD$), which may be written as the composition of the inverse of the Thom isomorphism followed by Poincar{\'e} duality, we obtain the identities
\begin{align*}
AD(l^s_*U_{\mathbb{L}_{(2)}^s}) &= [Y] \frown ((\nu_Y^\Top)^*\ell)  \in \bigoplus_{i\geq0} H_{d-4i}(Y,\partial Y;\Z_{(2)})\\
AD(r_*U_{\mathbb{L}_{(2)}^s}) &= [Y] \frown ((\nu_Y^\Top)^*\mathbb{V})  \in \bigoplus_{j\geq0} H_{d-(4j+1)}(Y,\partial Y;\Z/2).
\end{align*}
By definition of the fundamental class $[Y]_{\mathbb{L}^s_{(2)}}\in (\mathbb{L}^s_{(2)})_d(Y/\partial Y)$ as the image of the generalised Thom class $U_{\mathbb{L}_{(2)}^s}\in (\mathbb{L}^s_{(2)})^0(\Th(\nu^{\Top}_Y))$ under Atiyah duality, we have identities $AD(l^s_*U_{\mathbb{L}_{(2)}^s}) = l^s_*AD(U_{\mathbb{L}_{(2)}^s}) = l^s_* [Y]_{\mathbb{L}_{(2)}^s} \in \bigoplus_{i\geq0} H_{d-4i}(Y,\partial Y;\Z_{(2)})$. Similarly, for the other case we have identities $AD(r_*U_{\mathbb{L}_{(2)}^s}) = r_*AD(U_{\mathbb{L}_{(2)}^s}) = r_* [Y]_{\mathbb{L}_{(2)}^s} \in \bigoplus_{j\geq0} H_{d-(4j+1)}(Y,\partial Y;\Z/2)$. Combining  these with the previous display yields the statement of the lemma.
\end{proof}

The two identifications of  $\left[Y/\partial Y, \G/\Top\right]_{(2)}$ with \eqref{eqn:RHS-of-two-identifications} discussed above are related as follows.

\begin{lemma}\label{lem:TheTwoIsomorphisms}
The diagram
\[
\begin{tikzcd}[cramped, column sep = tiny]
(\mathbb{L}^q\langle 1 \rangle_{(2)})^0(Y/\partial Y) \ar[d,"="] & \left[Y/\partial Y, \G/\Top\right]_{(2)} \ar[l,"\cong"'] \ar[r,"\mathrm{D}","\cong"'] & (\mathbb{L}^q\langle 1 \rangle_{(2)})_d(Y) \ar[d,"="]  \\
{\left[\Sigma^\infty (Y/\partial Y), \mathbb{L}^q\langle 1 \rangle_{(2)}\right]} \arrow[d, "TW \eqref{TW-identification-one}", "\cong"']\ar[rr,"{[Y]_{\mathbb{L}^s_{(2)}} \frown -}","\cong"'] && {\left[\bS^d, \Sigma^\infty Y_+ \wedge \mathbb{L}^q\langle 1 \rangle_{(2)}\right]} \arrow[dd, "TW \eqref{TW-identification-two}", "\cong"']  \\
 \begin{tabular}{c}
$\bigoplus_{i\geq1} H^{4i}(Y,\partial Y;\Z_{(2)})$ \\ $\hspace{8ex}\oplus\bigoplus_{j\geq0} H^{4j+2}(Y,\partial Y;\Z/2)$ \end{tabular}
\ar[d,"{\Phi}","\cong"']\\
 \begin{tabular}{c}
$\bigoplus_{i\geq1} H^{4i}(Y,\partial Y;\Z_{(2)})$ \\ $\hspace{8ex}\oplus\bigoplus_{j\geq0} H^{4j+2}(Y,\partial Y;\Z/2)$ \end{tabular}
\arrow[rr, "{[Y] \frown -}", "\cong"'] &&
 \begin{tabular}{c}
$\hspace{-6ex}\bigoplus_{i\geq1}H_{d-4i}(Y;\Z_{(2)})$ \\
$\hspace{2ex}\oplus \bigoplus_{j\geq0} H_{d-4j-2}(Y;\Z/2)$
\end{tabular}
\end{tikzcd}
\]
commutes, where the map $\Phi$ is given by
$$(u,v) \mapsto ((\nu^{\Top}_Y)^* \ell \smile u, (\nu^{\Top}_Y)^*\ell \smile v + \delta^*((\nu^{\Top}_Y)^*(\mathbb{V} \smile v))$$
for $\delta^*$ the cohomology Bockstein associated to the coefficient sequence $0\to \Z_{(2)} \overset{\cdot 2}\to  \Z_{(2)}\to \Z/2\to 0$.
Here $TW$~\eqref{TW-identification-one} and $TW$~\eqref{TW-identification-two} are the maps denoted $\cong_{TW}$ in \eqref{TW-identification-one} and~\eqref{TW-identification-two} respectively.
\end{lemma}

\begin{proof}
The top square commutes by the definition of the map $\mathrm{D}$ from \eqref{eqn:defn-of-D}.

Write $m\colon\mathbb{L}^s_{(2)}\wedge \mathbb{L}^q\langle 1 \rangle_{(2)}\to \mathbb{L}^q\langle 1 \rangle_{(2)}$ for the $\mathbb{L}^s_{(2)}$-module structure map. According to~\cite[(1.13)]{MR557167}, the spectrum cohomology classes $l^q$, $k$, $l^s$ and $r$ defining the identifications~\eqref{eq:symmetric} and~\eqref{eq:quadratic} behave as follows under pullback along the map $m$:
\begin{equation*}
m^*l^q=l^s\wedge l^q+\delta^*(r\wedge k)\qquad\text{and}\qquad m^*k=l^s\wedge k.
\end{equation*}
By \eqref{eq:quadratic}, the spectrum $\mathbb{L}^q\langle 1 \rangle_{(2)}$ is homotopy equivalent to a coproduct of Eilenberg--Maclane spectra via the maps $l^q$ and $r$, so these formulae fully determine the homotopy class of the map $m$.

Using the formula for $[Y]_{\mathbb{L}^s_{(2)}}$ under the splitting \eqref{eq:symmetric} given by Lemma \ref{lem:LThyFundamentalClassUnderSplitting}, and this description of the map $m$, we may derive a formula for the cap product map
\[
[Y]_{\mathbb{L}^s_{(2)}} \frown -\colon \left[\Sigma^\infty (Y/\partial Y), \mathbb{L}^q\langle 1 \rangle_{(2)}\right] \overset{\cong}\lra \left[\bS^d, \Sigma^\infty Y_+ \wedge \mathbb{L}^q\langle 1 \rangle_{(2)}\right]
\]
under the identifications \eqref{TW-identification-one} and \eqref{TW-identification-two}. Namely, it is given by sending
\[
(u,v)\in \bigoplus_{i\geq1} H^{4i}(Y,\partial Y;\Z_{(2)})\oplus \bigoplus_{j\geq0} H^{4j+2}(Y,\partial Y;\Z/2)
\]
to
\begin{multline*}
([Y]\frown ((\nu^{\Top}_Y)^*\ell \smile u),[Y]\frown ((\nu^{\Top}_Y)^*\ell \smile v)+ [Y]\frown\delta^*( (\nu^{\Top}_Y)^*\mathbb{V})\smile v))\\
\in \bigoplus_{i\geq1}H_{d-4i}(Y;\Z_{(2)})\oplus \bigoplus_{j\geq0} H_{d-4j-2}(Y;\Z/2).
\end{multline*}
This describes the composition $TW \eqref{TW-identification-two} \circ ([Y]_{\mathbb{L}^s_{(2)}} \frown -)\circ (TW \eqref{TW-identification-one})^{-1}$ in the diagram, which agrees with the composition $([Y] \frown -)\circ \Phi$ by the definition of $\Phi$.
\end{proof}

\subsection{Proof of Theorem \ref{thm:main-technical}~\ref{it:main-technicalC}}\label{sub:proof-thm-C}

Now we specialise the discussion of the previous sections to the case at hand, with $Y = X \times I \times I$ for an oriented 4-manifold $X$. We may consider the diagram
\begin{equation}\label{eq:developed-diagram}
\begin{tikzcd}[column sep=2.5ex]
H_2(X;\Z/2) & {\left[\tfrac{X \times I \times I}{\partial (X \times I \times I)}, \G/\Top\right]_{(2)}} \arrow[d, "\mathrm{D}","\cong"'] \arrow[l, swap, "\ks"] \arrow[rr, "\sigma"] & &  L_{6}(\Z[\pi])_{(2)}
\\
& {\left[\bS^6, \Sigma^\infty X_+ \wedge \mathbb{L}^q\langle 1 \rangle\right]_{(2)}} \arrow[d, "TW","\cong"'] \rar & {\left[\bS^6, \Sigma^\infty X_+ \wedge \mathbb{L}^q\right]_{(2)}} \rar{\mft_*} & {\left[\bS^6, \Sigma^\infty \BB\pi_+ \wedge \mathbb{L}^q\right]_{(2)}} \arrow[u, "A_\pi"] \arrow[d, "TW","\cong"']\\
& \begin{tabular}{l}
$H_0(X ; \Z/2)$ \\
$\hspace{2ex}\oplus H_2(X;\Z_{(2)})$ \\
$\hspace{4ex}\oplus H_4(X;\Z/2)$
\end{tabular} \arrow[rr, "\mft_* \oplus 0"] \arrow[luu, bend left, swap, "\Psi"]& &
\begin{tabular}{l}
$H_0(\pi ; \Z/2)$ \\
$\hspace{2ex}\oplus H_2(\pi;\Z_{(2)})$ \\
$\hspace{4ex}\oplus H_4(\pi;\Z/2)$ \\
$\hspace{6ex}\oplus H_6(\pi;\Z_{(2)}).$
\end{tabular}
\end{tikzcd}
\end{equation}
We use that $X \times I \times I \simeq X$ to simplify notation here. The top rectangle comes from Subsection~\ref{subsection:surgery-obstr-and-assembly}.  For the bottom rectangle we applied $TW$ to the $\mathbb{L}^q$ homology of both $\Sigma^{\infty}X_+$, and $\Sigma^{\infty}\BB\pi_+$, using naturality with respect to the truncation map $\mft$ and the change of homology theory $\mathbb{L}^q\langle 1 \rangle \to \mathbb{L}^q$ to see that the rectangle commutes.
The map $\Psi$, defined so as to make the left-hand triangle commute, can be evaluated by applying Lemma \ref{lem:KSformula} and Lemma \ref{lem:TheTwoIsomorphisms}.

\begin{lemma}\label{lem:Psi}
We have $\Psi(a,b,c) = \mathrm{red}_2(b)$.
\end{lemma}
\begin{proof}
We consider the diagram of Lemma \ref{lem:TheTwoIsomorphisms} for $Y=X \times I \times I$. This has
\[(\nu^{\Top}_Y)^*\ell = 1 + (\nu^{\Top}_Y)^*\ell_1 \in H^0(Y,\partial Y;\Z_{(2)}) \oplus H^4(Y,\partial Y;\Z_{(2)})\]
and
\[(\nu^{\Top}_Y)^*(\mathbb{V}) = (\nu^{\Top}_Y)^*(\sum_{j \geq 0} V_{2j} \smile \mathrm{Sq}^1(V_{2j})) = (\nu^{\Top}_Y)^* (V_2 \smile \mathrm{Sq}^1(V_{2})) \in H^5(Y,\partial Y;\Z/2)\]
 for degree reasons, namely that $\Sq^1(V_0)=0$ and $H^k(Y,\partial Y;\Z_{(2)})=0$ for $k >6$. Thus the isomorphism
\begin{multline*}
\Phi\colon  H^2(Y, \partial Y ; \Z/2) \oplus H^4(Y, \partial Y ; \Z_{(2)}) \oplus H^6(Y, \partial Y ; \Z/2)\\
 \lra H^2(Y, \partial Y ; \Z/2) \oplus H^4(Y, \partial Y ; \Z_{(2)}) \oplus H^6(Y, \partial Y ; \Z/2)
\end{multline*}
is given by $(u,v,w) \mapsto (u,v, w+(\nu^{\Top}_Y)^* \ell_1 \smile u)$.

Start with a class
\[(a,b,c) = ([Y] \frown z, [Y] \frown y, [Y] \frown x) \in H_0(Y ; \Z/2) \oplus H_2(Y;\Z_{(2)})\oplus H_4(Y;\Z/2)\]
in the bottom-left corner of the diagram~\eqref{eq:developed-diagram}, which equals the bottom-right corner of the diagram of Lemma \ref{lem:TheTwoIsomorphisms}  (in this case we take just $i=1$ and $j=0,1$). Here $(x,y,z)$ in the codomain of~$\Phi$ are uniquely determined by Poincar\'{e} duality.
We wish to pass from here anticlockwise around the diagram from Lemma \ref{lem:TheTwoIsomorphisms} to the domain of $\Phi$, from which we can evaluate the map $\ks$ using Lemma \ref{lem:KSformula}.
By Lemma \ref{lem:TheTwoIsomorphisms}, we can instead pass clockwise around the diagram from that lemma, i.e.\ apply the isomorphism $\Phi^{-1} \,{\circ}\, ([Y] \frown -)^{-1}$.
We obtain the class
$$(x, y, z-(\nu^{\Top}_Y)^*\ell_1 \smile x) \in H^2(Y, \partial Y ; \Z/2) \oplus H^4(Y, \partial Y ; \Z_{(2)}) \oplus H^6(Y, \partial Y ; \Z/2).$$
By Lemma \ref{lem:KSformula} the Kirby--Siebenmann class $\ks$ is given by $\mathrm{red}_2(l_4^q) + k_2 \smile k_2$, so is $\mathrm{red}_2(y) + x \smile x \in H^4(Y, \partial Y ; \Z/2)$. But as
$Y/\partial Y \cong (X/\partial X) \wedge S^2$
is a reduced suspension, the cup product $x \smile x$ vanishes. Thus under Poincar{\'e} duality the Kirby--Siebenmann class is $[Y] \frown \mathrm{red}_2(y) = \mathrm{red}_2([Y] \frown y) = \mathrm{red}_2(b)$, so $\Psi$ has the claimed description.
\end{proof}

We can now describe the map $I_2$ to which the statement of
Theorem~\ref{thm:main-guarantee} and Theorem \ref{thm:main-technical}~\ref{it:main-technicalC}  refers, and then complete the argument.

\begin{definition}\label{defn:I2}
The map $I_2\colon H_2(\pi ;\Z_{(2)}) \to L_{6}^q(\Z[\pi])_{(2)}$ is by definition the restriction of the right-hand column $A_\pi \circ TW^{-1}$ of the
diagram at the start of Section~\ref{sub:proof-thm-C} to the summand~$H_2(\pi ;\Z_{(2)})$.
\end{definition}

\begin{proof}[Proof of Theorem~\ref{thm:main-technical}~\ref{it:main-technicalC}]
By the strategy described at the end of Section~\ref{sec:GeomSurgery}, after diagram~\eqref{eq:diagram-pf-strategy-for-23C}, it suffices to produce, for each $u \in \ker \big(\delta_* \colon H_2(X ; \Z/2) \to H_1(X ; \Z_{(2)})\big)$, a lift $\xi_u \in {\big[\tfrac{X \times I \times I}{\partial (X \times I \times I)}, \G/\Top\big]_{(2)}}$ along $\ks$ such that $\sigma(\xi_u)=0$. We do so as follows. Given such a class $u$,
exactness of the Bockstein sequence
$$H_2(X ; \Z_{(2)}) \overset{\mathrm{red}_2}\lra H_2(X ; \Z/2) \overset{\delta_*}\lra H_1(X ; \Z_{(2)}) \overset{2}\lra H_1(X ; \Z_{(2)})$$
shows that we may choose a $u' \in H_2(X ; \Z_{(2)})$ which reduces modulo 2 to $u$. We then consider the element
$$(0, u', 0) \in H_0(X ; \Z/2) \oplus H_2(X;\Z_{(2)}) \oplus H_4(X;\Z/2)$$
in the bottom-left corner of the previous diagram~\eqref{eq:developed-diagram}. Then $(0,u',0)$ corresponds under the vertical isomorphism $\mathrm{D}^{-1} \circ TW^{-1}$ to an element
\[
\xi_u \in {\left[\tfrac{X \times I \times I}{\partial (X \times I \times I)}, \G/\Top\right]_{(2)}}
\]
which satisfies that $\ks(\xi_u)=u$, by the description of $\Psi$ in Lemma \ref{lem:Psi}. Under our assumption that the map $I_2\colon H_2(\pi ;\Z_{(2)}) \to L_{6}(\Z[\pi])_{(2)}$ is zero, the element $(0, u', 0)$ maps to zero in $L_{6}(\Z[\pi])_{(2)}$ by going anticlockwise around the diagram~\eqref{eq:developed-diagram}, and therefore $\sigma(\xi_u)=0$ as required. This completes the proof of Theorem \ref{thm:main-technical}~\ref{it:main-technicalC} and hence of Theorem~\ref{thm:main-guarantee}.
\end{proof}

\section{Proof of Theorem \ref{thm:main-realise}}\label{sec:proof-realisation-thm}

We recall that Theorem \ref{thm:main-realise} states the following.
For each smooth, connected, compact $4$-manifold~$X$ there is a $g \geq 0$ such that every class $x \in H_2(X \# g(S^2 \times S^2) ;\Z/2)$ arises as $\KS(F)$ for some topological pseudo-isotopy $F$ from the identity to some diffeomorphism $f\colon X \# g(S^2 \times S^2) \to X \# g(S^2 \times S^2)$.

\begin{proof}[Proof of Theorem \ref{thm:main-realise}]
Without loss of generality we may suppose that $X$ has nonempty boundary, by removing an open disc.

Given a class $x \in H_2(X;\Z/2)$, we use Poincar{\'e} duality and a suspension isomorphism to consider it as an element of
\[
H^3(X \times I, \partial ; \Z/2) = \left[\tfrac{X \times I}{\partial}, \Top/\OO\right],
\]
using the fact that the Postnikov 6-type of $\Top/\OO$ is a $K(\Z/2,3)$. Thus $x$ gives a smooth structure $(X \times I)_{\Sigma(x)}$ on the 5-manifold $X \times I$ relative to the standard smoothing of the boundary, by the discussion in Section \ref{sec:smoothing}. As $X = X \times \{0\} \hookrightarrow X \times I$ is a simple homotopy equivalence, the smooth $5$-manifold $(X \times I)_{\Sigma(x)}$ is an $s$-cobordism. By the $(S^2 \times S^2)$-stable $s$-cobordism theorem~\cite{quinn-1983}, there is a $g \geq 0$ such that the stabilisation of the $s$-cobordism $(X \times I)_{\Sigma(x)}$ obtained by gluing on the trivial smooth cobordism on $(\partial X \times I) \# g(S^2 \times S^2)$ is a trivial cobordism starting at \[X' := X \cup_{\partial X}\left ((\partial X \times I) \# g(S^2 \times S^2) \right) \cong X \# g(S^2 \times S^2).\]
This yields maps
\[
\begin{tikzcd}
 X' \times I \ar[r,"{\cong_{\Top}}"] \ar[r] & (X' \times I)_{\Sigma(x)} \ar[r,"{\cong_{\Diff}}"]  & X' \times I.
\end{tikzcd}
\]

The first is just the identity homeomorphism between two copies of $X' \times I$, while the second map is the diffeomorphism arising from the $(S^2 \times S^2)$-stable $s$-cobordism theorem.

This yields a self-homeomorphism of $X' \times I$ that is the identity on $X' \times \{0\}$ and that restricts on $X' \times \{1\}$ to a diffeomorphism $f$. This is the same as a topological pseudo-isotopy $F$ from the identity to $f$. By construction, this data has
$$\KS(F) = (x, 0) \in  H_2(X;\Z/2) \oplus H_2(g(S^2 \times S^2) ;\Z/2) = H_2(X';\Z/2).$$
By choosing $g$ large enough, since $H_2(X;\Z/2)$ is finite we may arrange that all of $H_2(X;\Z/2) \subseteq H_2(X',\Z/2)$ lies in the image of $\KS\colon Q(X') \to H_2(X';\Z/2)$.

For the summand $H_2(g(S^2 \times S^2) ;\Z/2) \subseteq H_2(X';\Z/2)$, we use that $W_{g,1} := g(S^2 \times S^2) \setminus \operatorname{Int}(D^4)$ is simply-connected and so the hypotheses of Theorem~\ref{thm:main-guarantee} are satisfied for this manifold (alternatively, the conclusion of Theorem~\ref{thm:main-guarantee} for simply connected closed manifolds was already a consequence of~\cite[Theorem~1]{Kreck-isotopy-classes}, and his arguments readily extend to $W_{g,1}$). It follows that
\[
\pi_1(\pseudoHomeo^+_\partial(W_{g,1})) \overset{\alpha}\lra Q(W_{g,1}) \overset{\ks}\lra H_2(W_{g,1};\Z/2)
\]
is surjective. Plugging such topological pseudo-isotopies of $W_{g,1}$ into $X'$ shows that $H_2(g(S^2 \times S^2) ;\Z/2) \subseteq H_2(X';\Z/2)$ also lies in the image of $\KS$ for the manifold $X'$. Since $\KS$ is a homomorphism by Lemma~\ref{lem:ksConcordanceInvariant}, this completes the proof of Theorem \ref{thm:main-realise}.
\end{proof}

\begin{remark}\label{rem:observation}
It is a consequence of the proof of Theorem~\ref{thm:main-realise} that the pseudo-isotopy $F$ realising a given $x\in H_2(X\#g(S^2\times S^2);\Z/2)$ can be assumed to restrict to the identity map on $D^4\times I$, for any $D^4\subseteq X$ away from $\partial X$. To see this:~remove the interior $\mathring{D}^4\subseteq X$ at the beginning of the proof, perform the proof to obtain a pseudo-isotopy that is the identity map on $(\partial D^4)\times I$, then fill~$\mathring{D}^4\times I$ back in, extending the pseudo-isotopy by the identity map.
\end{remark}

\section{Stable topological pseudo-isotopy of homeomorphisms}
\label{sec:newsection}

The main challenge in proving Theorem~\ref{thm:main-example} will be the construction of the diffeomorphism~$f\colon X\to X$ topologically pseudo-isotopic but not smoothly pseudo-isotopic to the identity, and the entirety of Section~\ref{sec:proof-thm-main-example} is dedicated to this. We also state in Theorem~\ref{thm:main-example} that $f$ is not even \emph{stably} smoothly pseudo-isotopic to the identity. In this section we prove a technical result of independent interest, that will help deal with this upgrade. The results of this section will be used exactly once, towards the end of the proof of Theorem~\ref{thm:main-example}.
The next lemma is well-known, but we could not find a published proof; for convenience we provide a proof here, following Lurie \cite{lurie2009derivedalgebraicgeometryvi}*{p.~173}.

%complete proof in the literature; a sketch was provided on \cite{lurie2009derivedalgebraicgeometryvi}*{p.~173}.

% \mp{Added a sentence to make abundantly clear. Stating the length of time each previously known thing has been known for seems like a lot of extra work though... }

\begin{lemma}\label{lemma:KK-trivial-statements}
\leavevmode
Let $M$ be a topological $n$-manifold. Then the map $\ev \colon \mathrm{Emb}^{\Top}(D^n,M) \to \mathring{M}$ given by evaluating the embedding at $0 \in D^n$ is a Serre fibration, with fibre weakly homotopy equivalent to ~$\Top(n)$.
\end{lemma}
%\orw{The first part of this is on p.173 of Lurie's https://arxiv.org/pdf/0911.0018. Even if we don't want to cite that instead, please let us make it abundantly clear that this has been well known for 50 years.}

\begin{proof}
The evaluation map $\ev$ is a Serre fibration by the parametrised isotopy extension theorem. The fibre over a point $q \in \mathring{M}$ is $\Embt_q(D^n,M)$, the space of embeddings sending $0$ to $q$. There are origin-preserving embeddings $\iota_1 \colon D^n \to \R^n$ and $\iota_2\colon \R^n \to D^n$, which induce maps $\Embt_q(\R^n,M) \to \Embt_q(D^n,M)$ and
$\Embt_q(D^n,M) \to \Embt_q(\R^n,M)$, respectively.
Since $\iota_1 \circ \iota_2$ and $\iota_2 \circ \iota_1$ are both isotopic to the respective identities, these maps on embedding spaces are homotopy inverses, and $\Embt_q(D^n,M) \simeq \Embt_q(\R^n,M)$; this argument is referred to as saying that $D^n$ and $\R^n$ are \emph{isotopy equivalent}.
Thus it suffices to show that $\Embt_q(\R^n,M)$ is weakly homotopy equivalent to $\Top(n)$.

To show topological spaces are weakly equivalent we may replace them by their singular simplicial sets. Let us write $\Embt_q(\R^n,M)_\bullet$ for that of $\Embt_q(\R^n,M)$, and so on. Let $\R^n \cong U \subseteq M$ be a coordinate chart centred at $q$, and consider the induced inclusion $\Embt_q(\R^n,U)_\bullet \to \Embt_q(\R^n,M)_\bullet$. By restricting to the balls $\mathring{B}_r(0) \subseteq \R^n$ we obtain an induced map
\begin{equation}\label{eq:EqOnColim}
\underset{r\to 0}{\colim} \Embt_q(\mathring{B}_r(0),U)_{\bullet} \lra \underset{r\to 0}{\colim}\Embt_q(\mathring{B}_r(0),M)_{\bullet}
\end{equation}
between colimits of simplicial sets. This is certainly injective in each simplicial degree, but in fact is a bijection in each simplicial degree. To see this, note that a singular $p$-simplex $\Delta^p \to \Embt_q(\mathring{B}_r(0),M)$ is equivalent to an $e \colon \mathring{B}_r(0) \times \Delta^p \to M \times \Delta^p$ over $\Delta^p$. For such a representative~$e$ of a $p$-simplex, we have that $e^{-1}(U \times \Delta^p)$ is an open neighbourhood of $\{0\} \times \Delta^p \subseteq\mathring{B}_r(0) \times \Delta^p$, so by compactness of $\Delta^p$ (the ``Tube Lemma'') there is an $r' < r$ such that $\{0\} \times \Delta^p \subseteq \mathring{B}_{r'}(0) \times \Delta^p \subseteq e^{-1}(U \times \Delta^p)$. But then $e\vert_{\mathring{B}_{r'}(0) \times \Delta^p}$ is a $p$-simplex of $\Embt_q(\mathring{B}_{r'}(0),U)_{\bullet}$ which represents the same element of $\underset{r\to 0}{\colim}\Embt_q(\mathring{B}_r(0),M)_{\bullet}$ as $e$.

We also have that the restriction maps
\[\Embt_q(\R^n,U) \lra  \Embt_q(\mathring{B}_r(0),U)\quad \text{and} \quad \Embt_q(\R^n,M) \lra  \Embt_q(\mathring{B}_r(0),M)\]
are homotopy equivalences for each $r$, as $\mathring{B}_r(0) \subseteq \R^n$ is an isotopy equivalence, and the same holds for their singular sets. Because taking homotopy groups commutes with filtered colimits of (Kan) simplicial sets, we have a commutative diagram, for each $m \geq 0$:
\[\begin{tikzcd}[column sep = small, row sep = small]
  \pi_m\Embt_q(\R^n,U)_\bullet \ar[d,"\cong"] \ar[r] &  \pi_m\Embt_q(\R^n,M)_\bullet \ar[d,"\cong"]  \\  
   \underset{r\to 0}{\colim} \, \pi_m \Embt_q(\R^n,U)_\bullet \ar[d,"\cong"] \ar[r] &  \underset{r\to 0}{\colim} \, \pi_m \Embt_q(\R^n,M)_\bullet \ar[d,"\cong"]  \\
  \underset{r\to 0}{\colim}\, \pi_m \Embt_q(\mathring{B}_r(0),U)_{\bullet} \ar[d,"\cong"] \ar[r] &  \underset{r\to 0}{\colim} \, \pi_m \Embt_q(\mathring{B}_r(0),M)_{\bullet} \ar[d,"\cong"]  \\ 
   \pi_m \underset{r\to 0}{\colim}  \Embt_q(\mathring{B}_r(0),U)_{\bullet}  \ar[r] &  \pi_m \underset{r\to 0}{\colim}  \Embt_q(\mathring{B}_r(0),M)_{\bullet}
\end{tikzcd}\]
By the isomorphism \eqref{eq:EqOnColim}, the bottom horizontal map is an isomorphism, and hence by commutativity so is the top horizontal map. Thus the inclusion $\Embt_q(\R^n,U)_\bullet \to \Embt_q(\R^n,M)_\bullet$ is a weak homotopy equivalence,   
%\[\pi_m\Embt_q(\R^n,U)_\bullet \cong  \underset{r\to 0}{\colim} \pi_m \Embt_q(\R^n,U)_\bullet \cong \underset{r\to 0}{\colim} \pi_m \Embt_q(\mathring{B}_r(0),U)_{\bullet} \cong \pi_m \underset{r\to 0}{\colim}  \Embt_q(\mathring{B}_r(0),U)_{\bullet}\]
%that are natural with respect to $U$, and similarly for . It follows from the isomorphism \eqref{eq:EqOnColim} that $\pi_m \underset{r\to 0}{\colim}  \Embt_q(\mathring{B}_r(0),U)_{\bullet} \to \pi_m \underset{r\to 0}{\colim}  \Embt_q(\mathring{B}_r(0),M)_{\bullet}$ is an isomorphism, and hence by naturality the inclusion $\Embt_q(\R^n,U)_\bullet \to \Embt_q(\R^n,M)_\bullet$ is a weak homotopy equivalence.  
and hence $\Embt_q(\R^n,U) \to \Embt_q(\R^n,M)$ is too. Finally, we have
\[\Top(n) \lra \Embt_q(\R^n,\R^n) \cong \Embt_q(\R^n,U)\]
where the first map is a weak homotopy equivalence by Kister's theorem~\cite{Kister}.
\end{proof}

Given a homeomorphism $f$ of a $4$-manifold, a \emph{$k$-stabilisation} is the result of isotoping $f$ to fix a $4$-ball pointwise, performing a connected sum with $k(S^2\times S^2)$ using that fixed $4$-ball, and extending~$f$ by the identity map over the connected sum.

\begin{proposition}
\label{prop:stablemma}
Let $M=M_0\#N$ be a closed $4$-manifold, where $N\neq S^4$ is simply-connected, and suppose $f\colon M\to M$ is a homeomorphism that is topologically pseudo-isotopic to the identity. Then for any $k>0$, there is a unique topological isotopy class of $k$-stabilisation of $f$ that is topologically pseudo-isotopic to the identity.
\end{proposition}

\begin{proof}
Let $F$ be a topological pseudo-isotopy from $f$ to the identity $\Id_M$. Perform an initial isotopy on $F$, relative to $\Id_M$, so that $F$ fixes a thickened arc $D^4\times I\subseteq M\times I$ pointwise. This modifies~$f$ by a topological isotopy to a new homeomorphism, which we will continue to call $f$. Similarly, the isotoped $F$ is still called $F$. Let $k>0$ and form the specific $k$-stabilisation~$M \# k(S^2 \times S^2)=(M\sm\mathring{D}^4) \cup_{S^3} (k(S^2 \times S^2) \sm \mathring{D}^4)$.
By extending $F$ restricted to $(M\sm\mathring{D}^4) \times I$ by the identity on~$(k(S^2 \times S^2) \sm \mathring{D}^4) \times I$, we obtain a topological pseudo-isotopy from this $k$-stabilisation of $f$ to the identity.

Now suppose instead we had performed a potentially different initial isotopy to a homeomorphism $f'$ fixing a $4$-ball pointwise. By a further topological isotopy, we can assume $f'$ fixes the same $4$-ball $D^4\subseteq M$ as $f$. Assuming that $f'\#\Id_{k(S^2\times S^2)}$ is topologically pseudo-isotopic to the identity, we will show it is isotopic to $f\#\Id_{k(S^2\times S^2)}$.

Considering the long exact sequence of the fibration $\Homeo(M)\to\mathrm{Emb}^{\Top}(D^4,M)$, the maps $f$ and~$f'$ differ up to isotopy only by the action of $\pi_1(\mathrm{Emb}^{\Top}(D^4,M))$.
By the long exact sequence of the fibration $\ev \colon \mathrm{Emb}^{\Top}(D^4,M) \to M$ from Lemma~\ref{lemma:KK-trivial-statements},
we have that $\pi_1(\mathrm{Emb}^{\Top}(D^4,M))$ is an extension of some subgroup of $\pi_1(M,p)$ by a quotient of $\pi_1(\Top(4))$.
Here, the basepoint~$p$ lies in~$D^4$, and we will assume it lies in the boundary $S^3$.

It is well-known that $\pi_1(\Top(4))\cong\Z/2$; to compute this, use that the map $\Top(4)/\OO(4) \to \Top/\OO$ is 5-connected~\cite{FQ}*{Theorem~8.7A}, so in particular $\pi_i(\Top(4)/\OO(4)) \cong \pi_i(\Top/\OO) =0$ for $i=1,2$.
  The long exact sequence in homotopy groups of the fibration $\Top(4)/\OO(4) \to \BO(4) \to \BTop(4)$ then shows $\pi_2(\BO(4)) \cong \pi_2(\BTop(4))$, so $\pi_1(\Top(4)) \cong \pi_2(\BTop(4)) \cong \pi_2(\BO(4))  \cong \pi_1(\OO(4)) \cong \Z/2$.
  We thus see the element of $\Z/2$ comes from a generator of $\pi_1(\OO(4))$ and so is represented by a Dehn twist on a tubular neighbourhood of the boundary $S^3$. After performing the connected sum, this may be thought of as a Dehn twist on the boundary of~$k(S^2\times S^2)\sm\mathring{D}^4$, which is topologically isotopic to the identity, rel.~boundary~\cite{Orson-Powell}*{Corollary~C}. So, up to topological isotopy, $f\#\Id_{k(S^2\times S^2)}$ and $f'\#\Id_{k(S^2\times S^2)}$ differ only by the action of an element corresponding to some~$[\gamma]\in\pi_1(M,p)$. This based loop modifies the basepoint-fixing diffeomorphism $f$ via a ``point-pushing'' homeomorphism $P_\gamma\in\Homeo(X)$ that is the effect of an ambient isotopy that takes~$D^4$ and pushes it around the loop~$\gamma$, coming back to where it started. Hence $f'\#\Id_{k(S^2\times S^2)}$ is isotopic to~$(P_\gamma\circ f)\#\Id_{k(S^2\times S^2)}$.

Consider the action of this map on the~$\Z[\pi_1(M\#k(S^2\times S^2),p)]$-module $\pi_2(M\#k(S^2\times S^2))\cong\pi_2(M)\oplus(\Z\pi)^{2k}$. The action on the second summand is trivial, because the map is the identity there and the basepoint lies in~$k(S^2 \times S^2) \sm \mathring{D}^4$. As $f$ is topologically pseudo-isotopic to the identity via a pseudo-isotopy that fixes the point~$p$ throughout, we have $(P_\gamma\circ f)_*=(P_\gamma)_*$ on $\pi_2(M)$. Any basing arcs from~$p$ to spheres generating $\pi_2(M)$ are pulled around the loop $\gamma$ by $P_\gamma$, so $(P_\gamma)_*$ is by overall multiplication by $[\gamma]^{-1}$. By assumption $f'\#\Id_{k(S^2\times S^2)}$, and therefore $(P_\gamma\circ f)\#\Id_{k(S^2\times S^2)}$, is topologically pseudo-isotopic to the identity. Thus the action on $\pi_2(M\#k(S^2\times S^2))$ is overall multiplication by some element $[\gamma']\in\pi_1(M,p)$, where $\gamma'$ is the image of $p\times I$ under this pseudo-isotopy. Considering the map on the second summand, we see this forces $[\gamma']=1$. But as~$N\neq S^4$, we have at least one free~$\Z[\pi_1(M,p)]$-summand in $\pi_2(M)$, so considering the map on the first summand, we see~$[\gamma]^{-1}=[\gamma']=1$. Hence $P_\gamma$ is topologically isotopic to the identity. This shows~$f\#\Id_{k(S^2\times S^2)}$ and~$f'\#\Id_{k(S^2\times S^2)}$ are topologically isotopic to one another.
\end{proof}

\section{Proof of Theorem \ref{thm:main-example}}\label{sec:proof-thm-main-example}

We begin the proof of Theorem~\ref{thm:main-example} by
defining the $4$-manifold $X$ appearing in the statement.
Then we outline the strategy of the proof, motivating the conditions the manifold~$X$ will have  been constructed to satisfy.
The strategy
involves choosing judicious maps into $\PL$ normal invariant sets, and the commutativity of a diagram. The rest of this section, and the rest of the proof of Theorem~\ref{thm:main-example}, will comprise the proof that this diagram commutes.

\subsection{The $4$-manifold \texorpdfstring{$X$}{X}}\label{sub:the-manifold-X}

Let $A$ be a closed smooth $4$-manifold, and consider the following list of conditions.
\begin{enumerate}[label=(\roman*)]
\item\label{it:Aspherical} $A$ is aspherical,
\item\label{it:Fram} $A$ is stably framable,
\item\label{it:Centre} the map $\pi_1(\Diffeo^+(A)) \to \pi_1({\operatorname{hAut}\mkern 0mu}^+(A))$ is surjective,
\item\label{it:Tors} $H_1(\pi_1(A); \Z)$ has an element of order 2.
\end{enumerate}

\begin{remark}
We are interested in the higher homotopy groups of various automorphism groups based at the identity. As homotopy automorphisms which are homotopic to the identity are simple,  we may freely interchange ${\operatorname{hAut}\mkern 0mu}^+$ and ${\operatorname{sAut}\mkern 0mu}^+$, and similarly for $\hAut$ and $\sAut$. By the same token homeomorphisms in the identity component preserve orientation, so we may also freely interchange $\wt{\Homeo}\mkern 0mu$ and ${\wt{\Homeo}\mkern 0mu^+}$, and so on.
\end{remark}

\begin{remark}\label{rem:asphericalcentre}
For an aspherical space $A$, there is an isomorphism $\pi_1({\operatorname{hAut}\mkern 0mu}(A))\cong Z(\pi_1(A))$, where the latter denotes the centre~\cite[Corollary~I.13]{MR189027}. This can be argued by picking a point $x\in A$ and considering the map $\operatorname{ev}_x\colon {\operatorname{hAut}\mkern 0mu}(A)\to A$, given by evaluating the automorphism at $x$. The homotopy fibre of this map is the space of based homotopy equivalences ${\operatorname{hAut}\mkern 0mu}_x(A)$. As $A$ is aspherical, we have $\pi_0({\operatorname{hAut}\mkern 0mu}_x(A))\cong\Aut(\pi_1(A,x))$ and $\pi_k({\operatorname{hAut}\mkern 0mu}_x(A))=0$ for all $k\geq 0$. The long exact sequence of the fibration thus has a portion
\[
0\lra \pi_1({\operatorname{hAut}\mkern 0mu}(A), \id)\lra \pi_1(A, x)\lra\pi_0({\operatorname{hAut}\mkern 0mu}_x(A))\cong\Aut(\pi_1(A,x)).
\]
An element $g\in\pi_1(A,x)$ is sent to the map $\gamma\mapsto g\gamma g^{-1}\in \Aut(\pi_1(A,x))$ in this sequence, and thus $\pi_1(A)\to\pi_0({\operatorname{hAut}\mkern 0mu}_x(A))$ has kernel $Z(\pi_1(A))$. For further details we refer to~\cite{MR189027}.
\end{remark}

The next example shows that there exist many $4$-manifolds $A$ satisfying these conditions.

\begin{example}\label{ex:hyperbolic}
Let $\Sigma$ be a closed, oriented hyperbolic 3-manifold with 2-torsion in $H_1(\Sigma;\Z)$. Such manifolds $\Sigma$ are abundant; for example, take any hyperbolic knot $K\subseteq S^3$ and perform $n$-surgery on $S^3$ along $K$, where $n$ is even and $n/1$ is not one of the finite number of \emph{exceptional} slopes $p/q$ for $S^3\sm\nu K$ which produce a non-hyperbolic filling. A concrete example is given by performing~$6$-surgery on the figure eight knot~\cite[Theorem~4.7]{zbMATH01014124}.

The manifold $A := \Sigma \times S^1$ satisfies conditions \ref{it:Aspherical}, \ref{it:Fram}, \ref{it:Centre}, and \ref{it:Tors}. Indeed, condition \ref{it:Aspherical} holds because~$\Sigma$ is hyperbolic, and thus aspherical, so $\Sigma\times S^1$ is a product of aspherical spaces, thus also aspherical. Condition \ref{it:Fram} holds as oriented 3-manifolds are framable, and so $A$ is a product of framable manifolds. For condition~\ref{it:Centre} we will argue that $\pi_1(\Sigma)$ has trivial centre, so that the centre of $\pi_1(A)\cong\pi_1(\Sigma)\times\pi_1(S^1)$ is $\Z$ generated by the circle factor. But the circle action on $\Sigma \times S^1$ is smooth and generates this factor in $\pi_1({\operatorname{hAut}\mkern 0mu}(A))\cong Z(\pi_1(A))$ (recall Remark~\ref{rem:asphericalcentre}), showing that~\ref{it:Centre} is satisfied.
Condition~\ref{it:Tors} holds as $\Sigma$ was chosen to have $2$-torsion in $H_1(\Sigma;\Z)$.

It remains to see that the centre of the group $\pi_1(\Sigma)$ is trivial. For a contradiction, let $g$ be a nontrivial element of the centre $Z(\pi_1(\Sigma))$. Since $\pi_1(\Sigma)$ is torsion-free the subgroup $\langle g \rangle$ is isomorphic to $\Z$. By \cite[Corollary~III.$\Gamma$.3.10~(2),~p.~462]{Bridson-Haefliger}, $\langle g \rangle$ has finite index in its centraliser~$C_{\pi_1(\Sigma)}(g)$. Since $g$ is central, $C_{\pi_1(\Sigma)}(g) = \pi_1(\Sigma)$, so $\pi_1(\Sigma)$ has a subgroup $\Z$ of finite index. The corresponding finite cover $\wt{\Sigma}$ is a closed, orientable, aspherical 3-manifold and so $\Z \cong H_3(\wt{\Sigma};\Z) \cong H_3(\BB \pi_1(\wt{\Sigma});\Z) \cong H_3(\BB \Z;\Z) =0$, which is a contradiction.
\end{example}

\begin{definition}\label{def:XandA}
Choose a closed smooth $4$-manifold $A$ satisfying conditions \ref{it:Aspherical}, \ref{it:Fram}, \ref{it:Centre}, and~\ref{it:Tors}, above. Define a closed smooth $4$-manifold
\[
X := A \# g(S^2 \times S^2)
\]
with $g$ is chosen large enough that $\ks\colon Q(X) \to H_2(X;\Z/2)$
is an isomorphism; such a $g$ exists by Theorem \ref{thm:main-realise}.
Note that $X$ inherits properties~\ref{it:Fram} and~\ref{it:Tors} from $A$.
\end{definition}

In this case the Postnikov truncation map
$$\mft \colon X \lra \BB\pi \simeq A$$
may be modelled as the map which collapses $g(S^2\times S^2)\sm\operatorname{Int}(D^4)) \subseteq X$, so it has degree $\pm 1$.

\subsection{Outline of proof}
Let $X$ and $A$ be instances of the closed, oriented, connected $4$-manifolds from Section~\ref{sub:the-manifold-X}.
We consider the following diagram, whose terms will be defined below.
After that, we will prove the theorem assuming commutativity. Then the majority of the rest of this section will be devoted to proving that it commutes.

\begin{equation}\label{diagram:master-diagram}
\begin{tikzcd}[column sep=1em]
\pi_1(\pseudoHomeo^+(X))\ar[dd] \ar[r, "\alpha"]  & Q(X) \ar[r, "\beta"] \ar[d,"\cong", "\Sm"']
\arrow[ddrr,
"\KS","\cong"',
rounded corners,
to path={ -- ([xshift=4ex, yshift=-4ex]\tikztostart.east)
-- ([yshift=11.6ex]\tikztotarget.north)
-- (\tikztotarget)
\tikztonodes}]
& \pi_0(\pseudoDiffeo^+(X)) \arrow[d, phantom, ""{coordinate, name=Z}] \ar[r, "\gamma"] & \pi_0(\pseudoHomeo^+(X))
\arrow[dd, phantom, ""{coordinate, name=Y}]
\\
\arrow[r, phantom, "\eqref{diagram:Lshaped}"]& {[\tfrac{X \times I \times I}{\partial}, \BB(\Top/\OO)]} \ar[r, "\cong"] & {[X_+, \Omega(\Top/\OO)]} \ar[d,"\cong"] & \\
\pi_1(\hAut^+(X)) \arrow[dr, phantom, "\eqref{eq:etaUmkehr}"]\ar[r,"\eta"] \ar[d,"\mft_*"] & {[X_+, \Omega(\G/\PL)]} \ar[d,"\mft_!"]  & {[X_+, \Omega(\Top/\PL)]} \ar[l] \ar[d,"\mft_!"] \ar[r,"\cong", "PD"'] & H_2(X;\Z/2)  \ar[d,twoheadrightarrow,"\mft_*"]    \\
\pi_1(\hAut^+(A)) \ar[r,"\eta"] &  {[A_+, \Omega(\G/\PL)]} \ar[drr,hookleftarrow,"j", end anchor={[yshift=-2ex]}] & {[A_+, \Omega(\Top/\PL)]} \ar[l] \ar[r,"\cong", "PD"'] & H_2(A;\Z/2) \ar[d,twoheadrightarrow,"\delta_*"]\\
\pi_1(\Diffeo^+(A)) \ar[u,twoheadrightarrow] \ar[ur,"0"'] & & & { \ker\Big(
\begin{array}{c}H_1(A;\Z)\xrightarrow{\cdot 2} \\ \,\,\,\,H_1(A;\Z)\end{array} \Big) }
\end{tikzcd}
\end{equation}

\begin{itemize}[leftmargin=*]
\item The top row is the exact sequence of groups from \eqref{eqn:key-long-exact-seq}.
\item The diagram consists of groups and homomorphisms, and in fact all but possibly $\pi_0(\pseudoDiffeo^+(X))$ and $\pi_0(\pseudoHomeo^+(X))$ are abelian groups.
For this we consider the spaces $\BO$, $\BPL$,~$\BTop$,~$\BG$ as infinite loop spaces via Whitney sum~\cite{Boardman-Vogt}. The spaces $\Top/\OO$, $\Top/\PL$ and $\G/\PL$ inherit compatible infinite loop space structures coming from taking fibres of the infinite loop maps~$\BO\to \BB\PL \to \BB\Top \to \BB\G$.
\item The upper map $\eta$ is defined by the composition
\[
\pi_1(\hAut^+(X))\cong \pi_0(\hAut^+_\partial(X\times I))\lra \mathcal{S}^{\PL}_\partial(X\times I)\xrightarrow{\eta_{\PL}} [\tfrac{X\times I}{\partial},\G/\PL]\cong[X_+, \Omega(\G/\PL)],
\]
where $\eta_{\PL}$ is the $\PL$ normal invariant. The lower map $\eta$ is defined similarly, with $A$ in place of~$X$. Using the $\PL$ surgery sequence for manifolds of the form $Y \times I$ justifies that these normal invariant groups are abelian and that $\eta$ is indeed a homomorphism.
\item We use that $\Top/\PL \simeq K(\Z/2,3)$ to obtain $\Omega(\Top/\PL) \simeq K(\Z/2,2)$. Thus we identify $[Y_+,\Omega(\Top/\PL)] \cong H^2(Y;\Z/2)$ for $Y \in \{X,A\}$. With respect to this identification the maps~$PD$ correspond to ordinary Poincar\'{e} duality.
\item By our choice of $g$ in the definition of $X$, and Theorem~\ref{thm:main-realise}, the map $\KS$
% and  $\Sm$
is surjective. It is injective (without need to stabilise) by Theorem~\ref{thm:main-invariant}, and hence is an isomorphism.
\item The right-most map $\mft_*$ is induced by the Postnikov truncation $\mft\colon X \to \BB\pi \simeq A$. This is easily seen to be surjective from the fact that there exists a model for $\BB\pi$ arising from taking a cell complex for $X$ and then adding cells of dimension at least three.
\item By functoriality of Postnikov truncation, every homotopy automorphism of $X$ induces a homotopy automorphism of its truncation $\BB\pi \simeq A$, yielding the map $\mft_* \colon \pi_1(\hAut(X)) \to \pi_1(\hAut(A))$ on the left.
\item The map $\delta_*$ is the homology Bockstein for the coefficient sequence $0\to \Z\overset{\cdot 2}\to\Z\to\Z/2\to 0$.  The fact that $\delta_*$ is surjective follows from the Bockstein long exact sequence.

\item The isomorphism in the second row is given by the composition
\[
\big[\tfrac{X \times I \times I}{\partial}, \BB(\Top/\OO)\big] \cong [\Sigma^2 X_+,\BB(\Top/\OO)]  \cong [X_+,\Omega^2\BB(\Top/\OO)] \cong  [X_+, \Omega(\Top/\OO)].\]

\item In~\cref{lem:bockstein} below, the map $j$ in the diagram will be defined and proven to be injective, and the lower right triangle will be proven to commute.

\item We discuss the bottom left triangle. Recall from Section~\ref{sec:block} that the geometric realisation of $\hAut^+(A)$ is homotopy equivalent to $\mathrm{hAut}^+(A)$. Combining with property \ref{it:Centre} of the manifold~$A$, this implies the map~$\pi_1(\Diffeo^+(A))\twoheadrightarrow \pi_1(\hAut^+(A))$ is surjective. An element in the image of the up-right composition in the bottom left triangle is the $\PL$ normal invariant of a diffeomorphism of $A\times I$, and is thus trivial, showing that the bottom left triangle commutes.

\item The region involving the map $\KS$ commutes by definition of the map $\KS$.
Consequently, the map~$\Sm$ is an isomorphism.

\item The maps $\mft_!$ are defined as umkehr maps using Poincar\'{e} duality. This exploits the fact that stably framed manifolds are $\bS$-oriented and thus satisfy Poincar\'{e} duality in $E$-theory for any spectrum $E$. This is discussed in detail in Section~\ref{sec:stabfrm}, below. Commutativity of the central  square involving the maps $\mft_!$, and the right-hand square involving $\mft_!$ and $\mft_*$, is  discussed in Section~\ref{sec:stabfrm}, the latter in Lemma~\ref{lem:righthandsquare}.

\item The L-shaped region labelled \eqref{diagram:Lshaped} will be shown to commute in Section~\ref{sub:mapping-into-PL-normal-invariants}.

\item The square labelled \eqref{eq:etaUmkehr} will be shown to commute in Section~\ref{sub:functoriality-PL-normal-wrt-Post-truncation}. This will require several pages and takes up the majority of the rest of the proof.
\end{itemize}

\begin{lemma}\label{lem:bockstein}
For any closed $4$-manifold $Y$, the natural map $[Y_+,\Omega(\Top/\PL)]\to [Y_+,\Omega(\G/\PL)]$ factors as
\[
{[Y_+, \Omega (\Top/\PL)]} \xrightarrow{PD,\cong} H_2(Y;\Z/2) \overset{\delta_*}\twoheadrightarrow \ker\big(H_1(Y;\Z)\xrightarrow{\cdot 2}H_1(Y;\Z)\big) \overset{j}{\hookrightarrow} {[Y_+, \Omega (\G/\PL)]}.
\]
where as above $\delta_*$ denotes the homology Bockstein for the sequence $0\to\Z \overset{\cdot 2}\to \Z\to\Z/2\to0$, and $j$ is an injective map $($described in the course of the proof$)$.
\end{lemma}

\begin{proof}
The natural map $[\Sigma Y_+, \Top/\PL]\to [\Sigma Y_+, \G/\PL]$ factors as
\[
{[\Sigma Y_+, \Top/\PL]} \cong H^3(\Sigma Y_+;\Z/2) \twoheadrightarrow H^3(\Sigma Y_+;\Z/2)/H^3(\Sigma Y_+;\Z) \hookrightarrow {[Y_+, \Omega (\G/\PL)]},
\]
by~\cite[Annex~C, Theorem~15.1]{Kirby-Siebenmann:1977-1}. Next we use the identification $H^3(\Sigma Y_+; \Z/2)\cong H^2(Y;\Z/2)$, and by the Bockstein sequence for $0\to\Z\overset{\cdot 2}\to\Z\to\Z/2\to 0$, we identify \[H^3(\Sigma Y_+; \Z/2)/H^3(\Sigma Y_+; \Z)\cong H^2(Y; \Z/2)/H^3(Y; \Z) \cong \im(\delta^*),\] where $\delta^* \colon H^2(Y;\Z/2) \to H^3(Y;\Z)$ is the cohomology Bockstein. There is thus an exact sequence
\[
{[Y_+, \Omega(\Top/\PL)]} \cong H^2(Y;\Z/2) \overset{\delta^*}\twoheadrightarrow \im(\delta^*) \hookrightarrow {[Y_+, \Omega (\G/\PL)]}.
\]
By naturality of the Bockstein exact sequence under Poincar\'{e} duality, we have an identification $\im\big(\delta^*\colon H^2(Y;\Z/2) \to H^3(Y;\Z)\big) \cong \im \big(\delta_* \colon H_2(Y;\Z/2) \to H_1(Y;\Z)\big)$, and hence an exact sequence
\[
{[Y_+, \Omega(\Top/\PL)]} \cong H_2(Y;\Z/2) \overset{\delta_*}\twoheadrightarrow \im(\delta_*) \hookrightarrow {[Y_+, \Omega (\G/\PL)]}.
\]
Finally, $\im(\delta_*) = \ker\big(H_1(Y;\Z)\xrightarrow{\cdot 2}H_1(Y;\Z)\big)$ by the homology Bockstein long exact sequence.
\end{proof}

\begin{proof}[Proof of Theorem~\ref{thm:main-example} assuming that \eqref{diagram:master-diagram} commutes]
Recall that we constructed $X$ in Section~\ref{sub:the-manifold-X}.
Theorem~\ref{thm:main-example} is the statement that the map $\gamma$ in \eqref{diagram:master-diagram} is not injective.
 This is equivalent to the statement that there exists $F \in Q(X)$ with $\beta(F)$ nontrivial.

 As $H_1(\pi ; \Z) \cong H_1(A;\Z)$ has an element of order two, it follows that  $\ker \big(H_1(A;\Z)\xrightarrow{\cdot 2} H_1(A;\Z)\big)$ is nontrivial.
As indicated in the diagram, the composition \[\delta_*\,{\circ}\,\mft_*\,{\circ}\, \KS \colon Q(X) \lra \ker \big(H_1(A;\Z)\xrightarrow{\cdot 2} H_1(A;\Z)\big)\] is surjective. Choose a topological pseudo-isotopy $F \colon X \times I \to X \times I$ in $Q(X)$ with $\delta_*\,{\circ}\,\mft_*\,{\circ}\, \KS(F) \neq 0$.
We need to argue that its image under $\beta$ is nontrivial.  Supposing instead that $\beta(F)$ is trivial, we deduce that $F$ lies in the image of $\alpha$, and hence that the clockwise composition $\delta_*\,{\circ}\,\mft_*\,{\circ}\, \KS \,\circ \, \alpha$ is nontrivial. Since $j$ is injective, moreover the composition
\[j \,\circ \,\delta_*\,{\circ}\,\mft_*\,{\circ}\, \KS \,\circ \, \alpha \colon \pi_1(\pseudoHomeo^+(X)) \lra [A_+,\Omega(\G/\PL)]\] is nontrivial. By commutativity of the diagram \eqref{diagram:master-diagram}, this is equal to the anti-clockwise route
\[\pi_1(\pseudoHomeo^+(X)) \lra \pi_1(\hAut^+(X)) \overset{\mft_*}\lra  \pi_1(\hAut^+(A)) \overset{\eta}\lra  [A_+,\Omega(\G/\PL)]. \]
But commutativity of the bottom left triangle and surjectivity of $\pi_1(\Diffeo^+(A)) \twoheadrightarrow \pi_1(\hAut^+(A))$ implies that the bottom map $\eta$ is trivial, and hence the anti-clockwise route is trivial.  We obtain a contradiction, and so $\beta(F) = F|_{X \times \{1\}} \colon X \to X$ is our desired diffeomorphism topologically but not smoothly pseudo-isotopic to the identity.

We now argue that $\beta(F)$ is not even stably smoothly pseudo-isotopic to the identity. The definition of a stabilisation of $\beta(F)$ requires us to first isotope $\beta(F)$ to fix a $4$-ball in which to perform the stabilisations. We begin by arguing that $F$ may be chosen so that a particular stabilisation of $\beta(F)$ is not smoothly pseudo-isotopic to the identity, then extend this to show that any choice of stabilisation for this $\beta(F)$ is not smoothly pseudo-isotopic to the identity.

First, choose some $D^4\subseteq A$ and recall from Remark~\ref{rem:observation} that $F$ may be assumed to restrict to the identity on $D^4 \times I\subseteq X\times I$.
Write $X^\circ:=X\sm\mathring{D}^4$, let $g' > g$ and write $Z := (g'-g) (S^2 \times S^2) \sm \mathring{D}^4$.   Note that $X^\circ \cup_{S^3} Z = A \# g'(S^2 \times S^2)$, and denote this manifold by $X'$.
Thus by extending~$F$ restricted to $X^\circ \times I$ by the identity on $Z \times I$, we obtain a pseudo-isotopy $F' \in Q(X')$.
We studied such connected sums of pseudo-isotopies in Lemma~\ref{lem:cheaperlemma} where, in the language of that lemma,~$F'=F\#\Id_{Z\times I}$. As $\KS(F)$ is nontrivial, Lemma~\ref{lem:cheaperlemma} shows $\KS(F')$ is nontrivial.
As a result, we can run the entire proof of Theorem~\ref{thm:main-example} with $X$ and $F$ replaced with $X'$ and $F'$, to see that $\beta(F')$ is also nontrivial.

Finally, so long as $g>0$ in the definition of $X$ (which we are free to assume), the hypotheses of Proposition~\ref{prop:stablemma} are satisfied. Thus the stabilisation performed in the previous paragraph must be the unique stabilisation of $\beta(F)$ that is topologically pseudo-isotopic to the identity. In particular, no other stabilisation is smoothly pseudo-isotopic to the identity.
 \end{proof}

\subsection{Stably framed manifolds are $\bS$-oriented}\label{sec:stabfrm}

Recall that for the ring spectrum $\mathbb{S}$, the corresponding homology theory is stable homotopy $\pi^s_*(-)$. We remind the reader again that in stable homotopy theory, homology theories are reduced by default, so for example the stable stem $\pi^s_*$ is given by~$\pi^s_*(S^0) = \pi^s_*(\pt_+)$.

It is well-known that stably framed manifolds have fundamental classes with respect to the sphere spectrum $\bS$ and so satisfy Poincar\'{e} duality with respect to any generalised (co)homology theory, by the discussion in Section \ref{subsec:orientation}. As we will need the details of this correspondence, we give a short proof and explain some consequences.

First we prove a lemma that gives a simple way to decide if a stable homotopy class is a fundamental class.

\begin{lemma}\label{lem:DetectSTheoryFundamentalClass}
Let $Y$ be a $d$-dimensional compact topological manifold with $($possibly nonempty$)$ boundary. A class $\alpha \in \pi_d^s(Y/\partial Y)$ is an $\mathbb{S}$-theory fundamental class if and only if its Hurewicz image $h(\alpha) \in \widetilde{H}_d(Y/\partial Y ; \mathbb{Z})$ is an integral homology fundamental class.
\end{lemma}

\begin{proof}
If $q\colon Y/\partial Y \to Y/(Y - \mathrm{int}(D^d)) \cong D^d/\partial D^d \cong S^d$ is a coordinate chart collapse map then $q_*(\alpha) \in \pi_d^s(S^d)$. The Hurewicz map $h\colon \pi_d^s(S^d) \to \widetilde{H}_d(S^d ; \mathbb{Z})$ is an isomorphism, so $q_*(\alpha)$ corresponds to $\pm 1 \in \pi_0^s(S^0)$ under the suspension isomorphism if and only if $h(q_*(\alpha)) = q_*(h(\alpha))$ corresponds to $\pm 1 \in \widetilde{H}_0(S^0 ; \mathbb{Z})$ under the suspension isomorphism, i.e.\ if and only if $h(\alpha)$ is an integral homology fundamental class.
\end{proof}

Let $X$ and $A$ be as in Definition~\ref{def:XandA}. For some $N \gg 0$, let $e\colon X\hookrightarrow \R^{4+N}$ be a smooth embedding, and write $\nu_e^\Diff$ for the normal bundle. Apply the functor $\Sigma^{-N}\Sigma^\infty$ to the corresponding Thom collapse map $S^{4+N}\to \Th(\nu_e^\Diff)$ to obtain a Thom collapse map of spectra
\[
c\colon \bS^4 \lra \Th(\nu_{X}^\Diff).
\]
The target here is of course the same as $\Th(\nu_{X}^\G)$, for $\nu_X^\G$ the Spivak fibration of $X$: the Thom spectrum of a vector bundle only depends on the underlying spherical fibration, and we may freely interchange them.  Assume $N$ is large enough that the stable framing on $\nu_X^\Diff$ is represented by a framing of $\nu_e^\Diff$ by $\nu_e^\Diff \cong \R^N\times X$. Passing back to the stable version, and its Thom spectrum, we have
\begin{equation}\label{eq:equivalence}
\Th(\nu_X^\Diff)=\Sigma^{-N}\Sigma^\infty\Th(\nu_e^\Diff)\cong \Sigma^{-N}\Sigma^\infty\Th(\R^N\times X)\simeq \Sigma^\infty X_+.
\end{equation}
Under this sequence of equivalences of spectra, the map $c$ corresponds to a class $[X]_\bS \in \pi_4^s(X_+)$.

\begin{lemma}\label{lem:orientations}
The class $[X]_\bS \in \pi_4^s(X_+)$ is a fundamental class in stable homotopy theory. Furthermore, the image $[A]_\bS := \mft_*[X]_\bS \in \pi_4^s(A_+)$ under the Postnikov truncation map $\mft\colon X\to A$ is also a fundamental class in stable homotopy theory.
\end{lemma}

\begin{proof}
The group $H_4(\bS^4;\Z) \cong [\bS^0,\bS^4 \wedge \mathrm{H}\Z] \cong [\bS^0,\Sigma^{\infty}S^4_+ \wedge \mathrm{H}\Z]$ is isomorphic to $\Z$, generated by the class $[\bS^4]$, which is the class that maps to the fundamental class of the $4$-sphere under these isomorphisms.
It is a property of the Thom collapse map that $c_*[\bS^4] \in H_4(\Th(\nu_{X}^\Diff);\Z)$ is a generator. So under the Thom isomorphism (with respect to the orientation given by the stable framing) this class corresponds to an integral homology fundamental class $[X] \in H_4(X;\Z)$. This is the image of~$[X]_\bS$ under the Hurewicz map, so~$[X]_\bS$ is a fundamental class by Lemma~ \ref{lem:DetectSTheoryFundamentalClass}.

For the second statement, we observed already that $\mft \colon X \to \BB\pi \simeq A$ has degree $\pm 1$, meaning that $\mft_*[X] \in H_4(A;\Z)$ is an integral homology fundamental class.  The class $\mft_*[X]$ is the Hurewicz image of $[A]_{\bS}:=\mft_*[X]_\bS \in \pi_4^s(A_+)$, so again by Lemma \ref{lem:DetectSTheoryFundamentalClass} it  follows that  $[A]_{\bS}$ is also a stable homotopy fundamental class.
\end{proof}

Every spectrum is an $\bS$-module, so an $\bS$-orientation on a manifold induces an orientation in any generalised (co)homology theory. This means that for any spectrum $F$ the manifolds $X$ and $A$ have Poincar\'{e} duality isomorphisms (see Subsection~\ref{subsec:orientation}). Using these, any map of $\bS$-oriented manifolds $f\colon X\to A$ determines an umkehr map
\[
f_!\colon {F}^r(X)\xrightarrow{PD}{F}_{4-r}(X)\xrightarrow{f_*} {F}_{4-r}(A)\xrightarrow{PD^{-1}} {F}^r(A).
\]

\begin{lemma}\label{lem:UmkehrSplit}
If $f_*([X]_\bS) = [A]_\bS$ then the umkehr map satisfies the identity $f_! \circ f^* = \mathrm{Id}$.
\end{lemma}
\begin{proof}
Use respectively the definition of $f_!$, the projection formula, and the given identity $f_*([X]_\bS) = [A]_\bS$, to obtain
\[[A]_\bS \frown (f_!\circ f ^*)(-) = f_*([X]_\bS \frown f^*(-)) = f_*([X]_\bS) \frown (-) = [A]_\bS \frown (-).\] The lemma then follows from the fact that $[A]_\bS \frown (-)$ is an isomorphism.
\end{proof}

We  use the unkehr construction to define the maps $\mft_!$ in diagram~\eqref{diagram:master-diagram}. Consider the Postnikov truncation map $\mft\colon X\to A$.
Recall the spaces $\Top/\PL$ and $\G/\PL$ were given compatible infinite loop space structures coming from taking fibres of the infinite loop maps~$\BB\PL \to \BB\Top \to \BB\G$.
The umkehr maps~$\mft_!$ in diagram~\eqref{diagram:master-diagram} are formed using the spectra corresponding to these infinite loop space structures. Thus the central square involving~$\mft_!$, in diagram~\eqref{diagram:master-diagram}, commutes by naturality of the umkehr construction with respect to infinite loop maps.

\begin{lemma}\label{lem:righthandsquare}
The right-hand square involving~$\mft_!$, in diagram~\eqref{diagram:master-diagram}, commutes.
\end{lemma}

\begin{proof}
The map $\mft_! \colon [X_+, \Omega(\Top/\PL)] \to [A_+, \Omega(\Top/\PL)]$ is the umkehr map for the cohomology theory represented by the connective spectrum $\Sigma^{-1} (\mathrm{top}/\mathrm{pl})$ obtained from the infinite loop space $\Omega(\Top/\PL)$. This space is a $K(\Z/2,2)$, so the spectrum $\Sigma^{-1}( \mathrm{top}/\mathrm{pl})$ is the Eilenberg--MacLane spectrum $\Sigma^2\mathrm{H}(\Z/2)$. So this umkehr map is identified with the usual one $\mft_! \colon H^2(X;\Z/2) \to H^2(A;\Z/2)$, which is indeed Poincar{\'e} dual to $\mft_* \colon H_2(X;\Z/2) \to H_2(A;\Z/2)$.
\end{proof}

\subsection{Mapping into the $\PL$ normal invariants}\label{sub:mapping-into-PL-normal-invariants}
We will consider the following diagram.
\begin{equation}\label{diagram:Lshaped}
\begin{tikzcd}[column sep = 1em]
\pi_1(\pseudoHomeo (X)) \rar{\alpha} \arrow[dd] & Q(X) \rar{\Sm} & {[\tfrac{X \times I \times I}{\partial}, \BB(\Top/\OO)]} \arrow[r, "\cong"] & {[\tfrac{X\times I}{\partial}, \Top/\OO]} \ar[d] \ar[r, "\cong"]& {[X_+, \Omega(\Top/\OO)]} \ar[d] \\
& & & {[\tfrac{X\times I}{\partial}, \Top/\PL]} \ar[d] \ar[r, "\cong"] & {[X_+, \Omega(\Top/\PL)]} \ar[d]  \\
\pi_1(\hAut (X))\ar[rr] \arrow[rrr, bend right=15, "\eta"]& & \mathcal{S}^{\PL}_\partial(X \times I) \rar{\eta_\PL} & {[\tfrac{X\times I}{\partial}, \G/\PL]}\ar[r, "\cong"] & {[X_+, \Omega(\G/\PL)]}
\end{tikzcd}
\end{equation}

The first map in the bottom row involves the map that considers a loop of homotopy automorphisms of $X$ as a $\PL$-manifold structure on $X \times I$. The second map in the bottom row is the $\PL$ normal invariant map $\eta_\PL \colon \mathcal{S}^{\PL}_\partial(X \times I) \to [\Sigma X_+, \G/\PL]$ followed by the $\Sigma$-$\Omega$ adjunction. The map~$\eta$ is defined so that the semi-circle commutes.

\begin{lemma}\label{lem:newdiagram}
Diagram~\eqref{diagram:Lshaped} commutes.
\end{lemma}

 \begin{proof}
 First, the two right-most squares commute by naturality of $\Sigma$-$\Omega$ adjunction with respect to maps of spectra, and the map $\eta$ was defined to make the semicircle commute. The commutativity of the left rectangle is essentially tautological; here are the details.

 Write $\varphi$
 for a loop in~$\pseudoHomeo (X)$ (based at the identity map). This is represented by a homeomorphism $\psi\in\Homeo _\partial (X\times I)$. Write
\[
\nu^{\Diff}_{X\times I}\colon X\times I\lra \BO, \quad \nu^{\PL}_{X\times I}\colon X\times I\lra \BPL\quad\text{and}\quad\nu_{X\times I}^{\Top}\colon X\times I\lra \BTop
\]
for the stable normal bundle, then its underlying stable $\PL$-bundle, and then stable microbundle. There is a homotopy from $\nu_{X\times I}^{\Top}$ to $(\psi^{-1})^*\nu_{X\times I}^{\Top}$, induced by the stable normal microbundle of the mapping cylinder of $\psi$. Since we model $\BO \to \BTop$ by a fibration, this lifts to a homotopy from~$\nu^{\Diff}_{X \times I}$ to some lift of the tangent microbundle $\nu_{X\times I}^{\Top}$, and the lift may be identified with~$(\psi^{-1})^*\nu^{\Diff}_{X\times I}$.

Consider the clockwise composition. We have $\alpha(\varphi)=[\psi] \in Q(X)$. Under the identifications on the top row, the class $\Sm([\psi])$ is the difference class $d((\psi^{-1})^*\nu_{X\times I}^{\Diff},\nu_{X\times I}^{\Diff})$, measuring the difference of these lifts of $\nu_{X\times I}^{\Top}$ from $\Top$ to $\OO$, relative to the fixed lift on $\partial(X\times I)$. Under the subsequent forgetful maps in the column, which send $\nu_{X\times I}^{\Top}$ to the Spivak normal fibration of $X\times I$, we obtain a $\PL$ normal invariant, which is again a difference class, now measuring the difference of the lifts from $\G$ to $\PL$
\begin{equation}\label{eq:tautology}
d((\psi^{-1})^*\nu_{X\times I}^{\PL},\nu_{X\times I}^{\PL})\in \left[\tfrac{X\times I}{\partial (X\times I)},\G/\PL\right].
\end{equation}

Now consider the anticlockwise composition. The first vertical map is forgetful, and the element of the $\PL$ structure set we obtain from $\varphi$ is again just the (class of) $\psi$. The $\PL$ normal invariant obtained from an element of the structure set is just the difference of the lifts to $\PL$ of the Spivak normal fibration. Thus it is the element displayed in~\eqref{eq:tautology}. This shows the diagram commutes, as claimed.
 \end{proof}

\subsection{Functoriality of $\PL$ normal invariants with respect to Postnikov truncation}\label{sub:functoriality-PL-normal-wrt-Post-truncation}

We will prove that the diagram
\begin{equation}\label{eq:etaUmkehr}
\begin{tikzcd}
\pi_1(\hAut (X)) \rar{{\eta}} \dar{\mft_*} & {[X_+, \Omega(\G/\PL)]} \dar{\mft_!}\\
\pi_1(\hAut (A)) \rar{{\eta}}  & {[A_+, \Omega(\G/\PL)]}
\end{tikzcd}
\end{equation}
commutes. This will complete the proof that diagram~\eqref{diagram:master-diagram} commutes and thus the proof of~Theorem~\ref{thm:main-example}. The proof that diagram~\eqref{eq:etaUmkehr} commutes will crucially use that the manifolds are 4-dimensional:~one should not expect it to hold in higher dimensions.

\subsubsection{Factoring through the tangential structure set}
To prove that diagram \eqref{eq:etaUmkehr} commutes we will factor the map $\eta$, using that $X$ (and $A$) are stably framed.
The approach is based on the ``tangential structure set'' idea of Madsen--Taylor--Williams~\cite{MR593058}, in which those authors consider an enriched structure set $\mathcal{S}^t_\partial(X \times I)$ where objects are not merely homotopy equivalences, but are also covered by maps of their (stable) tangent bundles. The purpose of this subsection is to establish the commutative diagram
\begin{equation}
\label{eq:MTWdiagram}
\begin{tikzcd}
\pi_1(\hAut (X)) \ar[r] \ar[d,"\wh{\eta}"] & \mathcal{S}^t_\partial(X \times I) \dar{\eta_t} \rar & \mathcal{S}^\PL_\partial(X \times I) \dar{\eta_\PL}\\
{[X_+, \Omega^{\infty+1} \bS^0]}\ar[r,"\cong"', "{(T_{0,1})_*}"] & {[X_+, \Omega \G]}  \rar & {[X_+, \Omega (\G/\PL)].}
\end{tikzcd}
\end{equation}
In this diagram, the homomorphism $\eta$, from earlier, is the right-right-down composition.
We will see from the development that the right-hand square is defined and commutes without the assumption that~$X$ is stably framed, but the top left horizontal arrow and the map $\widehat{\eta}$ require the stable framing of $X$ to even be defined.

We begin by defining the bottom left horizontal map.

\begin{definition}
Let $u$ and $v$ be elements of a loop space~$\Omega Z$. There is a homotopy equivalence of based spaces
\[
T_{u,v}\colon (\Omega Z,u) \lra (\Omega Z,v)
\]
given by composing the concatenating map $(\Omega Z,u) \lra (\Omega Z,vu^{-1}u);\;\; \gamma\mapsto vu^{-1}\gamma$ with the homotopy inverse of the concatenating map $(\Omega Z,v) \to (\Omega Z,vu^{-1}u);\;\; \delta \mapsto \delta u^{-1}u$.
We call this homotopy equivalence \emph{translating the basepoint} from $u$ to $v$.
\end{definition}

\begin{remark}\label{rem:basepoint}
The infinite loop space $\Omega^{\infty} \bS^0$ is the colimit $\mathrm{colim}_k\Omega^kS^k$ and therefore can be thought of as a colimit of based maps from a $k$-sphere to itself. Prominent basepoint choices for this infinite loop space include the identity map $1$ and the constant map $0$; by default, the basepoint of $\Omega^{\infty} \bS^0$ is~$0$. Recall that there is an equality of spaces $\mathrm{SG} = \Omega^\infty_{1} \bS^0$, where $\mathrm{SG}\subseteq \G$ and $\Omega^\infty_{1} \bS^0$ denote the components containing the basepoint $1$ (see e.g.~\cite[Corollary~3.8]{MadsenMilgram}). Thus translating the basepoint $0\in \Omega^\infty \bS^0$  to the basepoint $1\in \mathrm{SG}\subseteq \G$ determines a based homotopy equivalence $T_{0,1}\colon (\Omega^\infty \bS^0,0)\simeq (\mathrm{SG},1)$. Consequently, there is a based homotopy equivalence of loop spaces ~$\Omega T_{0,1}\colon \Omega^{\infty+1}\bS^0\simeq \Omega\G$.
\end{remark}

\begin{construction}\label{construction:eta}
We construct a homomorphism of abelian groups
\[
\widehat{\eta}\colon\pi_1(\hAut (X)) \lra \left[X_+, \Omega^{\infty+1} \bS^0 \right]
\]
crucially using that $X$ is stably framed, so that \eqref{eq:equivalence} is available.
A loop in the Kan semi-simplicial group $\hAut (X)$ is a homotopy equivalence $\phi \colon X \times I \to X \times I$ restricting to the identity on $\partial(X \times I)$. By gluing $X \times \{0\}$ to $X \times \{1\}$ in the domain and projecting to $X$ in the codomain, we obtain a map $\psi \colon X \times S^1 \to X$ which restricts to the identity on $X \times \{1\}$.
Write $\mathrm{pr}_X \colon X \times S^1 \to X$ for the projection map.
Recall
that for maps of spectra $f,g\colon U\to V$, we write $f-g\colon U\to V$ for a representative of the homotopy class $[f]-[g]\in[U,V]$.
On suspension spectra the difference
\[
\Sigma^\infty\psi - \mathrm{pr}_X \colon \Sigma^\infty ( X\times S^1)_+ \lra \Sigma^\infty X_+
\]
is canonically homotopically trivial when restricted to $\Sigma^\infty(X\times\{1\})_+$, so descends to a map
\begin{equation}\label{eq:PhiMinusId}
[\Sigma^\infty\psi - \mathrm{pr}_X] \colon \Sigma^{\infty+1} X_+ = \Sigma^\infty \tfrac{ (X\times S^1)_+}{(X\times \{1\})_+} \lra \Sigma^\infty X_+.
\end{equation}

As $X$ is stably framed, it has an $\bS$-theory fundamental class $[X]_{\bS} \in [\bS^4,\Sigma^{\infty}X_+]$. We consider the image $[\Sigma^\infty\psi - \mathrm{pr}_X]_*(\Sigma [X]_\bS) \in [\bS^5, \Sigma^\infty X_+]$ of the suspension $\Sigma [X]_\bS \in [\bS^5,\Sigma^{\infty+1} X_+]$, and  define~$\widehat{\eta}(\phi)$ to be element of  $[X_+, \Omega^{\infty+1} \bS^0 ]$  corresponding to  $[\Sigma^\infty\psi - \mathrm{pr}_X]_*(\Sigma [X]_\bS)$
 under the isomorphisms
\[
[X_+, \Omega^{\infty+1} \bS^0 ] \cong [\Sigma^{\infty+1} X_+, \bS^0] \underset{AD}{\cong} [\bS^5, \Th(\nu_X^\G)] \underset{\eqref{eq:equivalence}}{\cong} [\bS^5, \Sigma^\infty X_+]
\]
given by adjunction, Atiyah duality and \eqref{eq:equivalence}.
Here the Atiyah duality isomorphism is given by applying the second version of Atiyah duality  in Section~\ref{sec:Atiyah} with $E = \bS$, $\partial X = \emptyset$, and $r=-1$, together with the fact that $\Th(\nu^{\Top}_X) \simeq \Th(\nu^{\G}_X)$, to obtain
\[
[\Sigma^{\infty+1}X_+,\bS^{0}] \cong [\Sigma^{\infty}X_+,\bS^{-1}] \underset{AD}{\cong} [\bS^0,\Th(\nu_X^\G) \wedge \Sigma^{-1-4}\bS^0] \cong [\bS^5,\Th(\nu_X^\G)].
\]
\end{construction}

\begin{remark}
We argue that $\wh{\eta}$ is a homomorphism of abelian groups.
The domain $\pi_1(\hAut (X))$
of $\wh{\eta}$ is an abelian group because the geometric realisation of $\hAut (X)$ is a group-like topological monoid, so $\pi_1(\hAut (X)) \cong \pi_1(|\hAut (X)|) \cong \pi_2(\BB|\hAut (X)|)$.
 For the codomain, maps into an infinite loop space  always form an abelian group. Identifying $\pi_1(\hAut (X))$ with $\pi_0(\hAut_{\partial}(X \times I))$,
 the group structure on the domain corresponds to stacking, and this corresponds to the group structure on the codomain $[X_+, \Omega^{\infty+1}\bS] \cong [\Sigma X_+, \Omega^\infty\bS] = [\tfrac{X \times I}{\partial(X \times I)}, \Omega^\infty \bS]$ given by the co-$H$-space~$\tfrac{X \times I}{\partial(X \times I)} \simeq \Sigma X_+$.
\end{remark}

\begin{remark}
Intuitively, we could describe $\widehat{\eta}(\phi)$ as measuring how the loop of homotopy equivalences $\phi$ acts upon the fundamental class $[X]_{\bS}\in \pi_4^s(X_+)$.
\end{remark}

Next we describe the tangential structure set and tangential normal invariant. The description begins with a recap of the ordinary ($\PL$) normal invariant, which is the right vertical map of diagram~\eqref{eq:MTWdiagram}. We use the somewhat non-standard description in \cite[Section 2]{MR593058} (see also~\cite[Section 4]{BrumfielMadsen}), with the changes necessary to deal with the fact that we work relative to the boundary whereas that paper does not.

\subsubsection{$\PL$ normal invariants via $f$-maps}
Given a compact space $Y$, define an \emph{$f$-map} $(\nu^q, t, \xi^q)$ over $Y$ to consist of two $\PL$ microbundles~$\nu^q$ and $\xi^q$ of the same dimension~$q$ over $Y$ together with a fibre homotopy equivalence $t\colon S(\nu^q) \xrightarrow{\simeq} S(\xi^q)$ between their underlying spherical fibrations. The $f$-maps $(\nu_i^q, t_i, \xi_i^q)$ for $i=1,2$ are \emph{stably homotopic} if there exist $f$-maps $(\gamma_i, \Id, \gamma_i)$ for $i=1,2$, such that $(\nu_1^q, t_1, \xi_1^q)\oplus (\gamma_1, \Id, \gamma_1)$ and $(\nu_2^q, t_2, \xi_2^q)\oplus (\gamma_2, \Id, \gamma_2)$ are homotopic as $f$-maps. It is shown in \cite[Section 4]{BrumfielMadsen} that the space~$\G/\PL$ classifies $f$-maps, considered up to stable homotopy.

An element of the structure set $\mathcal{S}^\PL_\partial(X \times I)$ is represented by a simple homotopy equivalence~$(\phi, \partial \phi)\colon (M, \partial M) \xrightarrow{\simeq} (X \times I, \partial (X \times I))$ from a $\PL$ manifold, with $\partial \phi$ a $\PL$-isomorphism. Next we recall the details of $\eta_{\PL}(\phi)$.
There is a canonical map, over $\phi$, of $\PL$ microbundles~$\widetilde{\phi}\colon \nu^{\PL}_M\to (\phi^{-1})^*\nu^\PL_M$. This induces a
reduction
\[
c_\phi\colon\bS^{5} \overset{c_M}\lra \Th(\nu^\PL_M)/\Th(\nu^\PL_{\partial M}) \xrightarrow{\Th(\widetilde{\phi})} \Th((\phi^{-1})^*\nu^\PL_M)/\Th(\nu^\PL_{\partial (X \times I)})
\]
that under the connecting map $\partial\colon  \Th((\phi^{-1})^*\nu^\PL_M)/\Th(\nu^\PL_{\partial (X \times I)}) \to \Sigma \Th(\nu^\PL_{\partial (X \times I)})$ in the Puppe sequence agrees with the suspension
\begin{equation}\label{eq:RedOnBdy}
\Sigma c_{\partial (X \times I)} \colon  \bS^5 = \Sigma \bS^4 \lra \Sigma \Th(\nu^\PL_{\partial (X \times I)})
\end{equation}
of the reduction for $\partial (X \times I)$. Further, as $\phi$ is a homotopy equivalence, the reduction $c_\phi$ is compatible with the $\Z$-coefficient fundamental class of $X\times I$ in the sense that $h_*[c_\phi]\frown U_{H\Z}=[X \times I]_{H\Z}\in H_5(X\times I,\partial(X\times I);\Z)$.
This compatibility means that the uniqueness theorem for Spivak fibrations \cite[Theorem I.4.19]{Browder} applies, to give a fibre homotopy equivalence $t_\phi\colon S(\nu^\PL_{X \times I}) \xrightarrow{\simeq} S((\phi^{-1})^*\nu^\PL_M)$ rel.~boundary between the underlying spherical fibrations, unique such that
\begin{equation}\label{eq:FibHEqSquare}
\begin{tikzcd}
\bS^5 \rar{c_M} \dar{c_{X \times I}} & \Th(\nu^\PL_M)/\Th(\nu^\PL_{\partial M}) \dar{\Th(\widetilde{\phi})}\\
\Th(\nu^\PL_{X \times I})/\Th(\nu^\PL_{\partial (X \times I)}) \rar{\Th(t_\phi)} & \Th((\phi^{-1})^*\nu^\PL_M)/\Th(\nu^\PL_{\partial (X \times I)})
\end{tikzcd}
\end{equation}
commutes up to homotopy over $\Sigma \Th(\nu^\PL_{\partial (X \times I)})$. The $f$-map
data $(\nu^\PL_{X \times I}, t_\phi, (\phi^{-1})^*\nu^\PL_M)$ is classified by the element
\[
\eta_\PL(\phi) \in \left[\tfrac{X \times I}{\partial}, \G/\PL\right] \cong \left[X_+,\Omega (\G/\PL)\right].
\]

\subsubsection{Tangential $\PL$ normal invariants via $f$-maps}
Now we develop the central column of diagram~\eqref{eq:MTWdiagram} in parallel to the previous discussion, again following \cite[Section 2]{MR593058} and \cite[Section 4]{BrumfielMadsen}. The space~$\G$ may be considered as classifying a restricted type of $f$-map, consisting of the data $(\nu^q,t)$, where~$\nu^q$ is a single $\PL$ microbundle and $t\colon S(\nu^q) \xrightarrow{\simeq} S(\nu^q)$ is a fibre homotopy automorphism of its underlying spherical fibration, up to stable homotopy. For later use, we expand further on this classification. Given $(\nu^q,t)$, we may take a finite dimensional representative~$\xi$ of the stable inverse to $\nu^q$ and form $(\nu^q,t)\oplus(\xi,\Id)$. The resultant bundle is $\nu^q\oplus\xi\simeq \varepsilon^M$ for some integer $M$, and under this identification $t\ast\Id$ becomes a fibre homotopy equivalence of the trivial spherical fibration $Y\times S^{M-1}\to Y\times S^{M-1}$. Such a fibre homotopy equivalence is described by a map $Y\to \G(M)$. After postcomposing with the inclusion $\G(M)\subseteq \G$, it is the homotopy class of this map which classifies~$(\nu^q, t)$ up to stable homotopy; see~\cite[Section 4]{BrumfielMadsen}.

The \emph{$\PL$ tangential structure set} $\mathcal{S}^t_\partial(X \times I)$ is the (bordism group of) simple homotopy equivalences of pairs $(\phi, \partial \phi) \colon  (M, \partial M) \to (X \times I, \partial(X \times  I))$ from a $\PL$ manifold, with~$\partial \phi$ a $\PL$ isomorphism, and equipped with a $\PL$-microbundle map $\wh{\phi} \colon  \nu_M^\PL \to \nu_{X \times I}^\PL$ covering $\phi$ and restricting to the $\PL$ derivative of $\partial \phi$ over the boundary. The map $\widehat{\phi}$ determines a $\PL$ bundle identification $(\phi^{-1})^*\nu^{\PL}_M \cong \nu^{\PL}_{X \times I}$. Under this identification, the $\PL$ normal invariant classifies the data~$(\nu^\PL_{X \times I}, t_{\phi,\wh{\phi}}, \nu^\PL_{X \times I})$, where $t_{\phi, \wh{\phi}}\colon \nu^\G_{X \times I} \xrightarrow{\simeq} (\phi^{-1})^*\nu^\G_M \cong \nu^\G_{X \times I}$ is the composition whose first map is~$t_\phi$ and whose second is induced by our identification on the $\PL$ level.
There is then defined a \emph{tangential normal invariant}
\[
\eta_t \colon  \mathcal{S}^t_\partial(X \times I) \lra  \left[\tfrac{X \times I}{\partial}, \G\right] \cong [X_+, \Omega \G],
\]
where $\eta_t(\phi, \wh{\phi})\in \left[\tfrac{X \times I}{\partial}, \G\right] \cong [X_+,\Omega \G] $ is the element classifying the data $(\nu^\PL_{X \times I}, t_{\phi,\wh{\phi}})$.

\medskip

At this point we have enough to verify that the tangential $\PL$ normal invariant is a lift of the ordinary $\PL$ normal invariant.

\begin{lemma}
The right-hand square of diagram~\eqref{eq:MTWdiagram}
commutes.
\end{lemma}

\begin{proof}
Given $(\phi, \wh{\phi})$, the top horizontal arrow is the map that forgets $\wh{\phi}$ and thus the clockwise route sends $(\phi, \wh{\phi})$ to $(\nu^\PL_{X \times I}, t_\phi, (\phi^{-1})^*\nu^\PL_M)$. The bottom horizontal arrow maps $(\nu^q, t)\mapsto (\nu^q, t, \nu^q)$ and hence the anticlockwise route gives $(\nu^\PL_{X \times I}, t_{\phi, \wh{\phi}}, \nu^\PL_{X \times I})$. Using $\wh{\phi}$, we see these $f$-maps are isomorphic.
\end{proof}

To prove the left-hand square of diagram~\eqref{eq:MTWdiagram} commutes, we will first describe another point of view on the element $\eta_t(\phi, \wh{\phi})$. Given $(\phi, \wh{\phi})\in \mathcal{S}^t_\partial(X \times I)$, take the diagram
\begin{equation*}
\begin{tikzcd}
\bS^5 \rar{c_{M}} \arrow[rrd, swap, bend right = 10,   "\Sigma c_{\partial(X \times I)}"] & \Th(\nu^\PL_{M})/\Th(\nu^\PL_{\partial M}) \rar{\Th(\wh{\phi})} & \Th(\nu^\PL_{X \times I})/\Th(\nu^\PL_{\partial (X \times I)}) \dar{\partial} \\
 & & \Sigma \Th(\nu^\PL_{\partial (X \times I)}),
\end{tikzcd}
\end{equation*}
apply the functor $\SW\circ\Sigma^{-5}$,
to obtain
\begin{equation*}
\begin{tikzcd}
\bS   & \mbox{\hspace{30ex}} & \Sigma^\infty(X \times I)_+ \arrow[ll, "\SW\circ\Sigma^{-5}(\Th(\wh{\phi})\circ c_M)"'] \\
 & & {\Sigma^\infty\partial (X \times I)_+}. \uar \arrow[llu, bend left = 10,  "C_{\partial(X\times I)}"]
\end{tikzcd}
\end{equation*}
To verify this, recall the properties of Spanier--Whitehead duals from Section~\ref{sec:Atiyah}, and also apply~\cref{lem:SWcollapse}, where the notation $C_{\partial(X\times I)}$ was introduced for the map of spectra induced by the constant map $\partial(X \times I) \to \pt$. Applying $\Sigma$-$\Omega$ adjunction and discarding basepoints, we obtain a diagram of unbased spaces.
\begin{equation}\label{eq:SecondDefEtaT}
\begin{tikzcd}
\Omega^\infty \bS   & \mbox{\hspace{30ex}} & X \times I \arrow[ll, "\Ad\circ \SW\circ\Sigma^{-5}(\Th(\wh{\phi})\circ c_M)"'] \\
 & & {\partial (X \times I)}. \uar \arrow[llu, bend left = 10,  "1"]
\end{tikzcd}
\end{equation}
The adjoint to $C_{\partial (X \times I)}$
 is the constant map to the identity basepoint $1 \in \Omega^\infty \bS$, by~\cref{lem:SWcollapse}. Thus the horizontal map also lands in the path component $\mathrm{SG} \subseteq \Omega^\infty \bS$ of the identity map, and there is an induced map~$\tfrac{X \times I}{\partial} \to \mathrm{SG} \subseteq \G$.
The following is asserted in \cite[Section 2]{MR593058}, without proof. We provide one here.

\begin{lemma}\label{lem:newversion}
The tangential $\PL$ normal invariant $\eta_t(\phi, \wh{\phi})\in \left[\tfrac{X \times I}{\partial}, \G\right]$ is equal to the element
described by~\eqref{eq:SecondDefEtaT}.
\end{lemma}

\begin{proof}
There is a square
\[
\begin{tikzcd}
\bS^5 \ar[rr,"c_M"] \ar[d,"{c_{X \times I}}"]
&& \Th(\nu^\PL_M)/\Th(\nu^\PL_{\partial M}) \ar[d, "{\Th(\widehat{\phi})}"]
\\
\Th(\nu^\PL_{X \times I})/\Th(\nu^\PL_{\partial (X \times I)}) \ar[rr, "\Th(t_{\phi,\widehat{\phi}})"]
&& \Th(\nu^\PL_{X \times I})/\Th(\nu^\PL_{\partial (X \times I)})
\end{tikzcd}
\]
analogous to \eqref{eq:FibHEqSquare}, which commutes up to homotopy over $\Sigma \Th(\nu^\PL_{\partial (X \times I)})$.  The element \eqref{eq:SecondDefEtaT} is obtained by Spanier--Whitehead dualising the map $\Sigma^{-5} (\Th(\wh{\phi}) \circ c_M)$ over $\Sigma \Th(\nu^\PL_{\partial (X \times I)})$, hence obtaining a map under $\Sigma^\infty \partial(X \times I)_+$. By the square this map is homotopic under $\Sigma^\infty \partial(X \times I)_+$ to the result of Spanier--Whitehead dualising $\Sigma^{-5}(\Th(t_{\phi,\widehat{\phi}}) \circ c_{X \times I})$.  We must therefore show that the Spanier--Whitehead dual of $\Sigma^{-5}(\Th(t_{\phi,\widehat{\phi}}) \circ c_{X \times I})$ over $\Sigma \Th(\nu^\PL_{\partial (X \times I)})$ is $\Sigma$-$\Omega$-adjoint to the map $\eta_t(\phi, \wh{\phi})\colon X \times I \to \G \subset \Omega^\infty \bS$ under $1 \colon \partial(X \times I) \to \G \subseteq \Omega^\infty \bS$.

By Lemma~\ref{lem:SWcollapse}, the Spanier--Whitehead dual of $\Sigma^{-5}c_{X \times I}$ is the constant map $C_{X \times I} \colon \Sigma^\infty (X \times I)_+ \to \bS$, so to proceed we must understand the Spanier--Whitehead dual of $\Sigma^{-5}\Th(t_{\phi,\widehat{\phi}})$.
To do so, define the map $a \colon  \Sigma^\infty \G_+ \to \bS$
to be the adjoint of the inclusion $\G_+ \to \Omega^\infty\bS$.
For brevity, introduce the notation
\[
T := \Sigma^{-5} \tfrac{\Th(\nu^\PL_{X \times I})}{\Th(\nu^\PL_{\partial (X \times I)})} \qquad\text{and}\qquad T^\vee:=\Sigma^\infty (X \times I)_+,
\]
where $T^\vee$ indicates that these are Spanier--Whitehead dual.
\begin{claim}
The Spanier--Whitehead dual of
\[
\Sigma^{-5}\Th(t_{\phi, \wh{\phi}})  \colon   T\lra T,\qquad \text{over \,\,$\Sigma \Th(\nu^\PL_{\partial (X \times I)})$},
\]
is homotopic to the composition
\begin{equation}\label{eq:DualOfThT}
T^\vee \xrightarrow{\text{diag}} T^\vee \wedge T^\vee \xrightarrow{\Sigma^\infty\eta_t(\phi, \wh{\phi}) \wedge \Id} \Sigma^\infty \G_+ \wedge T^\vee \xrightarrow{a\wedge \Id} \bS\wedge T^\vee\simeq T^\vee,\qquad
\text{under \,\,$\Sigma^\infty \partial(X \times I)_+$}.
\end{equation}
\end{claim}

\begin{proof}[Proof of claim]
We first show that the map $\Th(t_{\phi, \wh{\phi}})$ may be decomposed as
\begin{equation}\label{eq:decomp}
\Sigma^5T\xrightarrow{\Delta} T^\vee\wedge \Sigma^5T\xrightarrow{\Sigma^\infty\eta_t(\phi, \wh{\phi}) \wedge \Id}\Sigma^\infty \G_+\wedge \Sigma^5T\xrightarrow{a\wedge \Id} \bS\wedge \Sigma^5T\simeq \Sigma^5T,
\end{equation}
where $\Delta$ denotes the Thom diagonal. To justify this, write $\nu\colon X\times I\to \BG(N)$ for the spherical fibration of a finite-dimensional representative of the $\PL$-normal bundle of $X\times I$. Write $\xi$ for the a finite-dimensional representative of the stable inverse to $\nu$, so that $\nu\ast\xi\simeq S(\varepsilon^M)$, the trivial spherical fibration for some $M$.
Now consider the self-equivalence
\[
t_{\phi,\widehat{\phi}}\ast \Id\ast\Id\colon \nu\ast\xi\ast\nu \to  \nu\ast\xi\ast\nu.
\]
We may identify the first two factors with $S(\varepsilon^M)$. Under this identification, the self-equivalence of spherical fibrations $t_{\phi,\widehat{\phi}}\ast \Id\colon \nu\ast\xi\to \nu\ast\xi$ is identified with the automorphism of the trivial spherical fibration $S(\varepsilon^M)=(X\times I)\times S^{M-1}$ described fibre-wise by the map $\eta_t(\phi,\widehat{\phi})\colon \tfrac{X\times I}{\partial}\to \G(M)$. Hence, the map $\Th(t_{\phi,\widehat{\phi}}\ast \Id\ast\Id)$ can be identified as the $M$-fold suspension of the composition~\eqref{eq:decomp}.
On the other hand, we can instead identifiy the latter two factors of $\nu\ast\xi\ast\nu$ with $S(\varepsilon^M)$, so that the map $t_{\phi,\widehat{\phi}}\ast \Id\ast\Id$ is just the $M$-fold stabilisation of $t_{\phi,\widehat{\phi}}$, giving $\Th(t_{\phi,\widehat{\phi}}\ast \Id\ast\Id)=\Sigma^M\Th(t_{\phi,\widehat{\phi}})$.
So the desired factorisation of $\Th(t_{\phi, \wh{\phi}})$ has been shown.

With this factorisation in hand, we can prove the claim. Write the coevaluation and evaluation maps for the pair $T$ and $T^\vee$ as
\[
\text{coev}  \colon  \bS \lra T^\vee \wedge T  \qquad\text{and}\qquad \text{ev}  \colon  T \wedge T^\vee \lra \bS.
\]
Consider the diagram
\begin{equation*}
\begin{tikzcd}
T^\vee \ar[r,"\simeq"] \ar[dd, "\text{diag}"]
&\bS\wedge T^\vee \ar[r,"{\text{coev} \wedge \Id}"]
& T^\vee \wedge T \wedge T^\vee \ar[d,"{\Id \wedge \Sigma^{-5}\Delta \wedge \Id}"]
& &&&
\\
 &&T^\vee \wedge T^\vee \wedge  T \wedge T^\vee \ar[rrrr,"{\Id \wedge \Sigma^\infty\eta_t(\phi, \wh{\phi}) \wedge \Id \wedge \Id}"] \ar[d, "{\Id \wedge \Id \wedge \text{ev}}"]
 &&&& T^\vee \wedge \Sigma^\infty \G_+ \wedge  T \wedge T^\vee \ar[d,"{\Id \wedge a\wedge\Id  \wedge \Id}"]
 \\
 T^\vee\wedge T^\vee\ar[rr,"\simeq"]& &T^\vee \wedge T^\vee\wedge\bS \ar[d,"{\Id\wedge\Sigma^\infty\eta_t(\phi, \wh{\phi}) \wedge  \Id}"]
  &&&& T^\vee \wedge \bS\wedge T \wedge T^\vee \ar[d,"{\Id \wedge \Id\wedge \text{ev}}"]
  \\
 & &  T^\vee\wedge\Sigma^\infty \G_+\wedge\bS \ar[rrrr,"{\Id\wedge a\wedge \Id} "]
  &&&& T^\vee.
\end{tikzcd}
\end{equation*}

By definition, the Spanier--Whitehead dual of the map $\Sigma^{-5}\Th(t_{\phi,\wh{\phi}})$ is the composition
\[
T^\vee\simeq \bS\wedge T^\vee\xrightarrow{\text{coev}\wedge\Id} T^\vee\wedge T\wedge T^\vee\xrightarrow{\Id\wedge \Sigma^{-5}\Th(t_{\phi,\wh{\phi}})\wedge\Id} T^\vee\wedge T \wedge T^\vee\xrightarrow{\Id\wedge \text{ev}} T^\vee\wedge \bS\simeq T^\vee.
\]
Using the factorisation~\eqref{eq:decomp}, this shows that the Spanier--Whitehead dual of the map $\Th(t_{\phi, \wh{\phi}})$ is the full clockwise composition in the diagram. Comparing the full anticlockwise composition in the diagram to the map~\eqref{eq:DualOfThT}, we have that the actual sequence of maps is a little different. However, as the initial map in the composition is the diagonal, the reader may check this makes the anti-clockwise composition equal to~\eqref{eq:DualOfThT}. Hence the claim is proved if we can show the diagram commutes up to homotopy under $\Sigma^\infty \partial(X \times I)_+$.

The right rectangle commutes, by inspection.
Now we show that the left rectangle commutes. The diagonal map $T^\vee\to T^\vee\wedge T^\vee$ is a co-ring structure on $T^\vee$ (true of any suspension spectrum), and the five-fold desuspension of the Thom diagonal $\Sigma^{-5}\Delta\colon T\to T^\vee\wedge T$ endows $T$ with the structure of co-module over $T^\vee$. This may be checked from the definition of the Thom diagonal (it corresponds to the well-known fact that the cohomology of the Thom space of a bundle is a module over the cohomology of the base space).  In particular, co-associativity $(\Id\wedge \Sigma^{-5}\Delta)\circ\Sigma^{-5}\Delta=(\text{diag}\wedge\Id)\circ\Sigma^{-5}\Delta$ holds. Recall that $\text{coev}$ is given by
\[
\bS \xrightarrow{\Sigma^{-5}c_{X \times I}}  T \xrightarrow{\Sigma^{-5}\Delta} T^\vee \wedge T,
\]
where the second map is the five-fold desuspension of the Thom diagonal. Hence, the co-associativity statement, above, means composition of the first two maps in the clockwise route around the left rectangle agrees with
\[
(\text{diag}\wedge\Id\wedge\Id)\circ (\text{coev}\wedge\Id)\colon \bS\wedge T^\vee \lra T^\vee \wedge T^\vee \wedge  T \wedge T^\vee.
\]
Now consider the full composition clockwise around the left rectangle as a map $\bS\wedge T^\vee\to T^\vee\wedge T^\vee\wedge \bS$. We may commute the latter two maps in the composition, as follows, because they operate on different factors:
\[
(\Id\wedge\Id\wedge\text{ev})\circ (\text{diag}\wedge\Id\wedge\Id)\circ(\text{coev}\wedge\Id)=(\text{diag}\wedge\Id)\circ(\Id\wedge\text{ev})\circ (\text{coev}\wedge\Id).
\]
But then $(\Id\wedge\text{ev})\circ (\text{coev}\wedge\Id)$ is the identity equivalence $\bS\wedge T^\vee\simeq T^\vee\wedge\bS$, so this equation is the statement that the left hand square commutes.
\end{proof}

We can now finish the proof of the lemma. By inspection, the following diagram
commutes up to homotopy under $\Sigma^\infty\partial(X \times I)_+$. In particular, recall that $C_{X\times I}\colon T^\vee\to\bS$ is simply projection to the basepoint, explaining commutativity of the left triangle.
 \[
 \begin{tikzcd}
 T^\vee
 \ar[rrrr, "{\Sigma^\infty\eta_t(\phi, \wh{\phi})}"]
 \ar[rd, "\simeq"]
 \ar[dd, "\text{diag}"']
&&&&
\Sigma^\infty\G_+
\ar[rrrr, "a"]
\ar[d, "\simeq"]
&&&&
\bS
\\
&
T^\vee\wedge\bS
\ar[rrr, "{\Sigma^\infty\eta_t(\phi, \wh{\phi})\wedge \Id}"]
&&&
\Sigma^\infty\G_+\wedge\bS
\ar[rr, "a\wedge \Id"]
&&
\bS\wedge\bS
\ar[rru, "\simeq"]
&&
\\
T^\vee\wedge T^\vee
\ar[rrrr, "{\Sigma^\infty\eta_t(\phi, \wh{\phi})\wedge \Id}"']
\ar[ru, "\Id\wedge C_{X\times I}"']
&&&&
\Sigma^\infty\G_+\wedge T^\vee
\ar[rr, "a\wedge \Id"']
\ar[u, "\Id\wedge C_{X\times I}"']
&&
\bS\wedge T^\vee
\ar[rr, "\simeq"']
\ar[u, "\Id\wedge C_{X\times I}"']
&&
T^\vee
\ar[uu, "C_{X\times I}"']
 \end{tikzcd}
 \]
  Under $\Sigma$-$\Omega$ adjunction, the composition along the top row is $\eta_t(\phi, \wh{\phi})\in \left[\tfrac{X \times I}{\partial}, \G\right]$. The anti-clockwise composition is~\eqref{eq:DualOfThT} postcomposed with the map $C_{X\times I}$. So by the claim, its Spanier--Whitehead dual is homotopic over $\Sigma \Th(\nu^\PL_{\partial (X \times I)})$ to $\Sigma^{-5}\Th(t_{\phi, \wh{\phi}})\circ \SW(C_{X \times I})$. Recalling that $\SW(C_{X \times I})=\Sigma^{-5}c_{X \times I}$ (Lemma~\ref{lem:SWcollapse}), concludes the proof of the lemma.
\end{proof}

\begin{remark}
  The structure of the proof of Lemma~\ref{lem:newversion} can be summarised as follows. We proved the following sequence of homotopies of maps of spectra $T^\vee \to \bS$ under $\Sigma^{\infty}\partial(X \times I)_+$:
  \begin{align*}
    \Ad^{-1} \eqref{eq:SecondDefEtaT} &\simeq \SW ( \Sigma^{-5}(\Th(\wh{\phi}) \circ c_M)) \simeq \SW (\Sigma^{-5} (\Th(t_{\phi,\wh{\phi}}) \circ c_{X \times I})) \\
     &\simeq \SW  \big( (\Sigma^{-5} \Th(t_{\phi,\wh{\phi}})) \circ (\Sigma^{-5} c_{X \times I})\big)
     \simeq \SW(\Sigma^{-5} c_{X \times I}) \circ \SW (\Sigma^{-5} \Th(t_{\phi,\wh{\phi}})) \\  &\simeq
     C_{X \times I} \circ \eqref{eq:DualOfThT}
     \simeq a \circ \Sigma^{\infty}\eta_t(\phi,\wh{\phi}) \simeq \Ad^{-1} \eta_t(\phi,\wh{\phi}).
  \end{align*}
  Taking adjoints and forgetting basepoints yields the statement of Lemma~\ref{lem:newversion}.
\end{remark}

\begin{proposition}
The left-hand square of diagram~\eqref{eq:MTWdiagram} commutes.
\end{proposition}

\begin{proof}

Let $\phi \colon X \times I \to X \times I$ be a homotopy equivalence which is the identity on the boundary, representing an element of $\pi_1(\hAut(X))$. As in Construction \ref{construction:eta}, we associate to $\phi$ the induced map $\psi \colon X \times S^1 \to X$ which is the identity on $X \times \{1\}$.
Using the framing of $\nu_X^\PL$ we can cover $\phi$ by the $\PL$-microbundle map $\wh{\phi} \colon  \nu_{X \times I}^\PL \to \nu_{X \times I}^\PL$ induced by the framing, giving an element~$(\phi, \wh{\phi}) \in \mathcal{S}^t_\partial(X \times I)$. By~\cref{lem:newversion}, $\eta_t(\phi, \wh{\phi})$ is described using the diagram
\begin{equation*}
\begin{tikzcd}
\bS^5 \rar{c_{X \times I}} \arrow[rrd, swap, bend right = 10,  "\Sigma c_{\partial(X \times I)}"] & \Th(\nu^\PL_{X \times I})/\Th(\nu^\PL_{\partial (X \times I)}) \rar{\Th(\wh{\phi})} & \Th(\nu^\PL_{X \times I})/\Th(\nu^\PL_{\partial (X \times I)}) \dar{\partial} \\
 & & \Sigma \Th(\nu^\PL_{\partial (X \times I)}),
\end{tikzcd}
\end{equation*}
by
applying the functor $\SW\circ\Sigma^{-5}$
and then $\Sigma$-$\Omega$ adjunction. The basepoint-translated element $(T_{1,0})_*(\eta_t(\phi, \wh{\phi})) \in [\tfrac{X \times I}{\partial}, \Omega^\infty \bS]$ is then described by the same recipe applied to the analogous diagram
\begin{equation*}
\begin{tikzcd}
\bS^5 \rar{c_{X \times I}} \arrow[rrrd, swap, bend right = 10,  "0"] & \Th(\nu^\PL_{X \times I})/\Th(\nu^\PL_{\partial (X \times I)}) \ar[rr, "{\Th(\wh{\phi}) - \Id}"] && \Th(\nu^\PL_{X \times I})/\Th(\nu^\PL_{\partial (X \times I)}) \dar{\partial} \\
 & && \Sigma \Th(\nu^\PL_{\partial (X \times I)}),
\end{tikzcd}
\end{equation*}
 Using the framing to trivialise these Thom spectra, we may write this diagram as
\begin{equation*}
\begin{tikzcd}
\bS^5 \ar[rr,"{[X \times I]_\bS}"] \arrow[rrrd, swap, bend right = 10,  "0"] && \Sigma^\infty \tfrac{X \times I}{\partial} \rar{\Sigma^\infty \phi - \Id} & \Sigma^\infty \tfrac{X \times I}{\partial} \dar{\partial} \\
 & && \Sigma^{\infty+1} \partial(X \times I)_+.
\end{tikzcd}
\end{equation*}

There is a factorisation
\[
\Sigma^\infty \phi - \Id \colon \Sigma^\infty \tfrac{X \times I}{\partial}\simeq \Sigma^\infty\tfrac{(X\times S^1)_+}{(X\times\{1\})_+} \xrightarrow{[\Sigma^\infty\psi - \pr_X]} \Sigma^\infty X_+ \simeq \Sigma^\infty (X \times I)_+ \overset{q}\lra \Sigma^\infty \tfrac{X \times I}{\partial},
\]
where the first map is described in \eqref{eq:PhiMinusId} and the second is the quotient.
The map $\partial$ in the diagram fits into the cofibre sequence
$\Sigma^\infty (X \times I)_+ \xrightarrow{q} \Sigma^\infty \tfrac{X \times I}{\partial} \xrightarrow{\partial}
  \Sigma^{\infty+1}\partial(X \times I)_+$, where  this cofibre sequence is obtained from taking the sequence $\partial(X \times I) \to X \times I \to \tfrac{X \times I}{\partial}$, applying $\Sigma^{\infty}(-)_+$ and extending to the right by one.
It follows that the map described by the diagram corresponds, under lifting along $q$, to the composition
\begin{equation}
\label{eq:composition}
\begin{tikzcd}
\bS^5\ar[rr, "{[X\times I]_{\bS}}"]
&&
 \Sigma^\infty \tfrac{X \times I}{\partial} \ar[rr, "{[\Sigma^\infty\psi-\pr_X]}"]
 &&
 \Sigma^\infty X_+.
\end{tikzcd}
\end{equation}
Under the identification $\Sigma^\infty\tfrac{X\times I}{\partial}\simeq \Sigma^{\infty+1}X_+$, the fundamental class $[X\times I]_{\bS}$ corresponds to $\Sigma[X]_\bS$. Using this, the composition~\eqref{eq:composition} corresponds to the element
$[\Sigma^\infty\psi-\pr_X]_*(\Sigma[X]_\bS)\in [\bS^5, \Sigma^\infty X_+]\cong [\bS^5, \Th(\nu^{\PL}_X)]$
where the isomorphism comes from another use of the framing. Applying the functor $\SW\circ\Sigma^{-5}$ to the composition, we obtain the definition of $\wh{\eta}(\phi)$.
 \end{proof}

\subsubsection{Proof that \eqref{eq:etaUmkehr} commutes}

The equivalence $\Omega \G \simeq \Omega^{\infty+1}\bS^0$, induced by basepoint translation (\cref{rem:basepoint}), is \emph{not an equivalence of infinite loop spaces} (for example, their deloopings have different Pontryagin rings~\cite[\textsection3.F]{MadsenMilgram}). This difference between the infinite loop space structures will account for a difficulty in the following proposition. The diagram in the following proposition is obtained from \eqref{eq:etaUmkehr} by applying the alternative factorisation of $\eta$ that we obtained by proving that diagram \eqref{eq:MTWdiagram} commutes.  Hence Proposition~\ref{prop:0map} completes the proof that \eqref{eq:etaUmkehr} commutes.

\begin{proposition}\label{prop:0map}
There is a commutative diagram of groups and homomorphisms
\begin{equation*}
\begin{tikzcd}
\eta \colon \pi_1(\hAut (X)) \rar{\widehat{\eta}} \dar{\mft_*} & \left[X_+, \Omega^{\infty+1} \bS^0 \right] \arrow[r, "\cong","(T_{0,1})_*"'] \dar{\mft_!} & \left[X_+, \Omega\G \right] \rar\dar{\mft_!} & {[X_+, \Omega(\G/\PL)]} \dar{\mft_!}\\
\eta \colon \pi_1(\hAut (A)) \rar{\widehat{\eta}}  & \left[A_+, \Omega^{\infty+1} \bS^0 \right] \arrow[r, "\cong","(T_{0,1})_*"'] & \left[A_+, \Omega \G \right] \rar & {[A_+, \Omega(\G/\PL)],}
\end{tikzcd}
\end{equation*}
where each $\mft_!$ denotes the umkehr map with respect to the indicated infinite loop space structure on the codomain.
\end{proposition}

\begin{proof}
Apart from the left column, the group structures arise from loop space structures indicated by (the first) $\Omega$. These are compatible, so the maps are homomorphisms.
The commutativity of the right-hand square is by naturality of the umkehr map with respect to maps of generalised cohomology theories, i.e.\ maps of infinite loop spaces. The commutativity of the middle square is therefore surprising, as we have pointed out that $\Omega \G \simeq \Omega^{\infty+1}\bS^0$ is not an equivalence of infinite loop spaces. We will address this below.

Next we prove the commutativity of the left-hand square.
Let $\phi$ be a loop in $\hAut (X)$  based at the identity map. As in Construction \ref{construction:eta}, we associate to $\phi$ the induced map $\psi \colon X \times S^1 \to X$ which is the identity on $X \times \{1\}$.  By functoriality of Postnikov truncation, $\psi$ yields $\mft_*(\psi) \in \pi_1(\hAut (A))$, and the diagram
\begin{equation*}
\begin{tikzcd}
\bS^5 \arrow[rr, "{\Sigma [X]_\bS}"] \arrow[drr, swap, "{\Sigma [A]_\bS}"]  && \Sigma^{\infty+1} X_+ \dar{\Sigma \mft_+} \ar[rrr, "{[\Sigma^\infty\psi - \mathrm{pr}_X]}"] &&& \Sigma^\infty X_+ \dar{\mft_+}\\
 && \Sigma^{\infty+1} A_+ \ar[rrr, "{[\Sigma^\infty\mft_*(\psi)-\mathrm{pr}_A]}"] &&& \Sigma^\infty A_+
\end{tikzcd}
\end{equation*}
commutes up to homotopy, using that $\mft_*({[X]_\bS}) = [A]_\bS$ by Lemma \ref{lem:orientations}. The clockwise composition in $[\bS^5, \Sigma^\infty A_+] =[\bS^5, \Th(\nu^{\PL}_A)]\cong_{\SW} [A_+, \Omega^{\infty+1}\bS^0]$ is by definition $\mft_! \widehat{\eta}(\phi)$; the anticlockwise composition is $\widehat{\eta}(\mft(\phi))$.

To prove commutativity of the middle square, let us write $\mft_!^\bS$ for the left-hand umkehr map and~$\mft_!^\G$ for the right-hand one. As $X$ is 4-dimensional we may replace $\Omega^{\infty+1} \bS^0$ by its 4-truncation, which we choose to write as $\Omega \tau_{\leq 5}\Omega^{\infty} \bS^0 = \Omega \Omega^{\infty} \tau_{\leq 5}\bS^0$. We then use that $\pi_4^s(\bS^0) = 0 = \pi_5^s(\bS^0)$ to see that $\tau_{\leq 5}\bS^0 = \tau_{\leq 3}\bS^0$, and hence we have that $\Omega \tau_{\leq 5}\Omega^{\infty} \bS^0 = \Omega \tau_{\leq 3}\Omega^{\infty} \bS^0$. Now we consider the equivalence of fibre sequences of connected spaces
\begin{equation*}
\begin{tikzcd}
\tau_{[2,3]}\Omega^{\infty}_0 \bS^0 \rar \arrow[d, "\simeq"] &  \tau_{\leq 3}\Omega^{\infty}_0 \bS^0 \rar{\rho} \arrow[d, "\simeq"]&  \tau_{\leq 1}\Omega^{\infty}_0 \bS^0 \arrow[d, "\simeq"]\\
\tau_{[2,3]}\mathrm{SG} \rar &  \tau_{\leq 3}\mathrm{SG} \rar&  \tau_{\leq 1}\mathrm{SG}
\end{tikzcd}
\end{equation*}
and make the following observation:~the base and fibre each admit a unique structure of an infinite loop space. For the base this is clear as it is a $K(\Z/2,1)$; for the fibre this follows from Mathew--Stojanoska~\cite[Theorem 5.1.2]{MathewStojanoska}, though it can also be proven by hand.
The citation applied with $n=2$ says that the functor $\Omega^{\infty}(-) \colon \mathcal{S}\mathrm{pectra}_{[2,3]} \to \mathcal{S}_*$, from the $\infty$-category of spectra with nontrivial homotopy groups only in degrees two and three to the $\infty$-category of pointed spaces, is fully faithful, i.e.\ homotopy equivalence on morphism spaces. If we have two a priori different infinite loop space structures on $\tau_{[2,3]}\mathrm{SG}$, they corresponds to two spectra in the domain of  $\Omega^{\infty}(-)$ that are homotopy equivalent in the codomain. The fact that $\Omega^{\infty}(-)$ is fully faithful leads to an equivalence of spectra, and hence of infinite loop space structures on $\tau_{[2,3]}\mathrm{SG}$.

Thus the maps on base and fibre are necessarily equivalences of infinite loop spaces. Looping the top fibre sequence and mapping in $X$ or $A$, the difference of the umkehr maps $\mft_!^\bS$ and $\mft_!^\G$ fit in to a map of short exact sequences
\begin{equation}\label{diagram:difference-of-t-bangs}
\begin{tikzcd}
0 \rar & {[X_+, \Omega \tau_{[2,3]}\Omega^{\infty}_0 \bS^0]} \rar \dar{\mft_!^\bS - \mft_!^\G=0} &  {[X_+, \Omega\tau_{\leq 3}\Omega^{\infty}_0 \bS^0]} \rar{\Omega \rho_*} \dar{\mft_!^\bS - \mft_!^\G}&  {[X_+, \Omega\tau_{\leq 1}\Omega^{\infty}_0 \bS^0]} \dar{\mft_!^\bS - \mft_!^\G = 0} \rar & 0\\
0 \rar &{[A_+, \Omega \tau_{[2,3]}\Omega^{\infty}_0 \bS^0]} \rar  &  {[A_+, \Omega\tau_{\leq 3}\Omega^{\infty}_0 \bS^0]} \rar &  {[A_+, \Omega\tau_{\leq 1}\Omega^{\infty}_0 \bS^0]} \rar & 0,
\end{tikzcd}
\end{equation}
where the two outer vertical maps are zero because here the two infinite loop space structures used to form the two umkehr maps agree, although we will only make use of the vanishing of the left-hand vertical map.

First we show that the right hand groups are both $\Z/2$. For this we compute the homotopy groups of the codomain $\Omega\tau_{\leq 1}\Omega^{\infty}_0 \bS^0$. The space $\tau_{\leq 1}\Omega^{\infty}_0 \bS^0$ has non-vanishing homotopy groups only in degrees 0 and 1, where they are the stable $0$-stem $\Z$ and the stable $1$-stem $\Z/2$. Taking the loop space shifts the homotopy groups down, so that only a $\Z/2$ in degree $0$ remains. We see that~$\Omega\tau_{\leq 1}\Omega^{\infty}_0 \bS^0$ has precisely two path components, each of which is weakly contractible.  It follows that $[X_+, \Omega\tau_{\leq 1}\Omega^{\infty}_0 \bS^0] \cong \Z/2 \cong  [A_+, \Omega\tau_{\leq 1}\Omega^{\infty}_0 \bS^0]$, as claimed.

We show that the middle vertical map $\mft_!^\bS - \mft_!^\G$ factors through the map $\Omega \rho_*$ to the quotient, yielding
\[(\mft_!^\bS - \mft_!^\G) \colon [X_+, \Omega\tau_{\leq 3}\Omega^{\infty}_0 \bS^0] \xrightarrow{\Omega \rho_*} [X_+, \Omega\tau_{\leq 1}\Omega^{\infty}_0 \bS^0] \xrightarrow{T} [A_+, \Omega\tau_{\leq 3}\Omega^{\infty}_0 \bS^0].\]
To see this, let $x,x' \in [X_+, \Omega\tau_{\leq 3}\Omega^{\infty}_0 \bS^0]$ be such that $\Omega \rho_*(x) = \Omega \rho_*(x')$.
Then $\Omega\rho_*(x-x') =0$, so~$x-x'$ lies in the image of $[X_+, \Omega \tau_{[2,3]}\Omega^{\infty}_0 \bS^0]$. By commutativity of the left square, and the fact that the left vertical map vanishes, we deduce that  $0= (\mft_!^\bS - \mft_!^\G)(x - x') = (\mft_!^\bS - \mft_!^\G)(x) - (\mft_!^\bS - \mft_!^\G)(x')$. Thus $(\mft_!^\bS - \mft_!^\G)(x) = (\mft_!^\bS - \mft_!^\G)(x')$ and the central $\mft_!^\bS - \mft_!^\G$ factors as $T \circ (\Omega \rho_*)$, as asserted.

Next we argue that the composition
\[[A_+, \Omega\tau_{\leq 3}\Omega^{\infty}_0 \bS^0] \overset{\mft^*}\lra [X_+, \Omega\tau_{\leq 3}\Omega^{\infty}_0 \bS^0] \lra [X_+, \Omega\tau_{\leq 1}\Omega^{\infty}_0 \bS^0]\]
is surjective. To see this recall that the last group here is $\Z/2$, and note that by a similar calculation~$\Omega\tau_{\leq 3}\Omega^{\infty}_0 \bS^0$ has precisely two path components. We can map all of $A$ to either of these components to obtain elements in the domain of the composition that map to each of the elements in the codomain $\Z/2$.

We have developed the following diagram.
\[
\begin{tikzcd}
  {[A_+, \Omega\tau_{\leq 3}\Omega^{\infty}_0 \bS^0]} \ar[r,twoheadrightarrow] \ar[d,"\mft^*"] & {[X_+, \Omega\tau_{\leq 1}\Omega^{\infty}_0 \bS^0]} \ar[d,"T"] \\
{[X_+, \Omega\tau_{\leq 3}\Omega^{\infty}_0 \bS^0]} \ar[r,"\mft_!^\bS - \mft_!^\G"'] \ar[ur,"\Omega \rho_*"] &  {[A_+, \Omega\tau_{\leq 3}\Omega^{\infty}_0 \bS^0]}.
\end{tikzcd}
\]
The identity $\mft_! \circ \mft^* = \mathrm{Id}$ from Lemma \ref{lem:UmkehrSplit} holds for both meanings of $\mft_!$, so that $(\mft_!^\bS - \mft_!^\G) \circ \mft^*=0$. It now follows by a straightforward diagram chase that $\mft_!^\bS - \mft_!^\G=0$. Here are the details.
Let $x \in [X_+, \Omega\tau_{\leq 3}\Omega^{\infty}_0 \bS^0]$. Since the top horizontal map is surjective, there exists $y \in [A_+, \Omega\tau_{\leq 3}\Omega^{\infty}_0 \bS^0]$ with $\Omega \rho_* \circ \mft^*(y) = \Omega\rho_*(x)$. Then
\[(\mft_!^\bS - \mft_!^\G)(x) = T \circ (\Omega\rho_*)(x) = T \circ (\Omega\rho_*) \circ \mft^*(y) = (\mft_!^\bS - \mft_!^\G)\circ \mft^*(y) =0.\]
 This shows that the middle square of \eqref{diagram:difference-of-t-bangs} commutes, completing the proof of the proposition.
\end{proof}

As explained beforehand, the proposition implies that the square \eqref{eq:etaUmkehr} commutes. This completes the proof that \eqref{diagram:master-diagram} commutes, and hence completes the proof of Theorem~\ref{thm:main-example}.

\def\MR#1{}
\bibliography{smoothing_pseudo}

\end{document}